\documentclass{amsart}

\usepackage{accents}
\usepackage[export]{adjustbox}
\usepackage{amsmath,amssymb,amsthm}
\usepackage[english]{babel}
\usepackage{chngcntr}
\usepackage{graphicx}
\usepackage{hyperref}
\usepackage{indentfirst}
\usepackage[utf8]{inputenc}
\usepackage{mathtools}
\usepackage{mathrsfs}
\usepackage[square,sort,comma,numbers]{natbib}
\usepackage{stmaryrd}
\usepackage{subcaption}
\usepackage{tikz-cd}
\usepackage{url}
\usepackage{xcolor}

\newtheorem{theorem}{Theorem}[section]
\newtheorem{conjecture}[theorem]{Conjecture}
\newtheorem{corollary}[theorem]{Corollary}
\newtheorem{example}[theorem]{Example}
\newtheorem{lemma}[theorem]{Lemma}
\newtheorem{proposition}[theorem]{Proposition}
\newtheorem{remark}[theorem]{Remark}

\theoremstyle{definition}
\newtheorem{definition}[theorem]{Definition}

\numberwithin{equation}{section}
\counterwithin{figure}{section}

\DeclareMathOperator{\C}{\mathbb{C}}
\DeclareMathOperator{\ch}{ch}
\DeclareMathOperator{\codim}{codim}

\DeclareMathOperator{\conv}{Conv}

\DeclareMathOperator{\graph}{graph}

\DeclareMathOperator{\Hom}{Hom}
\DeclareMathOperator{\id}{id}

\DeclareMathOperator{\Log}{Log}
\DeclareMathOperator{\MV}{MV}
\DeclareMathOperator{\Newton}{Newton}

\DeclareMathOperator{\Pic}{Pic}

\DeclareMathOperator{\R}{\mathbb{R}}

\DeclareMathOperator{\Residue}{Res}

\DeclareMathOperator{\todd}{Todd}

\DeclareMathOperator{\vol}{Vol}
\DeclareMathOperator{\Z}{\mathbb{Z}}

\newcommand{\add}{\text{add}}
\newcommand{\GS}{\text{GS}}
\newcommand{\poly}{\text{poly}}
\newcommand{\real}{\text{real}}
\newcommand{\simple}{\text{simple}}
\newcommand{\total}{\text{total}}
\newcommand{\trop}{\text{trop}}

\title{The Gamma Conjecture for Tropical Curves in Local Mirror Symmetry}
\author{Junxiao Wang}
\address{\tiny{Beijing International Center for Mathematical Research, No.5 Yiheyuan Road, Haidian, Beijing 100871, China}}
\email{wangjunxiao@bicmr.pku.edu.cn}

\begin{document}

\maketitle

\begin{abstract}
    In \cite{Hosono_2004}, Hosono conjectured the equality between the central charges of A and B sides in local mirror symmetry. In this paper, following the idea of the tropical approach to the central charges as in \cite{A-G-I-S_2018}, we relate a coherent sheaf supported on a holomorphic curve with its mirror Langrangian submanifold in local mirror symmetry through a tropical curve by interpreting their central charges using the combinatorial information of the tropical curve. This proves Hosono's conjecture in this specific case. Furthermore, we put this description in the Gross-Siebert model of local mirror symmetry as in \citep{Gross-Siebert_2014} and confirm the result in \cite{Ruddat-Siebert_2019} that the parameters in the Gross-Siebert model are the canonical coordinates in mirror symmetry.
\end{abstract}

\section{Introduction}
\label{section1}

The Gamma conjecture as stated in \cite{G-G-I_2016} studies the relation between the quantum differential equation of a Fano manifold and its Gamma class. It can be understood as the equality between the central charges of mirror objects as in \cite{Iritani_2009}, where the central charge of a coherent sheaf is the integral of the solution to the quantum differential equation paired with the Gamma class and the central charge of its mirror Lagrangian submanifold is the oscillatory integral on it. For the case of toric Fano varieties, the equality between central charges was proven in \cite{Fang_2018}, which leads to a proof of the Gamma conjecture for toric Fano varieties \cite{Fang-Zhou_2019}.

In this paper, we study the Gamma conjecture of local mirror symmetry, which is understood as the equality between the central charges of mirror objects in local mirror symmetry as proposed in \cite{Hosono_2004}:
\begin{conjecture}[The Gamma Conjecture of Local Mirror Symmetry \cite{Hosono_2004}] Suppose $X$ is a noncompact toric Calabi-Yau variety and let $Y$ be its Hori-Vafa mirror, then for a pair of mirror objects $E\in D^b Coh(X)\leftrightarrow L\in Fuk(Y)$, the central charge $Z_t(E)$ of $E$ is equal to the central charge $C_t(L)$ of $L$, where the central charges $Z_t(E)$ and $C_t(L)$ are defined in Definition \ref{def_centralcharge_period} and Definition \ref{def_centralcharge_chernclass}.
\label{conj_Gamma_conj_LMS}
\end{conjecture}

Following the philosophy of SYZ conjecture\cite{S-Y-Z_1996}, we regard $X$ and $Y$ as dual torus fibrations on a base space $B$ and relate the pair of mirror objects $E\leftrightarrow L$ with a tropical object $\beta_{\trop}$ on $B$. The central charges can then be interpreted using the combinatorial information of $\beta_{\trop}$. 

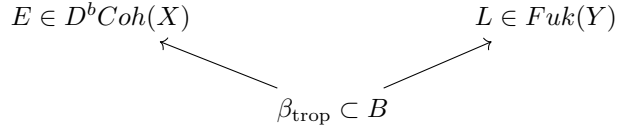
\begin{figure}[htbp!]
\label{fig_tropical_to_mirror_objects}
    \centering
    \begin{tikzcd}
    E\in D^bCoh(X) &                                                    & L\in Fuk(Y) \\
                                          & \beta_{\text{trop}}\subset B \arrow[lu] \arrow[ru] &                                   
    \end{tikzcd}
    \caption{Tropical object lifted to mirror objects}
\end{figure}

We focus on the case where $X=K_{X_\Sigma}$, the canonical bundle of an $n-$dimensional smooth projective toric Fano variety $X_{\Sigma}$, then the base space is $B=\R^n\times\R_{\geq0}$. Now suppose $E_C$ is a coherent sheaf supported on a holomorphic curve $C$ of $X_{\Sigma}$, we construct a tropical curve $\beta_{\trop,C}$ in $B$ with ends on the discriminant locus, and interpret the central charge $Z_t(E_C)$ using the combinatorial information of $\beta_{\trop,C}$.

\begin{theorem}
\label{thm_maintheorem_1}
Suppose $C$ is a holomorphic curve in $X_\Sigma$ which is the complete intersection of a \textit{nice family} of divisors $\{Q_{k}\}, k=1,2,\dots,n-1$ and let $E_C=i_*\mathcal{O}_C(\sum_{k=1}^{n-1}\tilde{Q}_k)$ be the coherent sheaf on $X$ supported on $C$, where $\tilde{Q}_k$ is the pullback divisor in $X$. Then there is a tropical curve $\beta_{\trop,C}$ in $\R^n$ with ends on $\partial G_\trop$, such that 
\begin{equation}
\label{eq_maintheorem_1}
    Z_t(E_C)
    =\sum\limits_{s=1}^{p-n}(2\pi\sqrt{-1})^{n}E_{n+s}(-\log \check{t}_s(t))+\frac{1}{2}(2\pi\sqrt{-1})^{n+1}(\sum\limits_{j=1}^{p}E_j+V),
\end{equation}
where $V$ is the sum of the volumes of the mixed cells dual to the vertices of $\beta_{\trop,C}$, $E_j$ is the sum of the weights of the edges of $\beta_{\trop,C}$ which have ends on the facet of $\partial G_\trop$ dual to $v_j$ and $\check{t}_s(t)$ is the mirror map as in (\ref{eq_mir_map}). Here $G_\trop$ is the unique compact component of the complement of the skeleton of the discriminant locus as defined in (\ref{eq_G_trop}) and $\{v_j\}, j=1,2,\dots,p$ are the generators of the one-cones of the fan $\Sigma$.
\end{theorem}

The construction of $\beta_{\trop,C}$ follows from the tropicalization of a holomorphic curve as in \cite{Mikhalkin_2004}. It is dual to a mixed subdivision $T_{\phi}$ of a polytope $G$ which has the same shape as the moment polytope of $X_\Sigma$ with respect to its anticanonical divisor. The right hand side of (\ref{thm_maintheorem_1}) can then be interpreted using the volumes of the mixed cells of $T_\phi$. On the other hand, the central charge $Z_t(E_C)$ can be represented by the intersection number of the divisors, which are also related to the volumes of the mixed cells according to Bernstein's theorem. So both sides of (\ref{eq_maintheorem_1}) are related to the volumes of the mixed cells, which leads to a proof of Theorem \ref{thm_maintheorem_1}.

We then construct a Lagrangian submanifold $L_{\trop,C}^t$ in the Hori-Vafa mirror $Y^t$ from the tropical curve $\beta_{\trop,C}$, whose central charge $C_t(L_{\trop,C}^t)$ can also be interpreted using the combinatorial information of $\beta_{\trop,C}$ and is equal to $Z_t(E_C)$.

\begin{theorem}
\label{thm_maintheorem_2}
For $(t_1,t_2,\dots,t_{p-n})=(t^{\lambda_{n+1}},t^{\lambda_{n+2}},\dots,t^{\lambda_{p}})$ with $\sum_{s=1}^{p-n}\lambda_{n+s}D_{n+s}$ an ample divisor and $t$ a sufficiently small positive real number, there exists a piecewise Lagrangian closed submanifold $L_{\beta_{\trop,C}}^t$ in $Y^t$, such that
\begin{equation}
\label{eq_maintheorem_2}
    C_t(L_{\beta_{\trop,C}}^t)=Z_t(E_C),
\end{equation}
where $D_{n+s}$ is the toric divisor of $X_\Sigma$ dual to $v_{n+s}$.
\end{theorem}

The piecewise Lagrangian submanifold $L_{\beta_{\trop,C}}^t$ is constructed along the moment map $\mu_0$ of $\C^2\times(\C^*)^n$ restricted to $Y^t$. It uses the cycles in the fiber of $\Log_t$ induced from the mixed cells as in Definition \ref{def_cycle_in_torus}. For a point of an edge $e$ of $\beta_{\trop,C}$, we take the cycle in its fiber which is induced by the mixed cell dual to $e$, then the union of them is an open Lagrangian submanifold $\Gamma_e$ in $Y^t$. For a vertex $v$ of $\beta_{\trop,C}$, we take a cycle $\Gamma_v$ in its fiber which is induced by the mixed cell dual to $v$, which glues the $\Gamma_e$'s for the edges incident with $v$. For an end $a$ of $\beta_{\trop,C}$, we construct an open Lagrangian submanifold $\Gamma_a$ by deforming the part of $Y^t$ near $a$ to a simple model, which caps $\Gamma_e$ for the edge $e$ with end $a$. Then $L_{\beta_{\trop,C}}^t$ is the union of $\Gamma_e$'s, $\Gamma_v$'s and $\Gamma_a$'s.

\begin{remark}
\begin{itemize}
    \item Our construction away from the ends of the tropical curve matches the constructions in \cite{Matessi_2018} and \cite{Mikhalkin_2019}, while we throw away the smoothness for the purpose of the period computation. It is expected that our piecewise Lagrangian submanifolds could be deformed to Lagrangian submanifolds within the same homology class. 
    \item The construction of the Lagrangian submanifold near the ends also appears in \cite{Mak-Ruddat_2020} and \cite{Hicks_2021}. In \cite{Mak-Ruddat_2020}, they focus on how to make sure the lifting near the end is Lagrangian since the symplectic form they use does not match the symplectic form induced from the fibration, while we do not have this issue. In \cite{Hicks_2021}, their model near the end is already simple since they care about the 2-dimensional case, and their construction near the end matches our construction in the simple model, while we have an additional step to deform it to the simple model.
\end{itemize}
\end{remark}

The central charge $C_t(L_{\beta_{\trop,C}}^t)$ is then the sum of the integrals of $\Omega$ on $\Gamma_e$'s, $\Gamma_v$'s and $\Gamma_a$'s. While the computation of the integrals on $\Gamma_e$'s and $\Gamma_v$'s is canonical, the evaluation of the integral on $\Gamma_a$'s requires some additional work. We add a cycle $\Gamma_{\add}$ to $\Gamma_a$ and isotope $\Gamma_a\cup\Gamma_{\add}$ to a singular fiber within $Y^t$. The integral on $\Gamma_a$ is then the difference between the integrals on $\Gamma_\add$ and the singular fiber, where the evaluation of the integral on $\Gamma_{\add}$ uses the method of the period evaluation as in \cite{Ruddat-Siebert_2019}.

According to Conjecture \ref{conj_Gamma_conj_LMS}, the equality between $C_t(L_{\beta_{\trop,C}}^t)$ and $Z_t(E_C)$ gives
\begin{conjecture}
\label{conj_mirror_objects}
The coherent sheaf $E_C$ and the Lagrangian submanifold $L_{\beta_{\trop,C}}^t$ are mirror to each other under homological mirror symmetry.
\end{conjecture}

\begin{remark}
\label{remark_tropical_approach_AGIS}
\begin{itemize}
    \item In \cite{Gross-Matessi_2018}, they conjectured the mirror between the line bundles on a toric Calabi-Yau threefold and the Lagrangian sections of its Hori-Vafa mirror. It contains a part of our case when $X_\Sigma$ is a smooth Fano 2-dimensional toric variety. The Lagrangian sections can be regarded as the Lagrangian submanifolds which are lifted from the tropical object which is the whole base space.
    \item The Lagrangian sections and line bundles are also mirror to each other in Batyrev mirrors and their central charges are conjectured to be equal. In \cite{A-G-I-S_2018}, they proposed a tropical approach to the central charge of the Lagrangian section. The tropical object is the whole base space and they evaluate the leading term of the central charge by measuring the difference between the volumes of the tropical amoeba and the actual amoeba. In our case, we evaluate all the central charge, with the subleading terms coming from a delicate evaluation of the lifting $\Gamma_a$ near the end $a$.
    \item In \cite{Hicks_2021}, they constructed a Lagrangian submanifold in $\C\mathbb{P}^2\backslash E$, which is lifted from a tropical curve, and showed that it is mirror to a coherent sheaf supported on an elliptic fiber of $\check{X}_{9111}$, the mirror of $\C\mathbb{P}^2\backslash E$. This gives evidence that the construction could be applied to general Gross-Siebert program.
\end{itemize}
\end{remark}

The same story can be put into the setting of the Gross-Siebert model of local mirror symmetry as described in \cite{Gross-Siebert_2014}. They put $X$ in a toric degeneration $\mathcal{X}$, then its dual toric degeneration $\mathcal{Y}$ contains a family of open sets $Y_{\GS}^{\check{t}}$ which is a modification of the Hori-Vafa mirror $Y^t$ with instanton correction inserted \cite{C-L-L_2012}. We can use the same method to construct a Lagrangian submanifold $L_{\trop,C}^{\check{t}}\subset Y_{\GS}^{\check{t}}$, whose central charge $C_{\check{t}}(L_{\beta_{\trop,C}}^{\check{t}})$ can also be interpreted using the combinatorial information of $\beta_{\trop,C}$.

\begin{theorem}
\label{thm_mainthm_GS}
For suitably chosen $\check{t}$ as in Theorem \ref{thm_maintheorem_2}, there exists a piecewise Lagrangian closed submanifold $L_{\beta_{\trop,C}}^{\check{t}}$ in $Y_\GS^{\check{t}}$, such that
\begin{equation}
\label{eq_maintheorem_GS}
    C_{\check{t}}(L_{\beta_{\trop,C}}^{\check{t}})=\sum_{s=1}^{p-n}(2\pi\sqrt{-1})^{n}E_{n+s}(-\log \check{t}_{s})+\frac{1}{2}(2\pi\sqrt{-1})^{n+1}(\sum_{j=1}^{p}E_j+V),
\end{equation}
where $E_j$ and $V$ are the same as in Theorem \ref{thm_maintheorem_1}.
\end{theorem}

It turns out that the parameters $t$ for $Y^t$ and $\check{t}$ for $Y_{GS}^{\check{t}}$ are the complex and K\"{a}hler parameters in local mirror symmetry, and the two models of local mirror symmetry are equivalent under the mirror map $\check{t}(t)$.

\begin{theorem}
\label{thm_central_charge_match_mirror_map}
There is a diffeomorphism 
$$
    \Psi_{GS\rightarrow HV}:Y^{\check{t}}_{GS}\rightarrow Y^t
$$
under the mirror map $\check{t}(t)$ when $|t|$ is sufficiently small, such that
\begin{equation*}
    C_t(L_{\beta_{\trop,C}}^t)(t)
    =\int_{\Psi_{GS\rightarrow HV}^{-1}(L_{\beta_{\trop,C}}^t)}d\log\check{u}d\log\check{z}_1\dots d\log\check{z}_n
    =C_{\check{t}}(L_{\beta_{\trop,C}}^{\check{t}})(\check{t}(t)).
\end{equation*}
\end{theorem}

This confirms the result in \cite{Ruddat-Siebert_2019} that the parameters $\check{t}$ in the Gross-Siebert model are the canonical coordinates in mirror symmetry.

This paper is organized as follows: Section \ref{section2} gives an introduction to local mirror symmetry and its Gamma conjecture. Section \ref{section3} is the main body of this paper, where we show how to relate $E_C$ and $L^t_{\trop,C}$ through the tropical curve $\beta_{\trop,C}$, and interpret the central charges using the combinatorial information of $\beta_{\trop,C}$. Section \ref{section4} studies the same story in the Gross-Siebert model of local mirror symmetry and we show it is equivalent to the Hori-Vafa mirror through the mirror map.

\section*{Acknowledgement}
I am grateful to Eric Zaslow for bringing up this problem to me. I would like to thank Helge Ruddat and Bernd Siebert for the patient explanation of their work. I would also like to thank Hiroshi Iritani, Cheuk Yu Mak, Ilia Zharkov, Diego Matessi, and Bohan Fang for the helpful discussions and comments.

\section{Local Mirror Symmetry and its Gamma Conjecture}
\label{section2}
Local mirror symmetry studies the Gromov-Witten invariants of a noncompact manifold which are contributed from the holomorphic curves on a compact Fano submanifold of it and its mirror model. It is firstly proposed by Chiang, Klemm, Yau and Zaslow \cite{C-K-Y-Z_1999}. In this section, we introduce local mirror symmetry for the canonical bundle of a smooth toric Fano variety and its Gamma conjecture as proposed by Hosono \cite{Hosono_2004}.

Suppose $X_\Sigma$ is a smooth projective Fano toric variety with $\Sigma$ a smooth Fano fan in $N_{\R}\cong{\R}^n$, whose generators of rays are
$$
v_1=(1,0,\cdots,0), v_2=(0,1,0\cdots,0), \dots,v_{n}=(0,0,\cdots,0,1),
$$
$$
v_{n+1}=(v_{n+1,1},v_{n+1,2},\cdots,v_{n+1,n}),\dots,v_{p}=(v_{p,1},v_{p,2},\cdots,v_{p,n}),
$$
such that the cone generated by $v_1,v_2,\dots,v_n$ is in $\Sigma$.
We denote the toric divisors corresponding to $v_j$ by $D_j$.

Suppose $D=\sum_{j=1}^pa_jD_j$ is a toric divisor, then the polyhedran $G_D$ corresponding to $D$ is a polyhedran in $M_{\R}$
$$
    G_D:=\{m\in M_{\R}\ |\ \langle m,v_j\rangle\geq-a_j\},
$$
where $M=\Hom(N,\mathbb{Z})$ and $M_{\R}=M\otimes\R$. In particular, the polyhedron corresponding to the anticanonical divisor $D_{\text{can}}=D_1+D_2+\cdots+D_p$ of $X_{\Sigma}$ is $G_{\text{can}}:=\{m\in M_{\R}\ |\ \langle m,v_j \rangle\geq -1,j=1,2,\dots,p\}$.

The sections of the line bundle $\mathcal{O}_{X_\Sigma}(D)$ are encoded in the combinatorial information of $G_D$.

\begin{proposition}[\cite{Cox_2011}, Proposition 4.3.3]
\label{prop_globalsection_polytope}
Suppose $D$ is a toric divisor and $G_D$ is the polyhedron corresponding to it, then
$$
    \Gamma(X_\Sigma,\mathcal{O}_{X_\Sigma}(D))=
    \bigoplus_{m\in G_D\cap M}\mathbb{C}\cdot\chi(m),
$$
where $\chi(m)=z^m$ is an element in $\mathbb{C}[M]$ which can be regarded as a function on the open dense torus of $X_\Sigma$.
\end{proposition}

\begin{corollary}
\label{corollary_polynomial-divisor}
Suppose $f\in\mathbb{C}[M]$ is a polynomial whose Newton polytope is $G_D$ and $Z(f)$ is the zero locus of $f$ in $(\mathbb{C}^*)^n$, then the closure of $Z(f)$ in $X_\Sigma$ is a divisor equivalent to $D$, i.e.,
$$
    [\overline{Z(f)}]=[D].
$$
\end{corollary}

The \textit{support function} $\varphi_D$ corresponding to $D$ is a piecewise linear function on $|\Sigma|$,
$$
    \varphi_{D}:|\Sigma|\rightarrow \R,
$$
which is induced by $\varphi_{D}(v_j)=a_j,j=1,2,\dots,p$. In particular, we denote the support function corresponding to $D_{\text{can}}$ by $\varphi_{\text{can}}$.

The Fano condition gives some propositions about the support function and the moment polytope.

\begin{proposition}[\cite{Cox_2011}, Theorem 6.1.14]
\label{prop_Fano-support}
    The support function $\varphi_D$ corresponding to an ample toric divisor $D$ is strictly convex, i.e., $\varphi_D$ is convex and $\varphi_D(x_1+x_2)<\varphi_D(x_1)+\varphi_D(x_2)$ if there are no cones of $\Sigma$ which contain both $x_1$ and $x_2$.
\end{proposition}

\begin{corollary}
\label{cor_Fanocor1}
    Suppose $P_{V}$ is the polytope with vertices being $v_1,v_2,\dots,v_p$, then $P_{V}$ is a convex polytope. Furthermore, it is strictly convex at each vertex $v_j$, i.e., there exists a hyperplane $H_j$ such that $P_V\backslash v_j$ is contained in one side of $\R^n$ divided by $H_j$.
\end{corollary}

\begin{proof}
Since $D_{\text{can}}=D_1+D_2+\cdots+D_p$ is ample, its support function $\varphi_{\text{can}}$ is strictly convex by Proposition \ref{prop_Fano-support}. Suppose $P_V$ is not strictly convex at some vertex, then there exist two points $x,x'\in\partial P_V$ which are not contained in the same cone of $\Sigma$, and some constant $\lambda\in[0,1]$ such that $\lambda x+(1-\lambda)x'\in N_{\R}\backslash P_V^\circ$ with $P_V^\circ$ being the interior of $P_V$. Then we have that $\varphi_{\text{can}}(\lambda x+(1-\lambda)x')\geq1=\lambda \varphi_{\text{can}}(x)+(1-\lambda)\varphi_{\text{can}}(x')$ since $\varphi_{\text{can}}(x)\geq1$ for $x\in N_{\R}\backslash P_V^\circ$. This contradicts the strict convexity of $\varphi_{\text{can}}$.
\end{proof}

\begin{proposition}[\cite{Cox_2011}, Theorem 8.3.4]
\label{prop_Fano-reflexive}
    If $X_\Sigma$ is a smooth projective Fano toric variety, then $G_{\text{can}}$ is a reflexive polytope, i.e., it has the origin as its unique interior integer point. Conversely, $\Sigma$ is the normal fan of $G_{\text{can}}$.
\end{proposition}

\begin{corollary}
\label{cor_Fanocor2}
    For $s\in\{1,2,\dots,p-n\}$, $\sum_{i=1}^nv_{n+s,i}-1<0$.
\end{corollary}

\begin{proof}
Since the cone generated by $v_1,v_2,\dots,v_n$ is in $\Sigma$ and $\Sigma$ is the normal fan of $G_{\text{can}}$ according to Proposition \ref{prop_Fano-reflexive}, we have that $(-1,-1,\cdots,-1)\in G_{\text{can}}$. So we have that $\langle(-1,-1,\dots,-1),v_{n+s}\rangle>-1$, thus $\sum_{i=1}^nv_{n+s,i}-1<0$.
\end{proof}  

Now let $X=K_{X_\Sigma}$ be the canonical bundle of $X_\Sigma$, which is a smooth noncompact Calabi-Yau toric variety determined by a smooth fan $\tilde{\Sigma}$ in $\tilde{N}_{\R}\cong N_{\R}\oplus\R$, whose generators of rays are
$$
\tilde{v}_0=(0,0,\cdots,0,1),\Tilde{v}_1=(1,0,\cdots,0,1), \tilde{v}_2=(0,1,0\cdots,0,1), \dots,\tilde{v}_{n}=(0,0,\cdots,0,1,1),
$$
$$
\tilde{v}_{n+1}=(v_{n+1,1},v_{n+1,2},\cdots,v_{n+1,n},1),\dots,\tilde{v}_{p}=(v_{p,1},v_{p,2},\cdots,v_{p,n},1).
$$
We denote the toric divisors corresponding to $\tilde{v}_j$ by $\tilde{D}_j$.

The projection from $\tilde{N}_{\R}$ to $N_{\R}$ induces a map from $\tilde{\Sigma}$ to $\Sigma$, which induces a map 
$$
    \pi:X\rightarrow X_{\Sigma}.
$$
Then $\tilde{D}_j=\pi^{-1}(D_j)$ for $j=1,2,\cdots,p$ and $\tilde{D}_0$ is the zero section of $\pi$.

The Picard group of $X$ is given by the exact sequence
$$
    0\longrightarrow \tilde{M}\longrightarrow \mathbb{Z}^{p+1}\longrightarrow \Pic(X)\longrightarrow 0,
$$
where $\tilde{M}$ is the dual lattice to $\tilde{N}$ and the map $\tilde{M}\longrightarrow \mathbb{Z}^{p+1}$ is given by $\tilde{m}\mapsto(\langle \tilde{m},\tilde{v}_0\rangle,\langle \tilde{m},\tilde{v}_1\rangle,\dots,\langle \tilde{m},\tilde{v}_p\rangle)$. If we use $e_0,e_1,\dots,e_{p}$ to denote the points $(1,0,\cdots,0),\\(0,1,\cdots,0),\dots,(0,0,\cdots,1)$ in $\mathbb{Z}^{p+1}$, then the images of $e_0,e_1,\dots,e_{p}$ in $\Pic(X)$ are the linear equivalence classes of the toric divisors $[\tilde{D}_0],[\tilde{D}_1],\dots,[\tilde{D}_{p}]$ corresponding to $\tilde{v}_0,\tilde{v}_1,\dots,\tilde{v}_{p}$.
Since $X$ is a smooth toric variety, $\Pic(X)\cong{\Z^{p-n}}$ and it can be generated by $[\tilde{D}_{n+1}],[\tilde{D}_{n+2}],\dots,[\tilde{D}_{p}]$.

The Mori cone $NE(X)$ of $X$ is then a cone in $\R^{p+1}$ whose generators dual to $[\tilde{D}_{n+1}],[\tilde{D}_{n+2}],\dots,[\tilde{D}_{p}]$ are

\begin{equation}
\label{eq_mori_generators}
    \begin{aligned}
        l^{(1)}=&(\sum_{i=1}^{n}v_{n+1,i}-1,-v_{n+1,1},-v_{n+1,2},\cdots,-v_{n+1,n},1,0,\cdots,0),\\
        l^{(2)}=&(\sum_{i=1}^{n}v_{n+2,i}-1,-v_{n+2,1},-v_{n+2,2},\cdots,-v_{n+2,n},0,1,\cdots,0),\\
        \vdots\\
        l^{(p-n)}=&(\sum_{i=1}^{n}v_{p,i}-1,-v_{p,1},-v_{p,2},\cdots,-v_{p,n},0,0,\cdots,0,1).\\
    \end{aligned}
\end{equation}

The Hori-Vafa mirror of $X$ \cite{H-I-V_2000} is then a family of noncompact Calabi-Yau manifolds $Y^t$ in $\mathbb{C}^2\times(\mathbb{C}^*)^n$ which is given by 

\begin{equation}
\label{eq_local_mirror_symmetry}
    Y^t=\{(u,v,z_1,z_2,\cdots,z_n)\in\mathbb{C}^2\times(\mathbb{C}^*)^n\ |\ uv=W_\Sigma(t_1,t_2,\dots,t_{p-n},z_1,z_2,\dots,z_p)\},
\end{equation}
where $t_1,t_2,\cdots,t_{p-n}$ are complex parameters and $W_\Sigma(t_1,t_2,\dots,t_{p-n},z_1,z_2,\dots,z_p)=1+\sum_{i=1}^n z_i+\sum_{s=1}^{p-n}t_sz^{v_{n+s}}$. 

\begin{definition}
\label{def_centralcharge_period}
    Given a cycle $\gamma_t\in H_{n+1}(Y^t,\mathbb{Z})$, the \textit{period integral}, or the $\textit{central charge}$ $\Pi_\gamma(t)$ of $\gamma_t$ is the following integral on $\gamma_t$,
    \begin{equation}
        \Pi_\gamma(t)=\int_{\gamma_t}\Omega_{Y^t},
    \end{equation}
    where $\Omega_{Y^t}=\Residue_{(uv-W_\Sigma(t,z)=0)}(\frac{1}{uv-W_\Sigma(t,z)}dudv\frac{dz_1}{z_1}\cdots\frac{dz_n}{z_n})$ is a holomorphic form on $Y^t$. We also denote the central charge of $\gamma$ by $C_t(\gamma)$.
\end{definition}

\begin{proposition}[\cite{C-K-Y-Z_1999}]
    \label{prop_Picard-Fuchs}
    The period integral $\Pi_\gamma(t)$ satisfies the following system of differential equations:
    \begin{align}
    \label{eq_Picard-Fuchs}
        \mathcal{L}_s(\Pi_\gamma(t))=0,\, s=1,2,\cdots,p-n,
    \end{align}
    where $\mathcal{L}_s=\prod_{l^{(s)}_j>0}\left(\vartheta_j(\vartheta_j-1)\cdots(\vartheta_j-(l^{(s)}_j-1))\right)-t_s\prod_{l^{(s)}_j<0}\left(\vartheta_j(\vartheta_j-1)\cdots(\vartheta_j-(-l^{(s)}_j-1))\right)$ with $\vartheta_j=\sum_{k=1}^{p-n}l^{(k)}_jt_k\frac{\partial}{\partial t_k}$.
\end{proposition}

The differential equations (\ref{eq_Picard-Fuchs}) can be solved using Frobenius method, i.e., we set
$$
    \Pi_\gamma(t)=\sum_{m_1\geq0,\dots,m_{p-n}\geq0}c(m_1+\rho_1,\dots,m_{p-n}+\rho_{p-n})t_1^{m_1+\rho_1}\cdots t_{p-n}^{m_{p-n}+\rho_{p-n}},
$$
and solve for the coefficients $c(m_1+\rho_1,\dots,m_{p-n}+\rho_{p-n})$. The result is a formal solution 
\begin{align}
\label{eq_solution_to_P-F}
    w(t,\rho)=\sum_{m_1\geq0,\dots,m_{p-n}\geq0}\frac{1}{\prod_{j=0}^p\Gamma(1+\sum_{s=1}^{p-n}l_j^{(s)}(m_s+\rho_s))}t_1^{m_1+\rho_1}\cdots t_{p-n}^{m_{p-n}+\rho_{p-n}}. 
\end{align}

Then the solutions to the Picard-Fuchs equations are given by the partial derivatives of $w(t,\rho)$ with respect to $\rho$ at $\rho=0$. In particular, the solutions
\begin{equation}
    \check{t}_s(t):=\exp(\frac{\partial w(t,\rho)}{\partial\rho_s}|_{\rho=0}),\ s=1,2,\dots,p-n
\label{eq_mir_map}
\end{equation}
are the \textit{mirror map} which transform the complex parameters $t$ to the K\"{a}hler parameters $\check{t}$.

The Gamma Conjecture for local mirror symmetry proposes a matching between the central charges of the mirror objects under homological mirror symmetry, i.e. given a mirror pair $E\leftrightarrow L$ for $E\in D^bCoh(X)$ and $L\in Fuk(Y^t)$, we have that $Z_t(E)=C_t(L)$. We now give the definition of $Z_t(E)$ using the formal solution $w(t,\rho)$.

Let us first introduce the cohomology-valued hypergeometric series induced from $w(t,\rho)$, which is 
\begin{align}
\label{eq_cohohypergeometricseries}
    &w(t_1,\cdots,t_{p-n},\frac{[\tilde{D}_{n+1}]}{2\pi\sqrt{-1}},\cdots,\frac{[\tilde{D}_{p}]}{2\pi\sqrt{-1}})\\
    =&w(t_1,\cdots,t_{p-n},\rho_1,\cdots,\rho_{p-n})|_{\rho_1=\frac{-[\tilde{D}_{n+1}]}{2\pi\sqrt{-1}},\dots,\rho_{p-n}=\frac{-[\tilde{D}_{p}]}{2\pi\sqrt{-1}}}\nonumber\\
    =&\sum_{m_1\geq0,\cdots,m_{p-n}\geq0}\frac{1}{\prod_{j=0}^p\Gamma(1+\sum_{s=1}^{p-n}l^{(s)}_j(m_s-\frac{[\tilde{D}_{n+s}]}{2\pi\sqrt{-1}}))}t_1^{m_1-\frac{[\tilde{D}_{n+1}]}{2\pi\sqrt{-1}}}\cdots t_{p-n}^{m_{p-n}-\frac{[\tilde{D}_{p}]}{2\pi\sqrt{-1}}}.\nonumber
\end{align}

\begin{definition}
\label{def_centralcharge_chernclass}
For an element $E\in D^bCoh(X)$, the \textit{central charge} $Z_t(E)$ of $E$ is defined to be
\begin{equation}
\label{eq_centralcharge_chernclass}
    Z_t(E)=(2\pi\sqrt{-1})^{n+1}\int_{X}w(t_1,\cdots,t_{p-n},\frac{[\tilde{D}_{n+1}]}{2\pi\sqrt{-1}},\cdots,\frac{[\tilde{D}_{p}]}{2\pi\sqrt{-1}})\ch(E)\todd_{X},
\end{equation}
where $\ch(E)$ is the Chern character of $E$, $\todd_{X}$ is the Todd class of $X$, and $\int_{X}$ is understood as taking intersection numbers by regarding $[\tilde{D}_{n+s}]$ and $\ch(E)$ as elements in the Chow group of $X$.
\end{definition}

To describe the Gamma Conjecture for $X$, we need the information of the K-group of $X$. Suppose $K(X)$ and $K^c(X)$ are the K-group of algebraic vector bundles on $X$ and the K-group of the complexes of algebraic vector bundles on $X$ which are exact on $X\backslash\tilde{D}_0$. There is a complete pairing between $K(X)$ and $K^c(X)$ which is 

\begin{align}
    \langle\cdot,\cdot\rangle: K(X)\times K^c(X)&\rightarrow\mathbb{Z},\\
    (F,S)&\mapsto\langle F,S\rangle=\int_{X}\ch(S)\ch(F)\todd_{X}.\nonumber
\end{align}

Now let us take a basis $\{F_1,F_2,\cdots,F_d\}$ which generate the K-group $K(X)$ and $\{S_1,S_2,\cdots,S_d\}$ the dual basis of $K^c(X)$ in the sense that $\langle F_i,S_j\rangle=\delta_{ij}$. Then the Gamma conjecture for local mirror symmetry can be stated as follows:

\begin{conjecture}[The Gamma Conjecture for Local Mirror Symmetry \cite{Hosono_2004}] If we expand the cohomology-valued hypergeometric series (\ref{eq_cohohypergeometricseries}) with respect to $\{\ch(F_k)\}$, which is a basis for $H^{even}_c(X)$, then it is conjectured that 
\begin{equation}
    w(t_1,\cdots,t_{p-n},\frac{[\tilde{D}_{n+1}]}{2\pi\sqrt{-1}},\cdots,\frac{[\tilde{D}_{p}]}{2\pi\sqrt{-1}})=(\frac{1}{2\pi\sqrt{-1}})^{n+1}\sum_{k=1}^d\left((\int_{mir(S_k)}\Omega_{Y^t})\ch(F_k)\right),
\end{equation}
where $mir(S_k)$ is the mirror object of $S_k$ under homological mirror symmetry and $\Omega_{Y^t}$ is the holomorphic form defined in Definition \ref{def_centralcharge_period}. This can also be described as the matching of the central charges of the mirror pair $(E,mir(E))$, i.e.,
\begin{align}
    Z_t(E)&=(2\pi\sqrt{-1})^{n+1}\int_{X}w(t_1,\cdots,t_{p-n},\frac{[\tilde{D}_{n+1}]}{2\pi\sqrt{-1}},\cdots,\frac{[\tilde{D}_{p}]}{2\pi\sqrt{-1}})\ch(E)\todd_{X}\\
    &=\sum_{k=1}^d\left((\int_{mir(S_k)}\Omega_{Y^t})\langle F_k,E \rangle\right)\nonumber\\
    &=\int_{mir(E)}\Omega_{Y^t}\nonumber\\
    &=C_t(mir(E)).\nonumber
\end{align}
\end{conjecture}

\begin{remark}
    The Todd class can be written as $\todd(X)=\prod_{j=0}^p\left(\Gamma(1-\frac{[\tilde{D}_j]}{2\pi\sqrt{-1}})\Gamma(1+\frac{[\tilde{D}_j]}{2\pi\sqrt{-1}})\right)$. Plug it into (\ref{eq_centralcharge_chernclass}) and we can see that it matches with the formula as in \cite{A-G-I-S_2018}. The Gamma class that appears in the formula accounts for the name `Gamma Conjecture'. 
\end{remark}

\section{The Gamma Conjecture through Tropical Geometry}
\label{section3}
In this section we show how to relate the mirror objects through tropical geometry in the base space of local mirror symmetry. We start from an introduction to tropical geometry.

\subsection{Background of Tropical Geometry}
\label{section3.1}
The \textit{tropical semiring} $(\mathbb{R}\cup\{-\infty\},\oplus,\odot)$ is the set of $\mathbb{R}\cup\{-\infty\}$ paired with two operations addition $\oplus$ and multiplication $\odot$ given by:
$$
    x\oplus y=\text{max}\{x,y\};\ \ \ x\odot y=x+y,
$$
where $-\infty\oplus x=x$ and $-\infty\odot x = -\infty$.

A \textit{tropical polynomial} is a linear combination of finite tropical monomials
$$
    \phi(x_1,x_2,\dots,x_n)=(a\odot x_1^{i_1}x_2^{i_2}\cdots x_n^{i_n})\oplus(b\odot x_1^{j_1}x_2^{j_2}\cdots x_n^{j_n})\oplus\cdots,
$$
where $a,b,\dots$ are real numbers and $i_1, j_1,\dots$ are integers.

The tropical polynomial is precisely the piecewise-linear convex function on $\mathbb{R}^n$ with integer coefficients, which is
$$
    \phi(x_1,x_2,\dots,x_n)=\text{max}\{a+i_1x_1+i_2x_2+\cdots i_nx_n,b+j_1x_1+j_2x_2+\cdots j_nx_n,\cdots\}.
$$

We define the following items related to a tropical polynomial $\phi$, which we will use in the later sections.
\begin{definition} Given a tropical polynomial $\phi(x_1,\dots,x_n)=\bigoplus_v (a_v\odot x^v)$,
\begin{itemize}
    \item The \textit{tropical hypersurface} $V(\phi)$ is the subset of $\mathbb{R}^n$ where the corresponding convex piecewise-liner function $\phi$ fails to be linear.
    \item The \textit{Newton polytope} $\Delta_{\phi}$ is $\text{Conv}\{v\}$, the convex hull of the integer degrees $v$'s appearing in the monomials..
    \item The \textit{Legendre transform} $\check{\phi}$ of $\phi$ is the piecewise-linear function on $\Delta_\phi$ which is determined by $\check{\phi}(v)=-a_v$.
\end{itemize}
\end{definition}

The Legendre transform $\check{\phi}$ induces a subdivision $T_\phi$ of $\Delta_\phi$ by projecting the affine parts of the convex hull of $\graph(\check{\phi})$ to $\Delta_\phi$. The tropical hypersurface $V(\phi)$ is then related to $(\Delta_\phi,T_\phi)$ through the following proposition.

\begin{proposition}[\cite{Mikhalkin_2004}, Proposition 2.1]
\label{prop_trop_dual_polytope}
Suppose $\phi$ is a tropical polynomial, then the corresponding tropical hypersurface $V(\phi)$ is the polyhedral complex which is dual to the subdivision $T_\phi$ of $\Delta_\phi$.
\end{proposition}

The union of tropical hypersurfaces is again a tropical hypersurface.

\begin{proposition}[\cite{M-S_2015}, Theorem 3.2.5]
\label{prop_unionoftrophypersurface}
    Suppose $\phi_1,\phi_2,\dots,\phi_m$ are tropical polynomials and $\phi=\phi_1\odot\phi_2\odot\cdots\odot\phi_m$. Then the tropical hypersurface $V(\phi)$ is the union of tropical hypersurfaces $\bigcup_{i=1}^mV(\phi_i)$.
\end{proposition}

\begin{definition}
\label{def_trophypersurface_intersect}
The tropical hypersurfaces $V(\phi_1), V(\phi_2),\dots,V(\phi_m)$ are called \textit{transversely intersect} if for any cell $\tau$ of $\bigcup_{i=1}^mV(\phi_i)$ and $V(\phi_{i_1}),V(\phi_{i_2}),\dots,V(\phi_{i_k})$ which contain $\tau$,
$$
\codim \tau=\codim \sigma_{i_1}+\codim \sigma_{i_2}+\cdots+\codim \sigma_{i_k},
$$
where $\sigma_{i_s}$ is the cell of $V(\phi_{i_s})$ which contains $\tau$ with the smallest dimension.
\end{definition}

We give a more detailed analysis of the cells of $V(\phi)$. Note that $V(\phi)$ is dual to $(\Delta_\phi, T_{\phi})$, where $\Delta_\phi=\Delta_{\phi_1}+\Delta_{\phi_2}+\cdots+\Delta_{\phi_m}$, the Minkowski sum of $\Delta_{\phi_1},\Delta_{\phi_2},\dots,\Delta_{\phi_m}$. Now let 
$$
    C(\Delta_{\phi_1},\Delta_{\phi_2},\cdots,\Delta_{\phi_m})=\conv\{e_1\times \Delta_{\phi_1},\dots,e_m\times \Delta_{\phi_m}\}\subset\mathbb{R}^{m+n}
$$
be the Cayley polytope, where $e_k=(0,\cdots,0,1,0,\cdots,0)\in\R^m$ is the integer point with $1$ in the $k$-th position. The integer points of $C(\Delta_{\phi_1},\Delta_{\phi_2},\cdots,\Delta_{\phi_m})$ are the points $e_k\times v_k$ with $v_k$ an integer point of $\Delta_{\phi_k}$. Thus, we can define a piecewise-linear function $\check{\phi}_{\text{mixed}}$ on $C(\Delta_{\phi_1},\Delta_{\phi_2},\cdots,\Delta_{\phi_m})$ by setting $\check{\phi}_{\text{mixed}}(e_k\times v_k)=\check{\phi}_k(v_k)$. The projection of the affine parts of the convex hull of $\graph(\check{\phi}_{\text{mixed}})$ then gives a subdivision $T_{\text{mixed}}$ of $C(\Delta_{\phi_1},\Delta_{\phi_2},\cdots,\Delta_{\phi_m})$. It turns out that the polytope $\Delta_{\phi}$ can be identified with $C(\Delta_{\phi_1},\Delta_{\phi_2},\cdots,\Delta_{\phi_m})\cap \{x_k=\frac{1}{m}\ |\ 1\leq k\leq m\}$ with $x_k$ the $k$th-coordinate of $\mathbb{R}^{m+n}$, and the subdivision $T_{\phi}$ is the same as the restriction of $T_{\text{mixed}}$ to $\Delta_\phi$. Furthermore, if $\sigma_{\text{mixed}}$ is a cell of $T_{\text{mixed}}$ with $\sigma_{\text{mixed}}\cap(e_k\times\R^n)=\sigma_k^{d_k}$, a $d_k$-cell of $T_{\phi_k}$, then the cell $\sigma^{d_1,d_2,\dots,d_m}$ of $T_\phi$ which is identified with $\sigma_{\text{mixed}}|_{\Delta_\phi}$ is of the form
\begin{equation}
\label{eq_mixed_cell}
    \sigma^{d_1,d_2,\dots,d_m}=\sigma_{1}^{d_1}+\sigma_{2}^{d_2}+\cdots+\sigma_{m}^{d_m},
\end{equation}
the Minkowski sum of $\sigma_{1}^{d_1},\sigma_{2}^{d_2},\dots,\sigma_{m}^{d_m}$. Note that if $V(\phi_1), V(\phi_2),\dots,V(\phi_m)$ intersect transversely, $\dim(\sigma^{d_1,d_2,\dots,d_m})=d_1+d_2+\cdots+d_m$. We call the subdivision $T_{\phi}$ and the cells $ \sigma^{d_1,d_2,\dots,d_m}$ \textit{mixed subdivision} and \textit{mixed cells}. The cells of $V(\phi)$ can then be classified by the mixed cells.

\begin{proposition}[\cite{M-S_2015}, Theorem 4.6.9]
\label{prop_dual_mixed_cell}
Suppose $V(\phi_1), V(\phi_2),\dots,V(\phi_m)$ intersect transversely, then the cells of $V(\phi)$ which are contained in $V(\phi_{i_1})\cap V(\phi_{i_2})\cap\cdots\cap V(\phi_{i_k})$ are dual to the mixed cells $\sigma^{d_1,d_2,\dots,d_m}$ with $d_j=0$ for $j\notin\{i_1,i_2,\dots,i_k\}$.
\end{proposition}

We now give the definition of \textit{tropical curve}, which is the intersection of several tropical hypersurfaces.

\begin{definition}
\label{def_tropcurve}
Suppose $V(\phi_1),\cdots,V(\phi_{n-1})$ are $n-1$ tropical hypersurfaces in $\mathbb{R}^n$ which intersect transversely, then the \textit{tropical curve}  $\beta_{\trop}(\phi_1,\phi_2,\cdots,\phi_{n-1})$ is the intersection of $V(\phi_1),\cdots,V(\phi_{n-1})$ whenever it is nonempty, i.e.,
$$
    \beta_{\trop}(\phi_1,\phi_2,\cdots,\phi_{n-1}):=\bigcap_{k=1}^{n-1}V(\phi_i).
$$
\end{definition}

Use Proposition \ref{prop_dual_mixed_cell}, we see that the mixed cell $\sigma_v$ dual to a vertex $v$ of $\beta_{\trop}(\phi_1,\phi_2,\cdots,\phi_{n-1})$ is of the form $\sigma^{d_1,d_2,\dots,d_{n-1}}$ with one of $\{d_1,d_2,\dots,d_{n-1}\}$ being $2$ and the others being $1$. The mixed cell $\sigma_e$ dual to an edge $e$ of $\beta_{\trop}(\phi_1,\phi_2,\cdots,\phi_{n-1})$ is of the form $\sigma^{1,1,\dots,1}$.

\subsection{From a Coherent Sheaf to a Tropical Curve}
\label{section3.2}
In this subsection, we consider a coherent sheaf supported on a holomorphic curve in $X_{\Sigma}$ and evaluate its central charge through tropical geometry.

\subsubsection{The tropicalization of a holomorphic curve}
\label{section3.2.1}
Let us first give the tropicalization of a polynomial with complex coefficients.
\begin{definition}
\label{def_tropicalization_of_polynomial}
Suppose $f\in\C[z_1,z_2,\dots,z_n]$ is given as 
$$
    f(z_1,z_2,\dots,z_n)=az_1^{i_1}\dots z_n^{i_n}+bz_1^{j_1}\dots z_n^{j_n}+\cdots,
$$
then the \textit{tropicalization} of $f$ is the tropical polynomial
$$
    f_{\trop}=(\log|a|\odot x_1^{i_1}x_2^{i_2}\cdots x_n^{i_n})\oplus(\log|b|\odot x_1^{j_1}x_2^{j_2}\cdots x_n^{j_n})\oplus\cdots.
$$
\end{definition}

Now let $\{Q_1,Q_2,\dots,Q_{n-1}\}$ be a family of ample divisors in $X_\Sigma$ such that $[Q_k]=[\sum_{j=1}^{p}a_{k,j}D_j]$, whose corresponding polytope is
$G_k=\{m\in M_{\mathbb{R}}\ |\ \langle m,v_j \rangle\geq -a_{k,j}\}$. The ampleness guarantees that the shapes of $G_k$'s are the same as the shape of the polytope $G_{\text{can}}$. We assume that $Q_k$ are taken generically, in the sense that $Q_k=\overline{Z(f_k)}$ with $\Newton(f_k)=G_k$ according to Proposition \ref{prop_globalsection_polytope}. We further assume that the tropical hypersurfaces $V(f_{k,\trop})$'s intersect transversely. We call such a family of divisors a \textit{nice family} of divisors in $X_\Sigma$. For the simplicity of notations, we abuse $f_k$ to denote $f_{k,\trop}$ in the rest of this section.

\begin{definition}
\label{def_trop-holo_curve}
Suppose $C=\bigcap_{k=1}^{n-1}Q_k$ is a holomorphic curve in $X_\Sigma$ which is a complete intersection of a nice family of divisors $\{Q_1,Q_2,\dots,Q_{n-1}\}$. Then $\beta_{\trop,C}$, the $\textit{tropicalization}$ of $C$, is the tropical curve $\beta_{\trop}(f_1,f_2,\dots,f_{n-1})$, i.e.,
\begin{equation}
\label{eq_trop_curve}
    \beta_{\trop,C}=\bigcap_{k=1}^{n-1}V(f_k).
\end{equation}
\end{definition}

Note that $\bigcup_{k=1}^{n-1}V(f_k)$ is dual to $(G,T_{f})$ with $G=G_1+G_2+\cdots+G_{n-1}$ and $f=f_1\odot f_2\odot\cdots\odot f_{n-1}$. The vertices and edges of $\beta_{\trop,C}$ are then dual to the mixed cells $\sigma_v$'s and $\sigma_e$'s of $T_f$.

\subsubsection{Evaluation of the central charge of a coherent sheaf}
\label{section3.2.2}
Let $\{Q_1,Q_2,\dots,Q_{n-1}\}$ be a nice family of divisors in $X_\Sigma$, and $C=\bigcap_{k=1}^{n-1}Q_k$ be the holomorphic curve which is the complete intersection of them. We consider a coherent sheaf
$$
    E_C=i_*\mathcal{O}_C(\sum_{k=1}^{n-1}\tilde{Q}_k),
$$
where $\tilde{Q}_k$ is the the divisor in $X$ pulled back from $Q_k$ along $\pi:X\rightarrow X_\Sigma$. We do the evaluation of the central charge $Z_t(E_C)$ as in Definition \ref{def_centralcharge_chernclass} in this subsection.

We first show what $\ch(E_C)$ is.

\begin{lemma}
\label{lemma_chern_character}
Then Chern character of the coherent sheaf $E_C$ is 
\begin{equation}
\label{eq_chern_class}
    \ch(E_C)=(1-e^{-[\tilde{D}_0]})\prod_{k=1}^{n-1}(1-e^{-[\tilde{Q}_k]})e^{\sum_{k=1}^{n-1}[\tilde{Q}_k]}.
\end{equation}
\end{lemma}

\begin{proof}
Use the exact sequence
$$
    0\rightarrow\mathcal{O}_{X}(-\tilde{Q}_k)\rightarrow \mathcal{O}_{X}\rightarrow i_*\mathcal{O}_{\tilde{Q}_k}\rightarrow 0,
$$
we have that 
$$
    \ch(i_*\mathcal{O}_{\tilde{Q}_k})=1-e^{-[\tilde{Q}_k]}.
$$
Now note that $E_C=i_*\mathcal{O}_{\tilde{D}_0}\otimes(\bigotimes\limits_{k=1}^{n-1}i_*\mathcal{O}_{\tilde{Q}_k})\otimes\mathcal{O}_X(\sum\limits_{k=1}^{n-1}\tilde{Q}_k)$, we have that 
\begin{equation*}
\begin{aligned}
    \ch(E_C)=&\ch(i_*\mathcal{O}_{\tilde{D}_0})\prod_{k=1}^{n-1}\ch(i_*\mathcal{O}_{\tilde{Q}_k})\ch(\mathcal{O}_X(\sum\limits_{k=1}^{n-1}\tilde{Q}_k))\\
    =&(1-e^{-[\tilde{D}_0]})\prod_{k=1}^{n-1}(1-e^{-[\tilde{Q}_k]})e^{\sum\limits_{k=1}^{n-1}[\tilde{Q}_k]}.
\end{aligned}
\end{equation*}
\end{proof}

Note that the degree of $\ch(E_C)$ is greater than $n$, so only the degree $0$ and degree $1$ terms of $\todd_{K_{X_\Sigma}}$ and $w(t_1,\cdots,t_{p-n},\frac{[\tilde{D}_{n+1}]}{2\pi\sqrt{-1}},\cdots,\frac{[\tilde{D}_{p}]}{2\pi\sqrt{-1}})$ contribute to $Z_t(E_C)$.

For $\todd_{X}$, its degree $0$ term is $1$, and its degree $1$ term is $0$ since $X$ is Calabi-Yau.

For $w(t_1,\cdots,t_{p-n},\frac{[\tilde{D}_{n+1}]}{2\pi\sqrt{-1}},\cdots,\frac{[\tilde{D}_{p}]}{2\pi\sqrt{-1}})$, recall that the generators $l^{(s)}$ of the Mori cone of $X$ are given as in (\ref{eq_mori_generators}), so its degree $0$ term is 
\begin{align*}
    &w(t,0)\\
    =&\sum_{m_1\geq0,\dots,m_{p-n}\geq0}\frac{1}{\prod_{j=0}^p\Gamma(1+\sum\limits_{s=1}^{p-n}l_j^{(s)}m_s)}t_1^{m_1}\cdots t_{p-n}^{m_{p-n}}\\
    =&\sum_{m_1\geq0,\dots,m_{p-n}\geq0}\frac{1}{\Gamma(1+\sum\limits_{s=1}^{p-n}m_s(\sum\limits_{i=1}^nv_{n+s,i}-1))\prod\limits_{i=1}^n\Gamma(1+\sum\limits_{s=1}^{p-n}m_s(-v_{n+s,i}))\prod\limits_{s=1}^{p-n}\Gamma(1+m_s)}t_1^{m_1}\cdots t_{p-n}^{m_{p-n}}\\
    =&1,
\end{align*}
where the third equation is because $(\sum\limits_{i=1}^nv_{n+s,i}-1)<0$ according to Corollary \ref{cor_Fanocor2}, thus only the term with all $m_s=0$ remains. Its degree $1$ term is

\begin{equation}
\label{eq_w(t)_deg1}
    \sum_{s=1}^{p-n}\frac{\partial w(t,\rho)}{\partial\rho_s}|_{\rho=0}\frac{-[\tilde{D}_{n+s}]}{2\pi\sqrt{-1}}=\sum_{s=1}^{p-n}\log\check{t}_s(t)\frac{-[\tilde{D}_{n+s}]}{2\pi\sqrt{-1}}
\end{equation}
according to (\ref{eq_mir_map}).

Plugging (\ref{eq_chern_class}) and (\ref{eq_w(t)_deg1}) into (\ref{eq_centralcharge_chernclass}), we have that the central charge $Z_t(E_C)$ is
\begin{equation}
\label{eq_result_centralcharge_chernclass}
    \begin{aligned}
        &Z_t(E_C)\\
        =&(2\pi\sqrt{-1})^{n}\sum\limits_{s=1}^{p-n}\Big((-[\tilde{D}_{n+s}][\tilde{D}_0]\prod_{k=1}^{n-1}[\tilde{Q}_k])\log \check{t}_s(t)\Big)+\\
        &(2\pi\sqrt{-1})^{n+1}(-\frac{1}{2})\big(([\tilde{D}_0]+\sum\limits_{k=1}^{n-1}[\tilde{Q}_k])[\tilde{D}_0]\prod\limits_{k=1}^{n-1}[\tilde{Q}_k]\big){+(2\pi\sqrt{-1})^{n+1}([\tilde{D}_0]\prod\limits_{k=1}^{n-1}[\tilde{Q}_k]\sum\limits_{k=1}^{n-1}[\tilde{Q}_k])}\\
        =&(2\pi\sqrt{-1})^{n}\sum\limits_{s=1}^{p-n}\Big(([D_{n+s}]\prod_{k=1}^{n-1}[Q_k])(-\log \check{t}_s(t))\Big)+(2\pi\sqrt{-1})^{n+1}\frac{1}{2}\big((\sum\limits_{j=1}^p[D_j]{+}\sum\limits_{k=1}^{n-1}[Q_k])\prod\limits_{k=1}^{n-1}[Q_k]\big),
\end{aligned}
\end{equation}
where the second equation is because $[\tilde{D}_0]=[X_{\Sigma}]$ and $\sum_{j=0}^p[\tilde{D}_j]=0$.

\subsubsection{Interpretation of the central charge through tropical geometry}
\label{section3.2.3}
We want to interpret the central charge $Z_t(E_C)$ using the combinatorial information of $\beta_{\trop,C}$. To do it, we need the concept of \textit{mixed volume}.

\begin{proposition}[H. Minkowski]
    Suppose $P_1,P_2,\dots,P_m$ are $m$ polytopes in $\mathbb{R}^n$, and let $P(\lambda_1,\dots,\lambda_m)$ be the Minkowski sum of the polytopes $\lambda_1P_1,\lambda_2P_2,\dots,\lambda_{m}P_m$ with $\lambda_k$'s nonnegative real numbers, then the volume of $P(\lambda_1,\dots,\lambda_m)$ with respect to the canonical metric of $\mathbb{R}^n$ is a real polynomial of degree $n$ in $\lambda_1,\lambda_2,\dots,\lambda_m$.
\end{proposition}

\begin{definition}[\cite{Ewald_1996}, Definition 3.3]
    If we write the volume of $P(\lambda_1,\dots,\lambda_m)$ as 
    $$
        \vol(P(\lambda_1,\dots,\lambda_m))=\sum_{a_1,a_2,\dots,a_n=1}^{m}\MV(P_{a_1},P_{a_2},\cdots,P_{a_n})\lambda_{a_1}\lambda_{a_2}\cdots\lambda_{a_n},
    $$
    then the coefficient $\MV(P_{a_1},P_{a_2},\cdots,P_{a_n})$ is called the \textit{mixed volume} of the polytopes $P_{a_1},P_{a_2},\cdots,P_{a_n}$.
\end{definition}

The mixed volume is the sum of the volumes of mixed cells.

\begin{proposition}[\cite{M-S_2015}, Proof of Proposition 4.6.3]
\label{prop_mixed_volume}
Suppose $\phi_1,\phi_2,\dots,\phi_m$ are $m$ tropical polynomials whose tropical hypersurfaces $V(\phi_1),V(\phi_2),\dots,V(\phi_m)$ intersect transversely, and $\phi=\phi_1\odot\phi_2\odot\cdots\odot\phi_{m}$. Let $(\Delta_{\phi_1},T_{\phi_1}),(\Delta_{\phi_2},T_{\phi_2}),\dots,(\Delta_{\phi_m},T_{\phi_m})$ and $(\Delta_{\phi},T_{\phi})$ be their Newton polytopes and corresponding subdivisions. Then when $d_1+d_2+\cdots+d_m=n$, the coefficients of $\lambda_1^{d_1}\lambda_2^{d_2}\cdots\lambda_m^{d_m}$ in $\vol(\Delta_\phi(\lambda_1,\lambda_2,\dots,\lambda_m))$ is equal to the sum of the mixed cells in $T_{\phi}$ of the form $\sigma^{d_1,d_2,\dots,d_m}$ as in (\ref{eq_mixed_cell}), i.e.,
\begin{equation}
\label{eq_sum_mixed_cell_mixed_volume}
\sum\limits_{\sigma^{d_1,d_2,\dots,d_m}\in T_{\phi}}\vol(\sigma^{d_1,d_2,\dots,d_m})=\sum\limits_{\{a_1,\dots,a_n\}=\{\underbrace{1,\dots,1}_{d_1\text{ times}},\dots,\underbrace{m,\dots,m}_{d_m\text{ times}}\}}\MV(\Delta_{\phi_{a_1}},\dots,\Delta_{\phi_{a_n}}).
\end{equation}
\end{proposition}

For our case of the tropical curve $\beta_{\trop,C}$ as in Definition \ref{def_tropcurve}, Proposition \ref{prop_mixed_volume} can be applied to the mixed cells $\sigma_v$ and we have that
\begin{equation}
\label{eq_sigma_v_mixed_volume}
    \begin{aligned}
        &\sum\limits_{v\in\{\text{vertices of }\beta_{\trop,C}\}}\vol(\sigma_v)\\ 
        =&\sum\vol(\sigma^{2,1,\dots,1})+\sum\vol(\sigma^{1,2,\dots,1})+\cdots+\sum\vol(\sigma^{1,1,\dots,2})\\ 
        =&\sum\limits_{\{a_1,\dots,a_n\}=\{1,1,2,3,\dots,n-1\}}\MV({G_{a_1}},\dots,{G_{a_n}})+\sum\limits_{\{a_1,\dots,a_n\}=\{1,2,2,3,\dots,n-1\}}\MV({G_{a_1}},\dots,{G_{a_n}})\\ 
        &+\cdots+\sum\limits_{\{a_1,\dots,a_n\}=\{1,2,3,\dots,n-2,n-1,n-1\}}\MV({G_{a_1}},\dots,{G_{a_n}})\\ 
        =&\sum_{k=1}^{n-1}\frac{n!}{2}\MV(G_k,G_1,G_2,\cdots,G_{n-1}).
\end{aligned}
\end{equation}

It can also be applied to the mixed cells $\sigma_e$. Note that $G_1,G_2,\dots,G_{n-1}$ and $G$ all have the same shape as $G_{\text{can}}$. Thus they all have $p$ facets, which are dual to $v_j,j=1,2,\dots,p$. We denote them by $F_{j,1},F_{j,2},\dots,F_{j,n-1},F_j$ respectively. Then we want to apply Proposition \ref{prop_mixed_volume} to a single facet $F_j$. To do it, let $H_j$ be the $(n-1)$-subspace of $\R^n$ which is parallel to the hyperplane containing $F_j$. Choose a basis $\{\xi_j^{\perp,1},\xi_j^{\perp,2},\dots,\xi_j^{\perp,n-1}\}$ of the sublattice $H_j\cap\Z^n$, and $\gamma_j$ such that $\{\xi_j^{\perp,1},\xi_j^{\perp,2},\dots,\xi_j^{\perp,n-1},\gamma_j\}$ is a basis of $\Z^n$. Then let 
$$
    \Psi_j:\R^{n}\rightarrow \R^{n-1}
$$
be the linear map which is induced by $\Psi_j(\xi_j^{\perp,i})=e_i,i=1,2,\dots,n-1$ and $\Psi_j(\gamma_j)=0$. It induces a map 
$$
    \Psi_{j,\poly}:\C[z_1,z_2,\dots,z_n]\rightarrow\C[z_1,z_2,\dots,z_{n-1}]
$$
which is induced by $z^v=z^{\psi_j(v)}$. Then the tropical hypersurfaces $V(\Psi_{j,\poly}(f_k))$ and $\bigcup\limits_{k=1}^{n-1}V(\Psi_{j,\poly}(f_{k}))$ are dual to $(\Psi_j(F_{j,k}),\Psi_j(T_{f_k}|_{F_{j,k}}))$ and $(\Psi_j(F_{j}),\Psi_j(T_{f}|_{F_{j}}))$. Then apply Proposition \ref{prop_mixed_volume} to the mixed cells $\sigma_e$ contained in $F_j$, we have that 
\begin{equation}
\label{eq_sigma_e_mixed_volume}
    \begin{aligned}
    &\sum\limits_{\sigma_e\subset F_j}\vol_{\text{int}}(\sigma_e)\\ 
    =&\sum\limits_{\{a_1,\dots,a_{n-1}\}=\{1,2,3,\dots,n-1\}}\MV(\Psi_j(F_{j,a_1}),\Psi_j(F_{j,a_2}),\dots,\Psi_j(F_{j,a_{n-1}}))\\ 
    =& (n-1)!\MV(\Psi_j(F_{j,1}),\Psi_j(F_{j,2}),\dots,\Psi_j(F_{j,n-1})),
\end{aligned}
\end{equation}
where $\vol_{\text{int}}(\sigma_e)=\vol(\Psi_j(\sigma_e))$.

Now we relate the intersection numbers $[D_j]\prod_{k=1}^{n-1}[Q_k]$ and $(\sum_{k=1}^{n-1}[Q_k])\prod_{k=1}^{n-1}[Q_k]$ in (\ref{eq_result_centralcharge_chernclass}) with the mixed volumes using Bernstein's theorem.

\begin{theorem}[Bernstein's Theorem \cite{Bernstein_1975}]
\label{thm_Bernsteinthm}
For a generic choice of the polynomials $f_1,f_2\dots,f_n$ whose Newton polytopes are $G_1,G_2,\dots,G_n$, the number of the solutions in $(\mathbb{C}^*)^{n}$ of the system
$$
    f_1(z)=f_2(z)=\dots=f_n(z)=0
$$
is equal to $n!\MV(G_1,G_2,\cdots,G_n)$.
\end{theorem}

Applying Bernstein's Theorem \ref{thm_Bernsteinthm} to (\ref{eq_result_centralcharge_chernclass}), we have that 
\begin{equation}
\label{eq_intersection_number_vertex_mixed_volume}
    (\sum_{k=1}^{n-1}[Q_k])\prod_{k=1}^{n-1}[Q_k]=n!\sum_{k=1}^{n-1}\MV(G_k,G_1,G_2,\cdots,G_{n-1})
\end{equation}
and
\begin{equation}
\label{eq_intersection_number_edge_mixed_volume}
    [D_j]\prod_{k=1}^{n-1}[Q_k]=(n-1)!\MV(\Psi_j(F_{1,j}),\dots,\Psi_j(F_{n-1,j}))
\end{equation}
for $j=1,2,\dots,p$.

Now we conclude the proof of Theorem \ref{thm_maintheorem_1}

\begin{proof}[Proof of Theorem \ref{thm_maintheorem_1}]
Combining (\ref{eq_sigma_v_mixed_volume}) and (\ref{eq_sigma_e_mixed_volume}) with (\ref{eq_intersection_number_vertex_mixed_volume}) and (\ref{eq_intersection_number_edge_mixed_volume}), we have that $$
    (\sum_{k=1}^{n-1}[Q_k])\prod_{k=1}^{n-1}[Q_k]=2\sum\limits_{v\in\{\text{vertices of }\beta_{\trop,C}\}}\vol(\sigma_v)$$
and 
$$
    [D_j]\prod_{k=1}^{n-1}[Q_k]=\sum\limits_{\sigma_e\subset F_j}\vol_{\text{int}}(\sigma_e)
$$
for $j=1,2,\dots,p$. Plug these into (\ref{eq_result_centralcharge_chernclass}), we have that  
\begin{equation}
\label{eq_result2_centralcharge_chernclass}
    Z_t(E_C)=(2\pi\sqrt{-1})^{n}\sum\limits_{s=1}^{p-n}\Big(E_{n+s}(-\log \check{t}_s(t))\Big)+(2\pi\sqrt{-1})^{n+1}\frac{1}{2}\big(\sum\limits_{j=1}^pE_{j}{+}V\big),
\end{equation}

where $V=\sum\limits_{v\in\{\text{vertices of }\beta_{\trop,C}\}}\vol(\sigma_v)$ and $E_j=\sum\limits_{\sigma_e\subset F_j}\vol_{\text{int}}(\sigma_e)$.
\end{proof}

The following is an example of tropical curve for $X_\Sigma=\mathbb{CP}^3$.

\begin{example}
The green and brown polyhedral complex in Figure \ref{CP3example_tropical_objets_mixed} are the tropical hypersurface $V(\phi_1)$ and $V(\phi_2)$ with $Q_1$ and $Q_2$ two degree 1 divisors in $\mathbb{CP}^3$. Let $C=Q_1\cap Q_2$, then the blue graph in Figure \ref{CP3example_tropical_objets_mixed} is the tropicalization $\beta_{\trop,C}$ of $C$. The mixed subdivision of $G$ is shown in Figure \ref{CP3example_dual_subdivision_mixed} and the two vertices of $\beta_{\trop,C}$ are dual to the two prisms in the mixed subdivision. Note that we have $\vol(\sigma_v)=\vol_{int}(\sigma_e)=1$ in this example. We can check that $([Q_1]+[Q_2])[Q_1][Q_2]=2$, which is equal to the number of the vertices of $\beta_{\trop,C}$, and $[D_j][Q_1][Q_2]=1$ for the toric divisor $D_j$ corresponding to the one-cone generated by $v_j$, which is equal to the number of $\sigma_e$'s contained in $F_j$.
\begin{figure}[htbp!]
\begin{subfigure}{.4\textwidth}
    \centering
    \includegraphics[scale=0.7]{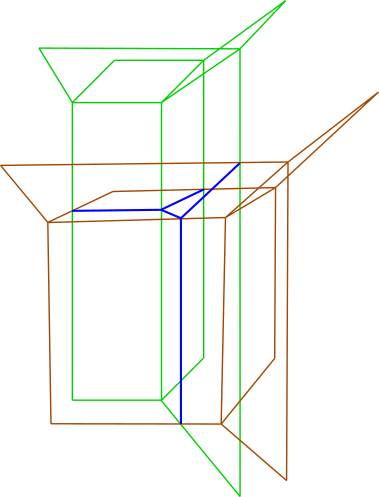}
    \caption{$V(\phi_1)$ and $V(\phi_2)$}
    \label{CP3example_tropical_objets_mixed}
\end{subfigure}
\begin{subfigure}{.4\textwidth}
    \centering
    \includegraphics[scale=0.4]{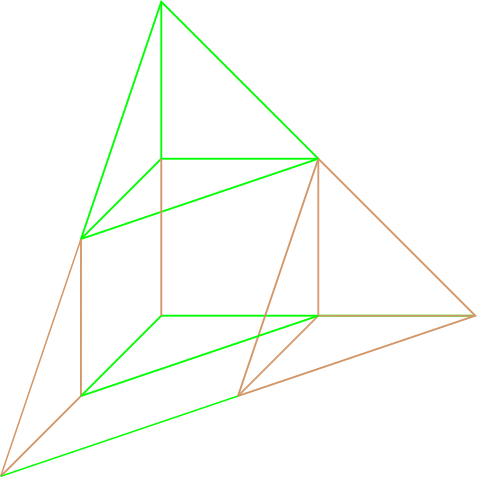}
    \caption{The mixed subdivision}
    \label{CP3example_dual_subdivision_mixed}
\end{subfigure}
\caption{An example of tropical curve in $\mathbb{CP}^3$}
\end{figure}
\end{example}

\subsection{From a Tropical Curve to a Lagrangian Submanifold}
\label{section3.3}
In this subsection we show how to lift the tropical curve $\beta_{\trop,C}$ to a Lagrangian submanifold $L^t_{\trop,C}$ in the Hori-Vafa mirror $Y^t$ and evaluate its central charge $C_t(L_{\trop,C}^t)$.

\subsubsection{The construction of Lagrangian submanifold from a tropical curve}
\label{section3.3.1}
Recall that the Hori-Vafa mirror $Y^t$ is given as in (\ref{eq_local_mirror_symmetry}). In this subsection, rather than considering the multi-parameters $(t_1,t_2,\dots,t_{p-n})\in\C^{p-n}$, we consider the one-parameter family of manifolds by setting $(t_1,t_2,\dots,t_{p-n})=(t^{\lambda_{n+1}},t^{\lambda_{n+2}},\dots,t^{\lambda_{p}})$ for $t\in\R_{>0}$. We further assume that $\sum_{s=1}^{p-n}\lambda_{n+s}D_{n+s}$ is an ample toric divisor. Then we have that
\begin{equation}
\label{eq_lms_one_parameter}
    Y^t=\{(z_1,z_2,\cdots,z_n,u,v)\in(\mathbb{C}^*)^{n}\times\mathbb{C}^2\ |\ uv=W_{\Sigma}(z_1,z_2,\dots,z_n,t)\}
\end{equation}
with
$$
    W_{\Sigma}(z_1,z_2,\dots,z_n,t)=1+\sum_{i=1}^nz_i+\sum_{s=1}^{p-n}t^{\lambda_{n+s}}z^{v_{n+s}}.
$$
We also write 
$$
    W_{\Sigma}(z_1,z_2,\dots,z_n,t)=1+\sum_{j=1}^pt^{\lambda_j}z^{v_j}
$$
by setting $\lambda_1=\lambda_2=\cdots=\lambda_n=0$, and use $W_\Sigma$ or $W_\Sigma(z)$ to denote it as a polynomial in $z_1,z_2,\dots,z_n$ for a fixed $t$ when there is no ambiguity.
\begin{remark}
The choice of $\lambda_{n+1},\lambda_{n+2},\cdots,\lambda_p$ corresponds to a choice of convex multivalued piecewise linear function in the Gross-Siebert model, which is mirror to the choice of K\"{a}hler class $\sum_{s=1}^{p-n}\lambda_{n+s}[\Tilde{D}_{n+s}]$ on $X$. 
\end{remark}

Now let 
$$
    \omega_0=\sum_{i=1}^n\frac{dz_i\wedge d\bar{z}_i}{|z_i|^2}+du\wedge d\bar{u}+dv\wedge d\bar{v}
$$ 
be the canonical symplectic form on $(\mathbb{C}^*)^n\times\mathbb{C}^2$, and $\omega^t:=\omega_0|_{Y^t}$ be the restriction of $\omega_0$ on $Y^t$. Then the submanifold $L^t_{\trop,C}$ in $Y^t$ that we will lift from $\beta_{\trop,C}$ is Lagrangian with respect to $\omega^t$. Consider the moment map $\mu_0$ for $\C^2\times(\C^*)^n$ with respect to $\omega_0$, which is
\begin{align}
\label{momentmap_canonical}
    \mu_0:(\mathbb{C}^*)^n\times\mathbb{C}^2&\rightarrow\mathbb{R}^{n}\times\mathbb{R}_{\geq0}^2\\
    \nonumber(z_1,z_2,\cdots,z_n,u,v)&\mapsto(\log|z_1|,\log|z_2|,\cdots,\log|z_n|,\frac{1}{2}|u|^2,\frac{1}{2}|v|^2),
\end{align}
we want to transfer $\beta_{\trop,C}$ to the base space $\R^n\times\R_{\geq0}^2$ and lift it to a Lagrangian submanifold of $Y^t$ along $\mu_0$. To do it, we need to show what the \textit{amoeba} of $W_{\Sigma}$ looks like.

\begin{definition}
\label{def_amoeba}
Suppose $f\in\C[z_1,\dots,z_n]$ is a polynomial in $z_1,\dots,z_n$ and
\begin{align*}
    \Log:(\mathbb{C}^*)^n&\rightarrow\mathbb{R}^n,\\
    (z_1,z_2,\cdots,z_n)&\mapsto(\log|z_1|,\log|z_2|,\cdots,\log|z_n|),
\end{align*}
is the logarithm map from $(\mathbb{C}^*)^n$ to $\mathbb{R}^n$, then 
the \textit{amoeba} $\mathscr{A}(f)$ of $f$ is the image of the zero locus of $f$ under $\Log$,
i.e.,
$$
    \mathscr{A}(f)=\Log(\{(z_1,\cdots,z_n)\ |\ f(z_1,\cdots,z_n)=0)\}).
$$
We call $\mathscr{A}_t(f)=\Log_t(\mathscr{A}(f))$ the \textit{rescaled amobea} of $f$, where $\Log_t(\cdot)=\frac{\Log(\cdot)}{\log t}$.
\end{definition}

The rescaled amobea $\mathscr{A}_t(W_{\Sigma})$ has a tropical hypersurface as its `skeleton'.

\begin{theorem}[\cite{Mikhalkin_2004}, Corollary 6.4]
\label{thm_mik_trop_amoeba}
The rescaled amoeba $\mathscr{A}_t(W_{\Sigma})$ converges in the Hausdorff metric to the tropical variety corresponding to the tropical polynomial $W_{\Sigma,\trop}$ as $t\rightarrow 0^+$, i.e.,
$$
    \lim_{t\rightarrow0^+} \mathscr{A}_t(W_{\Sigma})=-V(W_{\Sigma,\trop}),
$$
where $W_{\Sigma,\trop}=0\oplus x_1\oplus x_2\oplus\cdots\oplus x_n\oplus(-\lambda_{n+1}\odot\langle v_{n+1},x\rangle)\oplus(-\lambda_{n+2}\odot\langle v_{n+2},x\rangle)\oplus\cdots\oplus(-\lambda_{p}\odot\langle v_{p},x\rangle)$.
\end{theorem}

\begin{remark}
We can regard $W_{\Sigma}$ as a polynomial with coefficients in the field $K\{t\}$ of Puiseux series, then $W_{\Sigma,\trop}$ is given by taking the valuation of its coefficients and regard it as a tropical polynomial.
\end{remark}

Now to see what $V(W_{\Sigma,\trop})$ looks like, note that its Newton polytope is the polytope $P_V$ as in Corollary \ref{cor_Fanocor1} and then $V(W_{\Sigma,\trop})$ can be described by Proposition \ref{prop_trop_dual_polytope}. The Legendre transform of  $W_{\Sigma,\trop}$ is the support function corresponding to the ample divisor $\sum_{s=1}^{p-n}\lambda_{n+s}D_{n+s}$, so it is strictly convex according to Proposition \ref{prop_Fano-support}. So the subdivision $T_{W_{\Sigma,\trop}}$ of $P_V$ is given by connecting the origin to $v_j$, thus there is a unique compact component of $\R^n\backslash V(W_{\Sigma,\trop})$ which is given by $\{(x_1,x_2,\cdots,x_n)\in\R^n\,|\,-\lambda_{j}+\langle v_{j},x\rangle\leq 0 \text{ for }j=1,2,\dots,p\}$. So there is a unique compact component in $\R^n\backslash -V(W_{\Sigma,\trop})$ which is given by
\begin{equation}
\label{eq_G_trop}
    G_{\trop}:=\{x\in\R^n\ |\ \chi_j(x)\geq 0\},
\end{equation}
where we set
\begin{equation*}
    \chi_j(x_1,x_2,\dots,x_n)=
    \lambda_j+\langle v_j,x\rangle, j=1,2,\dots,p.
\end{equation*}
Note that $G_{\trop}$ is the polytope corresponding to the toric divisor $\sum_{s=1}^{p-n}\lambda_{n+s}D_{n+s}$, thus it has the same shape of $G_{\text{can}}$. We denote the facet of $G_{\trop}$ dual to $v_j$ by $F_{j,\trop}$.

Now to transfer $\beta_{\trop,C}$ to $\R^n\times\R_{\geq0}^2$, we identify the $\R^n$ where $G_\trop$ lies on with $\R^n\times(0,0)$. Note that $\beta_{\trop,C}$ is dual to $G=G_1+G_2+\cdots+G_{n-1}$, which also has the same shape as $G_{\text{can}}$. Thus we can transfer $\beta_{\trop,C}$ to a tropical curve in $\R^n\times(0,0)$ such that the unbounded edges of $\beta_{\trop,C}$ with direction $v_j$ are dual to $F_{j,\trop}$. We further require that the intersection points of $\beta_{\trop,C}$ and $F_{j,\trop}$ are contained in $U_j$ as in Lemma \ref{lemma_defomation_near_amoeba}, which can be guaranteed by shrinking the size of $\beta_{\trop,C}$. We remove the part of it which is outside of $G_{\trop}$ and then it is a tropical curve with ends on $\partial G_{\trop}$. We abuse the notation of $\beta_{\trop,C}$ to denote it.

\begin{remark}
\label{remark_trop_curve_under_moment_map}
The tropical curve $\beta_{\trop,C}$ with ends on $\partial G_{\trop}$ has the same shape as the skeleton of the image of the holomorphic curve $C$ under the moment map of $X_\Sigma$ with respect to the K\"{a}hler class $\sum_{s=1}^{p-n}\lambda_{n+s}[D_{n+s}]$.
\end{remark}

\begin{figure}[htbp!]
    \centering
    \includegraphics[scale=0.5]{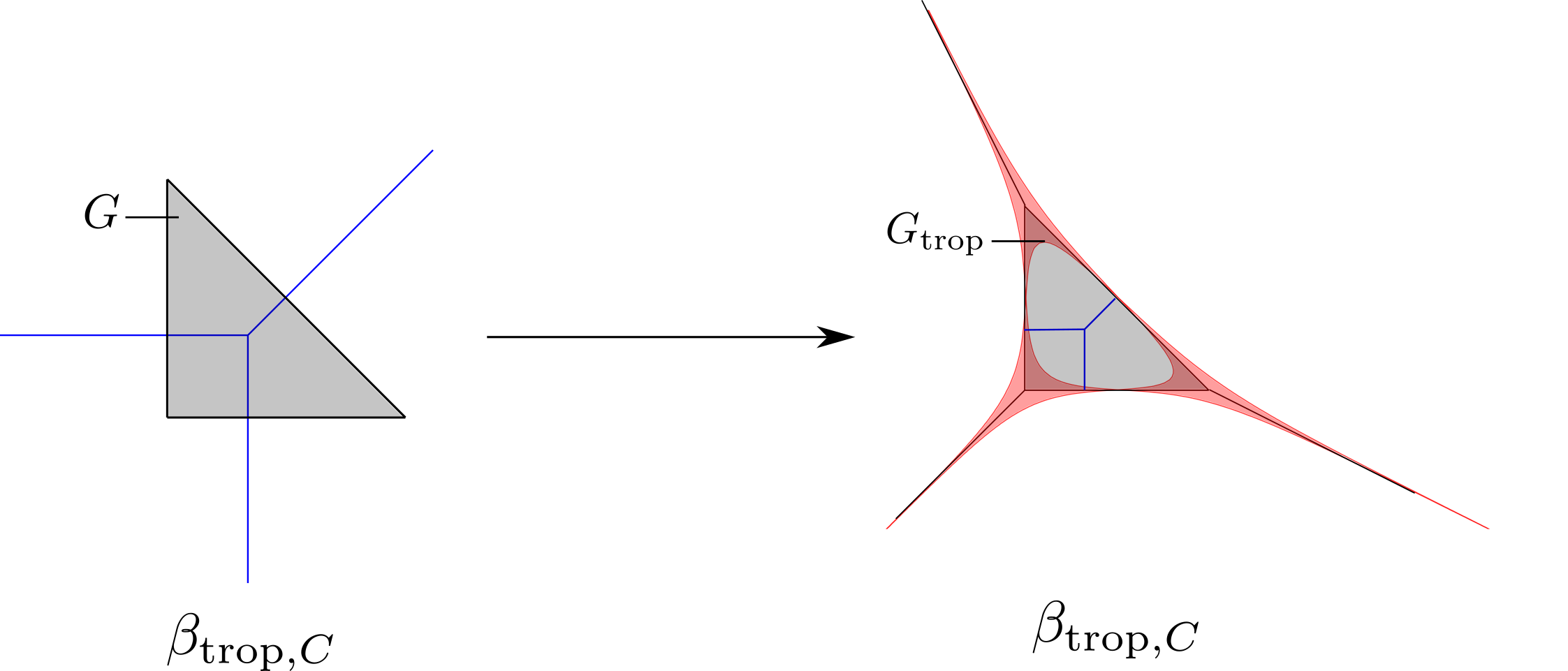}
    \caption{Transfer of a tropical curve}
    \label{fig_transfer_of_a_trop_curve}
\end{figure}

Now we show how to lift the tropical curve $\beta_{\trop,C}\subset \R^n\times\R_{\geq0}^2$ to a piecewise Lagrangian closed submanifoldin $Y^t$. The idea is that we decompose $\beta_{\trop,C}$ into different parts, lift them to different pieces and then glue them together. 

We decompose $\beta_{\trop,C}$ into three different types:
$$
    \beta_{\trop,C}=\bigcup\limits_{e\in\{\text{edges of }\beta_{\trop,C}\}}\beta_{e}\cup\bigcup\limits_{v\in\{\text{vertices of }\beta_{\trop,C}\}}\beta_{v}\cup\bigcup\limits_{\substack{j=1,\\ a\in\{\text{ends of }\beta_{\trop,C}\text{ at }F_{j,\trop}\}}}^p\beta_{a},
$$
where $\beta_e$ is the edge itself when $e$ is an interior edge and $e\backslash\beta_{a}$ when $e$ has an end $a$, $\beta_v$ is the vertex $v$ itself and $\beta_a$ is a segment of the edge with end $a\in F_{j,\trop}$ such that $\beta_a\subset U_j$ and the endpoints of $\beta_a$ are $a$ and a point in $U_j\backslash U_j'$, where $U_j,U_j'$ are given as in Lemma \ref{lemma_defomation_near_amoeba}.

We need some preparation before we do the lifting. 

\begin{definition}
\label{def_tropcurve_weight_balance}
The weight of a tropical curve $\beta$ is a map $w:\{\text{edges of } \beta\}\rightarrow\mathbb{N}$. The tropical curve $\beta$ is called \textit{balanced} if for each vertex $v$ we have that
$$
    \sum\limits_{e\text{ incident with }v}{\epsilon_{e,v}}w(e)\xi_e=0,
$$
where $\xi_e$ is the primitive tangent vector of $e$ and $\epsilon_{e,v}=+/-$ if $\xi_e$ points out of/towards $v$.
\end{definition}

\begin{definition}
\label{def_cycle_in_torus}
Suppose $\sigma$ is a $k-$dimensional integral polytope in $\R^n$, and $s(x)$ is a chosen point in $\Log_t^{-1}(x)$. Then the $k-$cycle $\sigma(s(x))$ in $\Log_t^{-1}(x)$ induced by $\sigma$ is the cycle  $\sigma\hookrightarrow\R^n\xrightarrow{\exp(2\pi\sqrt{-1}\cdot)} T^n\xrightarrow{\cdot s(x)}\Log_t^{-1}(x)$. 
\end{definition}

\begin{lemma}
\label{lemma_balance_of_trop_curve}
Suppose $e$ is an edge of $\beta_{\trop,C}$ which is dual to a cell $\sigma_e$ in the subdivision $T_f$. If we set the weight of $\beta_{\trop,C}$ as $w(e)=\vol_{\text{int}}(\sigma_e)$, then $\beta_{\trop,C}$ is balanced.
\end{lemma}

\begin{proof}
    Suppose $v$ is a vertex of $\beta_{\trop,C}$ and it is dual to a cell $\sigma_v$ of $T_f$, then $\sigma_v$ is of the form 
    $\sigma_v=\sigma_1+\sigma_2+\cdots+\sigma_{n-1}$ with one of $\sigma_k$'s a $2$-cell and the others are $1$-cells. Say $\sigma_1$ is a $2$-cell. Then
    $$
        \partial\sigma_v=(\partial\sigma_1+\sigma_2+\cdots+\sigma_{n-1})+(\sigma_1+\partial\sigma_2+\cdots+\sigma_{n-1})+\cdots+(\sigma_1+\sigma_2+\cdots+\partial\sigma_{n-1}).
    $$
    Take a point $s(v)\in\Log_t^{-1}(v)$, then we have that
    \begin{align*}
        0=&[\partial\sigma_v(s(v))]\\
        =&[(\partial\sigma_1+\sigma_2+\cdots+\sigma_{n-1})(s(v))]+\cdots+[(\sigma_1+\sigma_2+\cdots+\partial\sigma_{n-1})(s(v))]\\
        =&\sum\limits_{e\text{ incident with }v}[\sigma_e(s(v))]
    \end{align*}
    in $H_{n-1}(\Log_t^{-1}(v))$, where the third equation is because that $[(\partial\sigma_1+\sigma_2+\cdots+\sigma_{n-1})(s(v))]=\sum\limits_{e\text{ incident with }v}[\sigma_e(s(v))]$ and the other terms are $0$ since $\sigma_k$'s have integer endpoints. Now note that ${\epsilon_{e,v}}w(e)[\xi_e(s(v))]$ is the Poincar\'{e} dual of $[\sigma_e(s(v))]$, so we have that 
    $$
        \sum\limits_{e\text{ incident with }v}{\epsilon_{e,v}}w(e)[\xi_e(s(v))]=0 
    $$
    in $H_1(\Log_t^{-1}(v))$ and thus
    $$
        \sum\limits_{e\text{ incident with }v}{\epsilon_{e,v}}w(e)\xi_e=0.
    $$
\end{proof}

We need to deform $Y^t$ near the ends of $\beta_{\trop,C}$ at $F_{j,\trop}$ to a simpler model, for the purpose of the lifting near the ends. We have the following lemma.
\begin{lemma}
\label{lemma_defomation_near_amoeba}
    For sufficiently small $|t|$, there exist open sets $U_j''\subset\subset U_j'\subset\subset U_j$ in $\R^n$ which contain an open subset of $F_{j,\trop}$, and positive real numbers $R>R'>R''>\max\{\frac{1}{2}\max\limits_{z\in \Log_t^{-1}(U_j)}\{1+\sum\limits_{k=1}^p|t^{\lambda_k}z^{v_k}|,1\}$, such that there exists a symplectomorphism
    $$
        \Phi_{j}:Y^t_j\rightarrow Y^t_{j,\text{simple}},
    $$
    where 
    $$
        Y^t_j:=Y^t\cap\mu_0^{-1}([0,R]^2\times(\log t\cdot U_j)),
    $$ 
    $$
        Y^t_{j,\text{simple}}\cap\mu_0^{-1}([0,R'']^2\times(\log t\cdot U_j''))=(uv=1+t^{\lambda_j}z^{v_j})\cap\mu_0^{-1}([0,R'']^2\times(\log t\cdot U_j'')),
    $$ and the symplectic forms on $Y^t_j$ and $Y^t_{j,\text{simple}}$ are induced from the standard symplectic form $\omega_0$ on $\C^2\times(\C^*)^n$. Furthermore, $\Phi_{j}$ is an identity on $Y_j^t\backslash\mu_0^{-1}([0,R']^2\times(\log t\cdot U'_j))$ and for a point $x\in U_j\cap F_{j,\trop}$, $(x+\R\cdot v_j)\cap \Log_t(\{z\,|\,(0,0,z)\in Y^t_{j,\text{simple}}\})$ is connected.
\end{lemma}
We put the proof of this lemma in Appendix \ref{appx_proof_of_defomation_lemmma}. Note that for $(u,v,z)\in Y_{j,\simple}^t$, we have that $uv=W_{\Sigma,j}(u,v,z)$, where 
$$
    W_{\Sigma,j}(u,v,z)=W_{\Sigma}(z)-\tau(u,v)\rho_{j,t}(z)\sum_{k\neq j}t^{\lambda_k}z^{v_k}
$$ 
as in (\ref{eq_W_Sigma,j}). When $|u|<\sqrt{2R''}$ and $|v|<\sqrt{2R''}$, we have that $\tau(u,v)=1$, in this case, we use $W_{\Sigma,j}(z)$ or $W_{\Sigma,j}$ to denote $W_{\Sigma}(z)-\rho_{j,t}(z)(\sum_{k\neq j}t^{\lambda_k}z^{v_k})$ and regard it as a polynomial in $z$. The rescaled amobea $\mathscr{A}_t(W_{\Sigma,j})$ is then equal to $\Log_t(\{z\,|\,(0,0,z)\in Y^t_{j,\text{simple}}\})$. So the symplectomorphism $\Phi_j$ deforms the rescaled amoeba $\mathscr{A}_t(W_{\Sigma})$ to the rescaled amoeba $\mathscr{A}_t(W_{\Sigma,j})$, as shown in Figure \ref{fig_defomation_of_amobea}.
\begin{figure}[htbp!]
    \centering
    \includegraphics[scale=0.5]{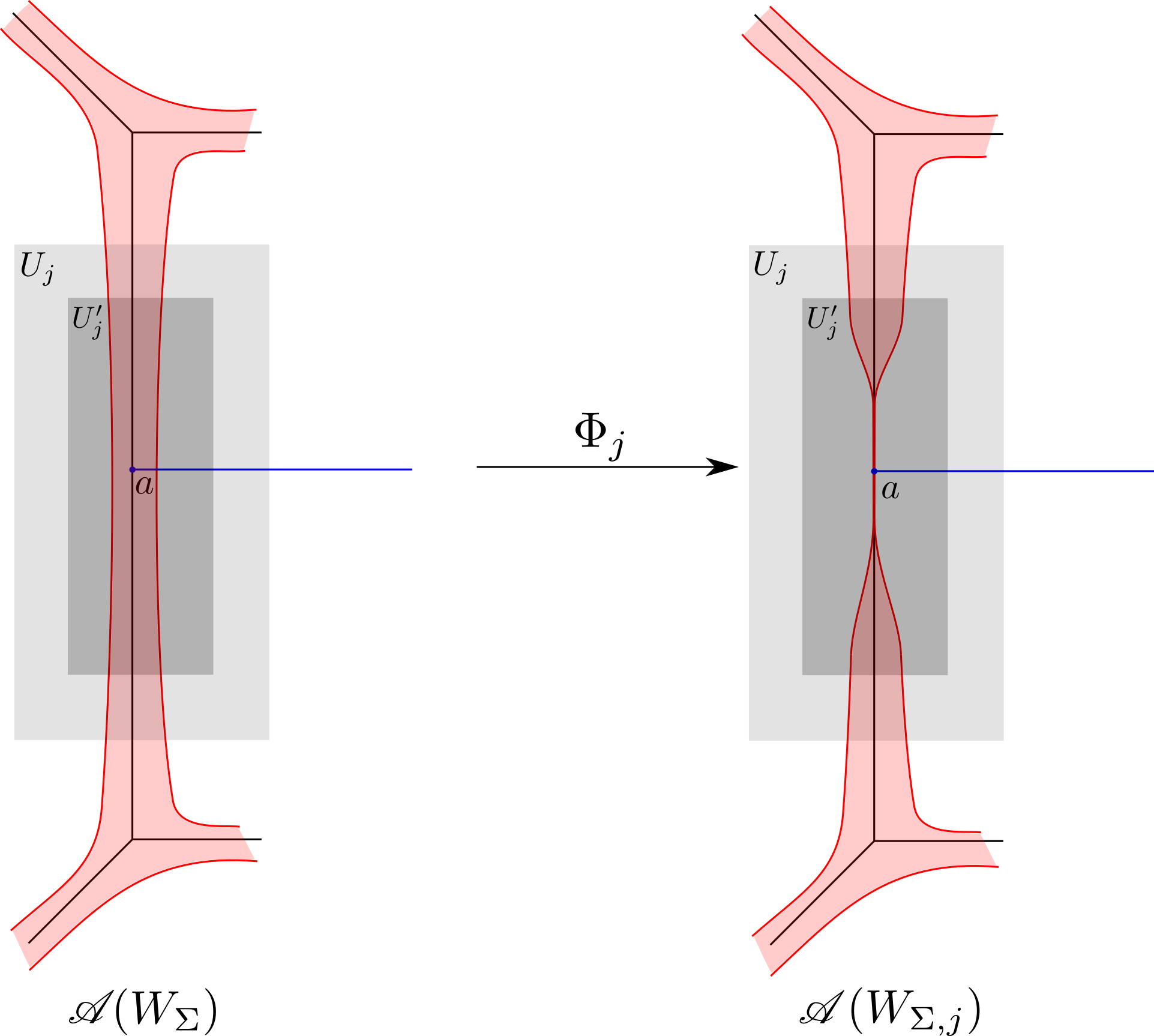}
    \caption{The deformation of the amoeba}
    \label{fig_defomation_of_amobea}
\end{figure}

Now let us do the lifting.

\begin{itemize}
    \item For $\beta_e$, suppose the edge $e$ is parametrized by $r_e(\tau)\in\R^n$, $\tau\in[0,1]$, its primitive tangent vector is $\xi_e$, which has the same direction as $r_e'(\tau)$, and it is dual to a cell $\sigma_e$ of $T_f$. Let $s_{\text{real}>0}(r_e(\tau))$ be the positive real locus of $\Log_t^{-1}(r_e(\tau))$, then the lifting $\Gamma_e$ is
    \begin{align*}
        \Gamma_e:=\bigcup\limits_{\tau\in[0,1]}\{&(z,u,v)\in Y^t\ |\ z\in \sigma_e(s_{\text{real}>0}(r_e(\tau))),|u|=|v|\}.       
    \end{align*}
    The orientation of $\Gamma_e$ is taken as $\frac{\partial}{\partial u}\wedge\frac{\partial}{\partial\tau}\wedge or(\sigma_e(s_{\text{real}>0}(r_e(\tau))))$, where $or(\sigma_e(s_{\text{real}>0}(r_e(\tau))))$ is determined through $\xi_e(s_{\text{real}>0}(r_e(\tau)))\wedge or(\sigma_e(s_{\text{real}>0}(r_e(\tau))))=or(\Log_t^{-1}(r_e(\tau)))$. Since $\frac{\partial}{\partial\tau}$ has the same direction as $\xi_e$, the orientation of $\Gamma_e$ is independent of the direction of $\xi_e$.
    \item For $\beta_v$, suppose the vertex $v$ is dual to a cell $\sigma_v$ of $T_f$, let $s_{\real>0}(v)$ be the positive real point of $\Log_t^{-1}(v)$, then the lifting $\Gamma_v$ is 
    \begin{align*}
        \Gamma_v:=\{&(z,u,v)\in Y^t\ |\ z\in \sigma_v(s_{\real>0}(v)), |u|=|v|\}.  
    \end{align*}
    The orientation of $\sigma_v(s_{\real>0}(v))$ is induced from the orientation of $\Log_t^{-1}(v)$. Note that $\partial\sigma_v(s_{\real>0}(v))=\bigcup\limits_{e\text{ incident with }v}\sigma_e(s_{\real>0}(v))$, thus $\partial\Gamma_v$ and $\partial\Gamma_e$ share a component, whose orientations induced from $\Gamma_v$ and $\Gamma_e$ are opposite. So $\Gamma_v$ glues the liftings of $\Gamma_e$'s for the incident edges together.
    \begin{figure}[htbp!]
        \centering
        \includegraphics[scale=0.5]{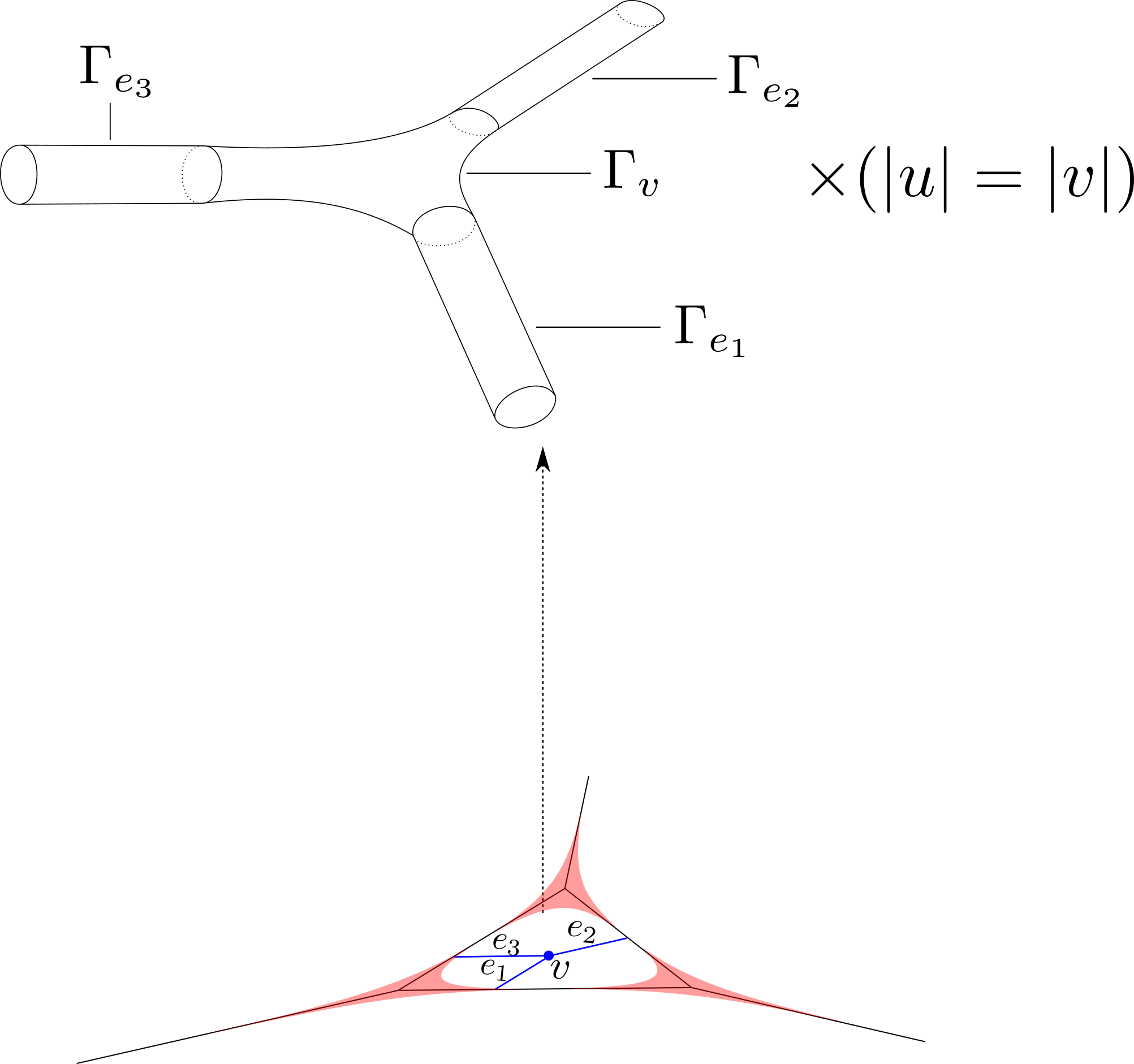}
        \caption{The lifting of $\Gamma_e$ and $\Gamma_v$}
        \label{fig_liftingofgamma_e_gamma_v}
    \end{figure}
    Figure \ref{fig_liftingofgamma_e_gamma_v} shows the construction of $\Gamma_e$ and $\Gamma_v$.
    \item For $\beta_a$, suppose it is parametrized by $r_a(\tau)\in\R^n,$ $\tau\in[0,1]$ with $r_a(1)=a$, the end of $\beta_{\trop,C}$ on $F_{j,\trop}$ and $r_a(0)\in U_j\backslash U_j'$. Its primitive tangent vector is $\xi_a=v_j$ and it is dual to a cell $\sigma_a$ of $T_f$. Choose a path $s(r_a(\tau))\in\Log_t^{-1}(r_a(\tau))$ such that $1+t^{\lambda_j}s(r_a(1))^{v_j}=0$, $s(r_a(0))=s_{\real>0}(r_a(0))$ and $\exp(\sqrt{-1}\arg((s(r_a(\tau)))^{v_j}))$ is a path in $S^1\backslash\{-\sqrt{-1}\}$ with counterclockwise direction. Then the lifting $\Gamma_a$ is
    \begin{align*}
        \Gamma_a:=\Phi_j^{*}(\bigcup\limits_{\tau\in[0,1]}\{&(z,u,v)\in Y^t_{j,\text{simple}}\ |\ z\in \sigma_a(s(r_a(\tau))), |u|=|v|\}).
    \end{align*}
    Note that $|u|<\sqrt{2R''},|v|<\sqrt{2R''}$ since $\beta_a\subset U_j$, so $\tau(u,v)\equiv 1$ and thus $|u|=|v|=|W_{\Sigma,j}(z)|^{1/2}$. So for $z\in \sigma_a(s(r_a(1)))$, we have $|u|=|v|=|1+t^{\lambda_j}z^{v_j}|^{1/2}=0$ since $1+t^{\lambda_j}s(r_a(1))^{v_j}=0$ and $\sigma_a$ is dual to $v_j$. So 
    $$
        \partial\Gamma_a=\Phi_j^*(\{(z,u,v)\in Y^t_{j,\text{simple}}\ |\ z\in \sigma_a(s(r_a(0))), |u|=|v|\}),
    $$
    which is the same as one component of $\partial\Gamma_e$ since $r_a(0)\in U_j\backslash U_j'$ and $\Phi_{j}$ is an identity on $Y_j^t\backslash\mu_0^{-1}([0,R']^2\times(\log t\cdot U'_j))$.
    The orientation of $\sigma_a(s(r_a(\tau)))$ is given in the same way as of $\sigma_e(s_{\real>0}(r_e(\tau)))$, so the orientations of $\partial\Gamma_a$ induced from $\Gamma_a$ and $\Gamma_e$ are opposite, thus $\Gamma_a$ caps $\Gamma_e$.
    \begin{figure}[htbp!]
        \centering
        \includegraphics[scale=0.5]{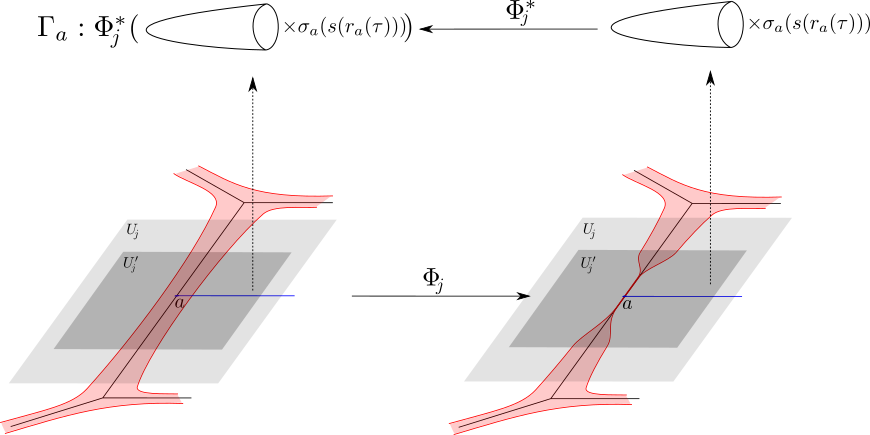}
        \caption{The lifting of $\Gamma_a$}
        \label{fig_liftingofgamma_a}
    \end{figure}
    Figure \ref{fig_liftingofgamma_a} shows the construction of $\Gamma_a$.
\end{itemize}

Now we show that these lifting pieces are Lagrangian.

\begin{proposition}
    \label{prop_Lagr_checking}
    The lifting pieces $\Gamma_e,\Gamma_v$ and $\Gamma_a$ are Lagrangian with respect to $\omega^t:=\omega_0|_{Y^t}$, i.e.,
    $$
        \omega_0|_{\Gamma_e}=0,\omega_0|_{\Gamma_v}=0,\omega_0|_{\Gamma_a}=0.
    $$
\end{proposition}

\begin{proof}
    For $\Gamma_e$, the $(n-1)-$cell $\sigma_e$ of $T_f$ dual to $e$ is of the form 
    $$
        \sigma_e=\sigma_1+\sigma_2+\cdots+\sigma_{n-1},
    $$
    where $\sigma_k$ is a $1-$cell of the subdivision $T_k$ of $G_k$. Then $\Gamma_e$ is parametrized by 
    \begin{align*}
        [0,1]\times[0,2\pi]^{n-1}\times[0,2\pi]\rightarrow&\Gamma_e\\
        (\tau,\eta_1,\dots,\eta_{n-1},\theta_0)\mapsto&(\exp(\sqrt{-1}(\eta_1\sigma_1+\eta_2\sigma_2+\cdots+\eta_{n-1}\sigma_{n-1}))\cdot s_{\real>0}(r_e(\tau)),\\
        &|W_{\Sigma}(\tau,\eta)|^\frac{1}{2}e^{\sqrt{-1}\theta_0},W_{\Sigma}(\tau,\eta)|W_{\Sigma}(\tau,\eta)|^{-\frac{1}{2}}e^{-\sqrt{-1}\theta_0}),
    \end{align*}
    where $\sigma_{k,i}$ is the $i$-th coordinate of $\sigma_k$ when we regard it as a vector, and $W_\Sigma(\tau,\eta)=W_\Sigma(\exp(\sqrt{-1}(\eta_1\sigma_1+\eta_2\sigma_2+\cdots+\eta_{n-1}\sigma_{n-1}))\cdot s_{\real>0}(r_e(\tau)))$. Thus
    \begin{align*}
        &\sum\limits_{i=1}^n\frac{dz_i\wedge d\bar{z}_i}{|z_i|^2}|_{\Gamma_e}\\
        =&\sum\limits_{i=1}^n\big(\frac{s_{\real>0,i}'(r_e(\tau))}{s_{\real>0}(r_e(\tau))}d\tau+\sum\limits_{k=1}^{n-1}\sqrt{-1}\sigma_{k,i}d\eta_k\big)\wedge\big(\frac{\overline{s_{\real>0,i}}'(r_e(\tau))}{\overline{s_{\real>0}}(r_e(\tau))}d\tau-\sum\limits_{k=1}^{n-1}\sqrt{-1}\sigma_{k,i}d\eta_k\big)\\
        =&-2\sqrt{-1}\sum\limits_{i=1}^{n}\sum\limits_{k=1}^{n-1}r_{e,i}'(\tau)\sigma_{k,i}d\tau\wedge d\eta_k\\
        =&-2\sqrt{-1}\sum\limits_{k=1}^{n-1}\partial\langle r_e'(\tau),\sigma_k\rangle d\tau\wedge d\eta_k\\
        =&0
    \end{align*}
    where the last equation is because that $e$ is dual to $\sigma_e$. We also have that
    \begin{align*}
        &(du\wedge d\bar{u}+dv\wedge d\bar{v})|_{\Gamma_e}\\
        =&d(|W_\Sigma(\tau,\eta)|^{\frac{1}{2}}e^{\sqrt{-1}\theta_0})\wedge d(|W_\Sigma(\tau,\eta)|^{\frac{1}{2}}e^{-\sqrt{-1}\theta_0})\\
        &+d(W_\Sigma(\tau,\eta)|W_\Sigma(\tau,\eta)|^{-\frac{1}{2}}e^{-\sqrt{-1}\theta_0})\wedge d(\overline{W_\Sigma}(\tau,\eta)|W_\Sigma(\tau,\eta)|^{-\frac{1}{2}}e^{\sqrt{-1}\theta_0})\\
        =&\sqrt{-1}d\theta_0\wedge d(|W_\Sigma(\tau,\eta)|^{\frac{1}{2}}|W_\Sigma(\tau,\eta)|^{\frac{1}{2}})-\sqrt{-1}d\theta_0\wedge d(\overline{W_\Sigma}(\tau,\eta)|W_\Sigma(\tau,\eta)|^{-\frac{1}{2}}W_\Sigma(\tau,\eta)|W_\Sigma(\tau,\eta)|^{-\frac{1}{2}})\\
        =&0.
    \end{align*}
    So $\omega_0|_{\Gamma_e}=0$.
    For $\Gamma_a$, since $\Phi_j$ is a symplectomorphism, $\omega_0|_{\Gamma_a}=0$ for the same reason as for $\Gamma_e$.
    For $\Gamma_v$, $\omega_0|_{\Gamma_v}=0$ since $\omega_0|_{\Log_t^{-1}(v)\times(|u|=|v|)}=0$.
\end{proof}

Now we can give the Lagrangian lifting of $\beta_{\trop,C}$.

\begin{definition}
\label{def_lagr_lifting}
The \textit{Lagrangian lifting} of the tropical curve $\beta_{\trop,C}$ is the piecewise Lagrangian closed submanifold
\begin{align}
    L_{\trop,C}^t=\bigcup\limits_{e\in\{\text{edges of }\beta_{\trop,C}\}}\Gamma_{e}\cup\bigcup\limits_{v\in\{\text{vertices of }\beta_{\trop,C}\}}\Gamma_{v}\cup\bigcup\limits_{\substack{j=1,\\ a\in\{\text{ends of }\beta_{\trop,C}\text{ at }F_{j,\trop}\}}}^p\Gamma_{a}
\end{align}
in $Y^t$.
\end{definition}

\subsubsection{Evaluation of the central charge}
\label{4.3.2}
In this subsubsection we evaluate the central charge $C_t(L_{\trop,C}^t)$ as defined in Definition \ref{def_centralcharge_period}. We first need to check that the form
$$
    \Omega=d\log u\wedge d\log z_1\wedge\cdots\wedge d\log z_n
$$
is well-defined on $Y^t$.

\begin{lemma}
\label{lemma_holoform_well_def}
    The form $\Omega$ is a well-defined closed form on $Y^t$ when $|t|$ is sufficiently small.
\end{lemma}

\begin{proof}
It is sufficient to show that it is well-defined on $u=v=0$. Since $W_{\Sigma}^{-1}(0)$ is smooth when $|t|$ is sufficiently small, $(\frac{\partial W_{\Sigma}(z)}{\partial z_1},\frac{\partial W_{\Sigma}(z)}{\partial z_2},\dots,\frac{\partial W_{\Sigma}(z)}{\partial z_n})$ is nonzero at $W_{\Sigma}(z)=0$ when $|t|$ is sufficiently small. Say $\frac{\partial W_{\Sigma}(z)}{\partial z_1}\neq0$ at $u=v=0$, then since
$$
    udv+vdu=\sum_{i=1}^n\frac{\partial W_{\Sigma}(z)}{\partial z_i}dz_i,
$$
on $Y^t$, we can replace $d\log z_1$ in $\Omega$ by $u(\frac{\partial W_{\Sigma}(z)}{\partial z_1})^{-1}z_1^{-1}dv$ and thus 
$$
    \Omega|_{Y^t}=(\frac{\partial W_{\Sigma}(z)}{\partial z_1})^{-1}z_1^{-1}dudvd\log z_2\cdots\log z_n,
$$
which is well-defined on $u=v=0$. It is also closed on $Y^t$ since the coefficients are holomorphic.
\end{proof}

\begin{lemma}
\label{lemma_match_of_forms}
The $(n+1)$-form $\Omega|_{Y^t}$ is the same as the holomorphic form $\Omega_{Y^t}$ as in Definition \ref{def_centralcharge_period}.
\end{lemma}

\begin{proof}
    Set $q=uv-W_{\Sigma}(z)$, then we have that
    \begin{align*}
        \Omega_{Y^t}=&\Residue_{(q=0)}(\frac{1}{q}du\frac{dq}{u}\frac{dz_1}{z_1}\cdots\frac{dz_n}{z_n})\\
        =&d\log ud\log z_1\cdots\log z_n\\
        =&\Omega|_{Y^t}.
    \end{align*}
\end{proof}

Now we want to evaluate the period integrals, i.e., the integrals of $\Omega$ on $\Gamma_e$, $\Gamma_v$ and $\Gamma_a$. To do this, we need the following lemma.

\begin{lemma}
\label{lemma_periodevaluation1}
Suppose $\xi$ is a segment in $\R^n$ with integral endpoints and $[\xi]\in H_1(T^n)$ is induced through $\xi\hookrightarrow\R^n\xrightarrow{\exp(2\pi\sqrt{-1}\cdot)}T^n$. Let $[\sigma_\xi]\in H_{n-1}(T^n)$ be the Poincar\'{e} dual of $[\xi]$, then
$$
    \int_{[\sigma_\xi]}d\theta_1d\theta_2\cdots \widehat{d\theta_j}\cdots d\theta_n=(-1)^{i-1}(2\pi)^{n-1}\xi_i,
$$
where $\xi_i$ is the $i$-th coordinate of $\xi$ when we regard it as a vector, and we parametrize $T^n$ as $T^n=\{(e^{\sqrt{-1}\theta_1},\cdots,e^{\sqrt{-1}\theta_n})\ |\ (\theta_1,\cdots,\theta_n)\in[0,2\pi]^n\}$.
\end{lemma}

\begin{proof}
    Let $e_k=(0,\cdots,0,1,0,\cdots,0)$ be the segment in $\R^n$ with $1$ in the $k$-th position and $[\sigma_{e_k}]\in H_{n-1}(T^n)$ be the Poincar\'{e} dual of $[e_k]$, then 
    $$
        \int_{[\sigma_{e_k}]}d\theta_1d\theta_2\cdots \widehat{d\theta_i}\cdots d\theta_n=(-1)^{i-1}(2\pi)^{n-1}\delta_{ik}.
    $$
    Thus, $[\sigma_\xi]=\sum\limits_{k=1}^n\xi_k[\sigma_{e_k}]$ and we have that
    $$
        \int_{[\sigma_\xi]}d\theta_1d\theta_2\cdots \widehat{d\theta_i}\cdots d\theta_n=(-1)^{i-1}(2\pi)^{n-1}\xi_i.
    $$
\end{proof}

Now let us do the computation. We parametrize $(\C^*)^n$ by 
$$
(\C^*)^n=\{(e^{r_1+\sqrt{-1}\theta_1},\cdots,e^{r_n+\sqrt{-1}\theta_n})\ |\ (r_1,\cdots,r_n)\in\R^n,(\theta_1,\cdots,\theta_n)\in[0,2\pi]^n\}
$$ and $|u|=|v|$ by 
$$
(u,v)=(|W_{\Sigma}(z)|^{\frac{1}{2}}e^{\sqrt{-1}\theta_0},W_{\Sigma}(z)|W_{\Sigma}(z)|^{-\frac{1}{2}}e^{-\sqrt{-1}\theta_0}),\theta_0\in[0,2\pi].
$$
\begin{itemize}
    \item For $\Gamma_e$,
    \begin{equation}
    \begin{aligned}
        &\int_{\Gamma_e}\Omega\\ 
        =&\int_{\Gamma_e}d(\log(|W_{\Sigma}(z)|^{\frac{1}{2}}e^{\sqrt{-1}\theta_0}))(dr_1+\sqrt{-1}d\theta_1)\cdots(dr_n+\sqrt{-1}d\theta_n)\\ 
        =&\int_{\tau\in[0,1]}\int_{\theta_0\in[0,2\pi]}\int_{[\sigma_e(s_{\real>0}(\tau))]}(\sqrt{-1}d\theta_0)(\log tr_{e,1}'(\tau)d\tau+\sqrt{-1}d\theta_1)\cdots(\log tr_{e,n}'(\tau)d\tau+\sqrt{-1}d\theta_n)\\ 
        =&2\pi(\sqrt{-1})^n\log t\int_{\tau\in[0,1]}\int_{[\sigma_e(s_{\real>0}(\tau))]}\sum_{i=1}^n(-1)^{i-1}r_{e,i}'(\tau)d\tau d\theta_1d\theta_2\cdots \widehat{d\theta_i}\cdots d\theta_n\\ 
        =&(2\pi\sqrt{-1})^n\log t w(e) \langle r_e(1)-r_e(0),\xi_e \rangle,
    \end{aligned}
    \label{eq_integral_on_Gamma_e}
    \end{equation}
    where we use the fact that $[\sigma_e(s_{\real>0}(\tau))]$ is the Poincar\'{e} dual of $w(e)[\xi_e(s_{\real>0}(\tau))]$ and Lemma \ref{lemma_periodevaluation1} for the last equation.
    \item For $\Gamma_v$,
    \begin{equation}
    \label{eq_integral_on_Gamma_v}
    \begin{aligned}
        &\int_{\Gamma_v}\Omega\\ 
        =&\int_{\Gamma_v}d(\log(|W_{\Sigma}(z)|^{\frac{1}{2}}e^{\sqrt{-1}\theta_0}))(dr_1+\sqrt{-1}d\theta_1)\cdots(dr_n+\sqrt{-1}d\theta_n)\\ 
        =&\int_{\Gamma_v}(\sqrt{-1})^{n+1}d\theta_0d\theta_1\cdots d\theta_n\\ 
        =&(2\pi)(\sqrt{-1})^{n+1}\int_{\sigma_v(s_{\real>0}(v))}d\theta_1\cdots d\theta_n\\ 
        =&(2\pi\sqrt{-1})^{n+1}\vol(\sigma_v).
    \end{aligned}
    \end{equation}
    \item For $\Gamma_a$ with $a\in F_{j,\trop}$, it is hard to do the computation directly due to the symplectomorphism $\Phi_j$. To resolve it, we add a cycle $\Gamma_\text{add}$, whose period integral can be computed, to $\Gamma_a$, and isotope $\Gamma_{a,\text{total}}:=\Gamma_a\cup\Gamma_{\text{add}}$ to a \textit{singular fiber} within $Y_j^t$, whose period integral can also be computed. Then $\int_{\Gamma_a}\Omega$ is the difference of these two integrals, i.e.,
    \begin{equation}
    \label{eq_integral_on_Gamma_a_idea}
        \int_{\Gamma_a}\Omega       =\int_{\Gamma_{a,\text{total}}}\Omega-\int_{\Gamma_\add}\Omega.
    \end{equation}
\end{itemize}

We show the details of the construction of $\Gamma_{\text{add}}$. It consists of 7 pieces,
$$
    \Gamma_{\add}:=\Gamma^{+,\text{part}1}\cup\Gamma_{av}\cup\Gamma^{+,\text{part}2}\cup\Gamma^-\cup\Gamma_{\text{in}}\cup\Gamma_{\text{out}}\cup\Gamma_{\text{co}}.
$$
Suppose $a$ is an end at $F_{j,\trop}$ of an edge $e$, then the primitive tangent vector $\xi_a$ of $e$ is equal to $v_j$. Suppose the cell $\sigma_e$ dual to $e$ is of the form $\sigma_e=\xi_a^{\perp,1}+\xi_a^{\perp,2}+\cdots+\xi_a^{\perp,n-1}$, then let us set

\begin{itemize}
    \item $\gamma_a$: An integral vector such that $\langle \gamma_a,v_j \rangle=1$.
    \item $\tilde{\sigma}_a:=\sigma_e+\gamma_a$: An $n$-cycle in $\R^n$ which is the Minkowski sum of $\sigma_e$ and $\gamma_a$. Note that $\vol(\tilde{\sigma}_a)=w(e)$.
    \item A path $r_{\text{add}}(\tau)\subset U_j\backslash U_j'$, $\tau\in[0,1]$, and $|t|$ sufficiently small, such that $1>|t|^{\chi_j(r_{\text{add}}(0))}+(p-1)\max_{k\neq j}\{|t|^{c_k}\}$, $r_{\text{add}}(\frac{1}{2})=r_a(0)$ and $|t|^{\chi_j(r_{\text{add}}(1))}>1+(p-1)\max_{k\neq j}\{|t|^{c_k}\}$, where $c_k$'s are given in (\ref{eq_U_j}).
    \item $\tilde{\sigma}_a(\tau)$: The $(n+1)$-cycle $\tilde{\sigma}_a(s_{\real>0}(r_{\add}(\tau)))$ with orientation induced from $\Log_t^{-1}(r_{\add}(\tau))$.
    \item $r_0^+(\tau), r_0^-(\tau)\subset(0,+\infty)$, $\tau\in[0,1]$, such that $2\sqrt{r_0^+(0)r^-_0(0)}=|W_{\Sigma}(s_{\real>0}(r_{\add}(0)))|$ and $2\sqrt{r_0^+(1)r^-_0(1)}=|W_{\Sigma}(s_{\real>0}(r_{\add}(1)))|$.
\end{itemize}

\begin{figure}[htbp!]
    \centering
    \includegraphics[scale=0.5]{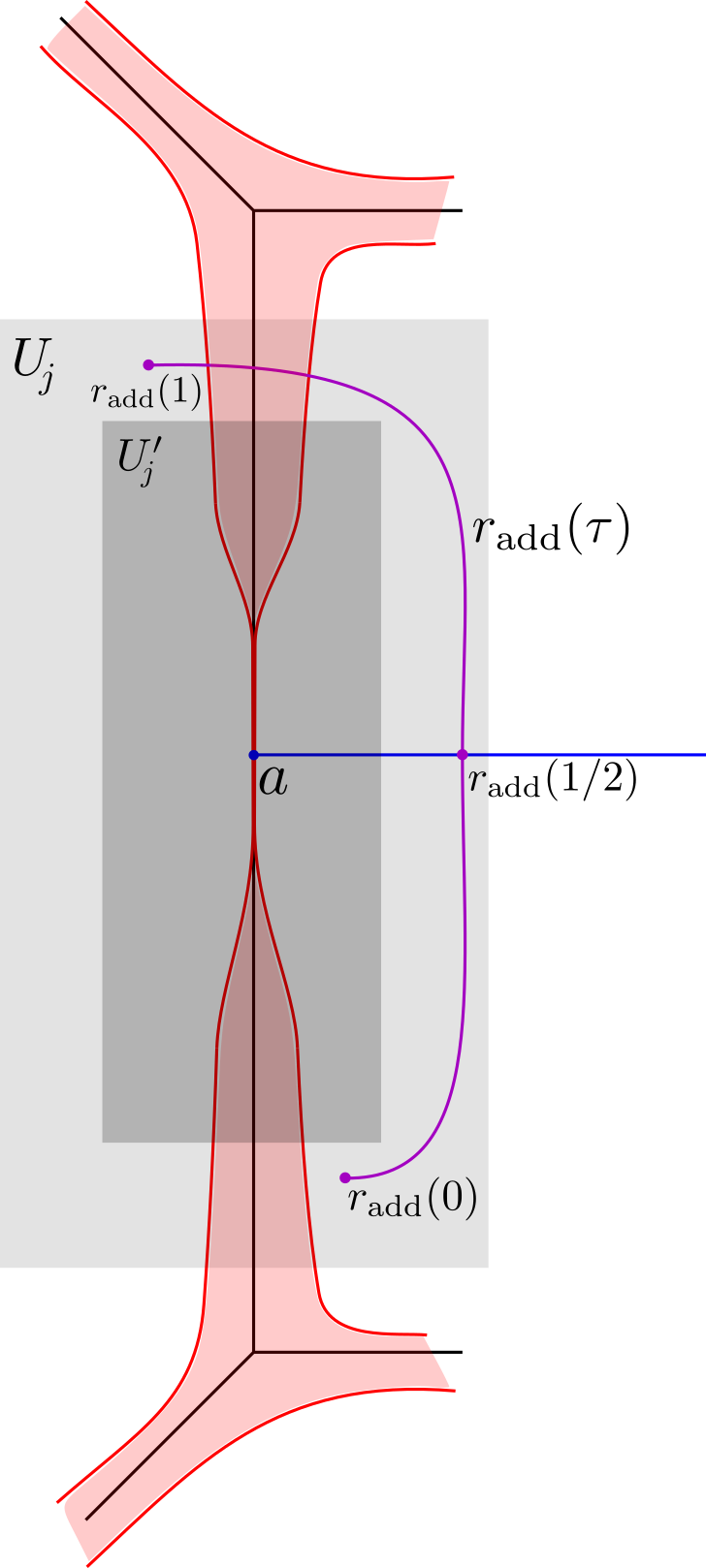}
    \caption{The path $r_{\add}(\tau)$ near the end}
    \label{fig_addin_to_the_amoeba}
\end{figure}

All the pieces of $\Gamma_{\add}$ will be constructed within $Y^t$ and then glued to $\Gamma_a$, they are given as follows:
\begin{itemize}
    \item For $\Gamma^{+,\text{part}1}$, it is given as
    \begin{align*}
        \Gamma^{+,\text{part}1}:=\{(u,v,z)\in Y^t\ |\ &z\in\tilde{\sigma}_a(\tau), u=\sqrt{2r_0^+(\tau)},\tau\in[0,1/2]\},
    \end{align*}
    with orientation $\frac{\partial}{\partial\tau}\wedge or(\tilde{\sigma}_a(\tau))$.
    
    \item For $\Gamma_{av}$, it is given as 
    \begin{align*}
        \Gamma_{av}:=\{(u,v,z)\in Y^t\ |\ &z\in\tilde{\sigma}_a(1/2), 
        u=\exp\big(\sqrt{-1}\big(s\arg(z^{v_j})\big)\big)\sqrt{2r_0^+(1/2)}, s\in[0,1]\},
    \end{align*}
    with orientation $\frac{\partial}{\partial s}\wedge or(\tilde{\sigma}_a(\tau))$.

    \item For $\Gamma^{+,\text{part}2}$, it is given as
    \begin{align*}
        \Gamma^{+,\text{part}2}:=\{(u,v,z)\in Y^t\ |\ z\in\tilde{\sigma}_a(\tau),u=\exp(\sqrt{-1}\arg(z^{v_j}))\sqrt{2r_0^+(\tau)}, \tau\in[1/2,1]\},
    \end{align*}
    with orientation $\frac{\partial}{\partial\tau}\wedge or(\tilde{\sigma}_a(\tau))$.
    
    \begin{figure}[htbp!]
        \centering
        \includegraphics[scale=0.5]{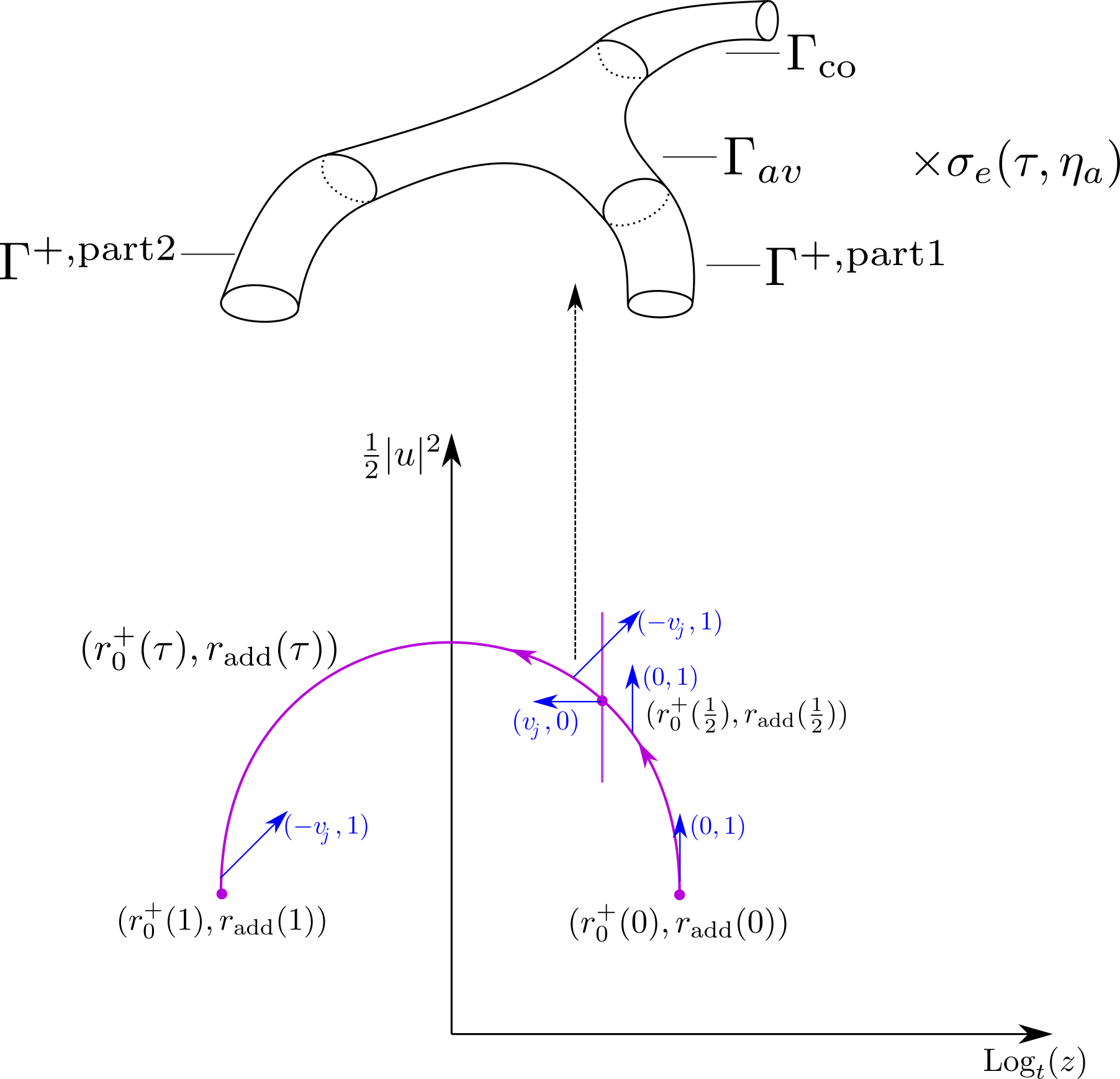}
        \caption{The lifting of $\Gamma^+$}
        \label{fig_lifting_of_Gamma+}
    \end{figure}
    
    \item For $\Gamma^{-}$, it is given as
    \begin{align*}
        \Gamma^{-}:=\{(u,v,z)\in Y^t\ |\ z\in\tilde{\sigma}_a(\tau), v=\sqrt{2r_0^-(\tau)}, \tau\in[0,1]\},
    \end{align*}
    with orientation $-\frac{\partial}{\partial\tau}\wedge or(\tilde{\sigma}_a(\tau))$.

    \begin{figure}[htbp!]
        \centering
        \includegraphics[scale=0.5]{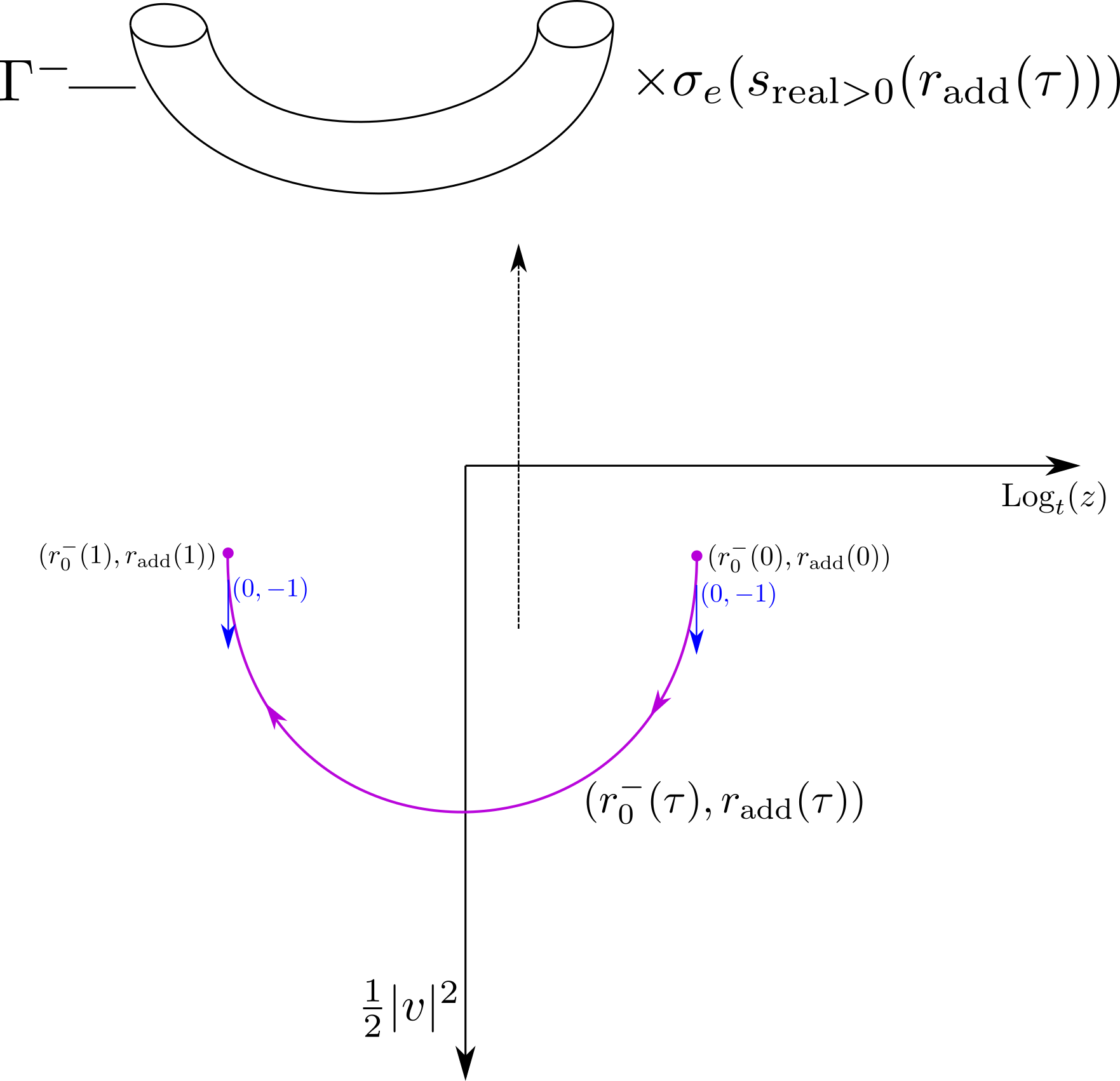}
        \caption{The lifting of $\Gamma^-$}
        \label{fig_lifting_of_Gamma-}
    \end{figure}
    
    \item For $\Gamma_{\text{out}}$, it is given as 
    \begin{align*}
        \Gamma_{\text{out}}:=\{(u,v,z)\in Y^t\ |\ &z\in\tilde{\sigma}_a(1),\\
        &v=\exp(\sqrt{-1}(\arg(\frac{\mu W_{\Sigma}(z)+(1-\mu)t^{\lambda_j}z^{v_j})}{z^{v_j}}))|\frac{W_{\Sigma}(z)}{W_{\Sigma}(s_{\real>0}(r_{\add}(1)))}|^{\mu}\sqrt{2r_0^-(1)},\\
        &\mu\in[0,1]\},
    \end{align*}
    with orientation $-\frac{\partial}{\partial\mu}\wedge or(\tilde{\sigma}_a(\tau))$.
    It is well-defined since 
    \begin{align*}
        &|\mu W_{\Sigma}(z)+(1-\mu)t^{\lambda_j}z^{v_j}|\\
        \geq&|t^{\chi_j(r_{\add}(1))}|-\mu(1+\sum_{k\neq j}|t^{\chi_k(r_{\add}(1))}|)\\
        >&|t|^{\chi_j(r_{\text{add}}(1))}-1-(p-1)\max_{k\neq j}\{|t|^{c_k}\}\\
        >&0.
    \end{align*}
    When $\mu=0$, $v=\sqrt{2r^-_0(1)}$, which is a component of $\partial\Gamma^-$. When $\mu=1$, we have that $u=\exp(\sqrt{-1}\arg(z^{v_j}))\frac{|W_{\Sigma}(s_{\real>0}(r_{\add}(1)))|}{\sqrt{2r_0^-(1)}}=\exp(\sqrt{-1}\arg(z^{v_j}))\sqrt{2r^+_0(1)}$, which is a component of $\partial\Gamma^{+,\text{part}2}$. So $\Gamma_{\text{out}}$ glues $\Gamma^{+,\text{part}2}$ and $\Gamma^-$.
    
    \item For $\Gamma_{\text{in}}$, it is given as 
    \begin{align*}
        \Gamma_{\text{in}}:=\{(u,v,z)\in Y^t\ |\ &z\in\tilde{\sigma}_a(0),\\
        &u=\exp(\sqrt{-1}\arg(\mu W_{\Sigma}(z)+(1-\mu)))|\frac{W_{\Sigma}(z)}{W_{\Sigma}(s_{\real>0}(r_{\add}(0)))}|^{\mu}\sqrt{2r_0^+(0)}, \mu\in[0,1]\},
    \end{align*}
    with orientation $-\frac{\partial}{\partial\mu}\wedge or(\tilde{\sigma}_a(0))$.
    It is well-defined since
    \begin{align*}
        &|\mu W_{\Sigma}(z)+(1-\mu))|\\
        \geq&1-\mu(|t^{\chi_j(r_{\add}(0))}|+\sum_{k\neq j}|t^{\chi_k(r_{\add}(0))}|)\\
        >&1-|t|^{\chi_j(r_{\text{add}}(0))}-(p-1)\max_{k\neq j}\{|t|^{c_k}\}\\
        >&0.
    \end{align*}
    When $\mu=0$, we have that $u=\sqrt{2r^+_0(0)}$, which is a component of $\partial\Gamma^{+,\text{part} 1}$. When $\mu=1$, we have that $v=\frac{|W_{\Sigma}(s_{\real>0}(r_{\add}(0)))|}{\sqrt{2r_0^+(0)}}=\sqrt{2r^{-}_0(0)}$, which is a component of $\partial\Gamma^-$. So $\Gamma_{\text{in}}$ glues $\Gamma^{+,\text{part}1}$ and $\Gamma^-$.
    \begin{figure}[htbp!]
        \centering
        \includegraphics[scale=0.7]{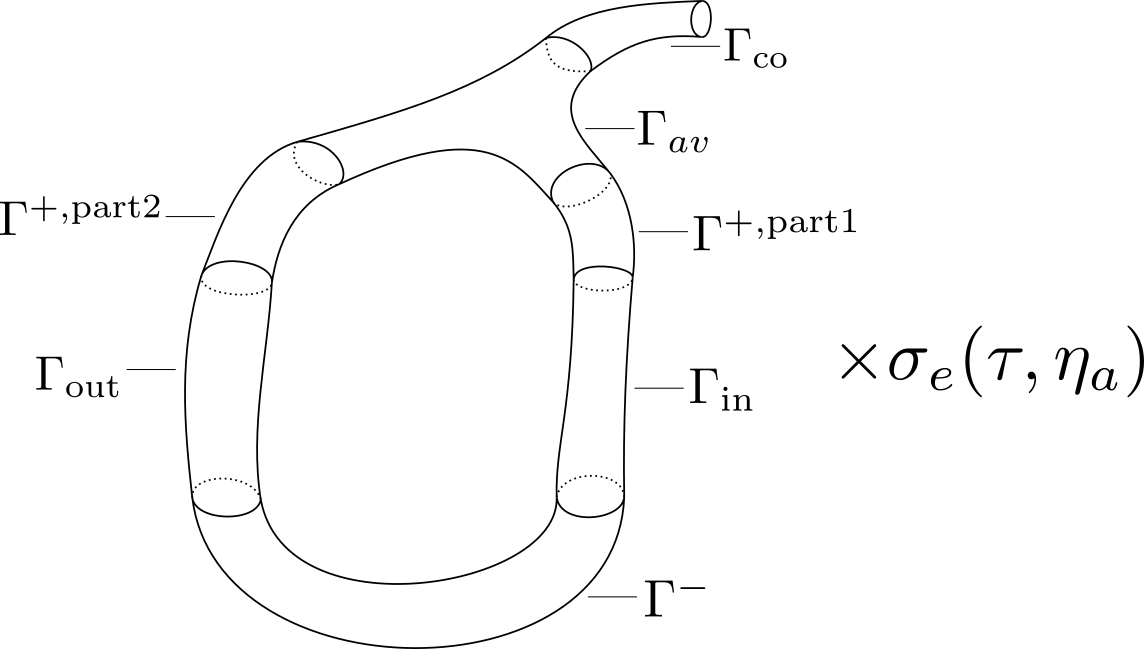}
        \caption{Gluing of $\Gamma^+$ and $\Gamma^-$}
        \label{fig_gluing_of_gamma+_gamma-}
    \end{figure}
    \item For $\Gamma_{\text{co}}$, it is given by
    \begin{align*}
        \Gamma_{\text{co}}:=\{(u,v,z)\in Y^t\ |\ &z\in\sigma_e(s_{\real>0}(r_{\add}(1/2))),\\
        &
        u=\big(\mu|W_{\Sigma}(z)|^{1/2}+((1-\mu)\sqrt{2r^+(1/2)})\big)e^{\sqrt{-1}\theta_0}, \theta_0\in[0,2\pi], \mu\in[0,1]\},
    \end{align*}    
    with orientation $\frac{\partial}{\partial\mu}\wedge\frac{\partial}{\partial\theta_0}\wedge or(\sigma_e(s_{\real>0}(r_{\add}(1/2))))$.
    When $\mu=0$, we have that $u=\sqrt{2r^+(1/2)}e^{\sqrt{-1}\theta_0}$, which is a component of $\partial\Gamma_{av}$. When $\mu=1$, we have that $u=|W_{\Sigma}(z)|^{1/2}e^{\sqrt{-1}\theta_0}$, which is a component of $\partial\Gamma_a$. So it glues $\Gamma_a$ and $\Gamma_{av}$.
\end{itemize}

We now give the integrals of $\Omega$ on the pieces. We parameterize $\tilde{\sigma}_a$ as $\tilde{\sigma}_a=\sigma_e+\eta_a\gamma_a$, $\eta_a\in[0,1]$. Thus $or(\tilde{\sigma}_a)=\gamma_a\wedge or(\sigma_e)$. We denote $\tilde{\sigma}_a(s_{\real>0}(r_{\add}(\tau)))$ and $\sigma_e(e^{2\pi\sqrt{-1}\eta_a\gamma_a}\cdot s_{\real>0}(r_{\add}(\tau)))$ by $\tilde{\sigma}_a(\tau)$ and $\sigma_e(\tau,\eta_a)$ respectively. Then 

\begin{align*}
    \int_{\Gamma^{+,\text{part}1}}\Omega=& (\sqrt{-1})^n \int_{\tau\in[0,\frac{1}{2}]}\int_{[\tilde{\sigma}_a(\tau)]}\frac{\partial\log(\sqrt{2r_0^+(\tau)})}{\partial \tau}d\tau d\theta_1\cdots d\theta_n\\
    =&(2\pi\sqrt{-1})^nw(e)\big(\log(\sqrt{2r^+_0(1/2)}-\log(\sqrt{2r^+_0(0)}))\big).
\end{align*}

\begin{align*}
    \int_{\Gamma_{av}}\Omega=&(\sqrt{-1})^{n+1}\int_{s\in[0,1]}\int_{[\tilde{\sigma}_a(1/2)]}\langle v_j,\theta\rangle ds d\theta_1\cdots d\theta_n\\
        =&(\sqrt{-1})^{n+1}\int_{s\in[0,1]}\int_{\eta_a\in[0,1]}\int_{[\sigma_e(1/2,\eta_a)]}2\pi\langle v_j,\gamma_a\rangle\eta_a\sum_{i=1}^n (-1)^{i-1}\gamma_{a,i} dsd\eta_ad\theta_1\cdots \widehat{d\theta_i}\cdots d\theta_n\\
        =&\frac{1}{2}2\pi(\sqrt{-1})^{n+1}\int_{[\sigma_e(1/2,\eta_a)]}\sum_{i=1}^n (-1)^{i-1}\gamma_{a,i}d\theta_1\cdots \widehat{d\theta_i}\cdots d\theta_n\\
        =&\frac{1}{2}2\pi(\sqrt{-1})^{n+1}w(e)\langle\gamma_a,v_j\rangle\\
        =&\frac{1}{2}(2\pi\sqrt{-1})^{n+1}w(e).
\end{align*}

\begin{align*}
    \int_{\Gamma^{+,\text{part}2}}\Omega=&\int_{\tau\in[\frac{1}{2},1]}\int_{\eta_a\in[0,1]}\int_{[\sigma_e(\tau,\eta_a)]}(\frac{\partial\log(\sqrt{2r_0^+(\tau)})}{\partial\tau}d\tau+2\pi\sqrt{-1}d\eta_a)\\
    &(\log t r_{\add,1}'(\tau)d\tau+\sqrt{-1}d\theta_1)\cdots(\log t r_{\add,n}'(\tau)d\tau+\sqrt{-1}d\theta_n)\\
    =&(\sqrt{-1})^n\big(\int_{\tau\in[\frac{1}{2},1]}\int_{[\tilde{\sigma}_a(\tau)]}\frac{\partial\log(\sqrt{2r_0^+(\tau)})}{\partial \tau}d\tau d\theta_1\cdots d\theta_n\\
    &+2\pi\sum\limits_{i=1}^{n}\int_{\tau\in[\frac{1}{2},1]}\int_{\eta_a\in[0,1]}\int_{[\sigma_e(\tau,\eta_a)]}(-1)^{i}\log t r'_{\add,i}(\tau)d\tau d\eta_ad\theta_1\cdots \widehat{d\theta_i}\cdots d\theta_n\big)\\
    =&(2\pi\sqrt{-1})^nw(e)\big(\log(\sqrt{2r^+_0(1)})-\log(\sqrt{2r^+_0(1/2)})-\log t \langle r_{\add}(1)-r_{\add}(1/2),v_j\rangle\big).
\end{align*}

\begin{align*}
    \int_{\Gamma^{-}}\Omega=&\int_{\Gamma^{-}}-d\log vd\log z_1\cdots d\log z_n\\
    =&(\sqrt{-1})^n \int_{\tau\in[0,1]}\int_{[-\tilde{\sigma}_a(\tau)]}-\frac{\partial\log(\sqrt{2r_0^-(\tau)})}{\partial \tau}d\tau d\theta_1\cdots d\theta_n\\
    =&(2\pi\sqrt{-1})^nw(e)\big(\log(\sqrt{2r^-_0(1)})-\log(\sqrt{2r^-_0(0)})\big).
\end{align*}

\begin{align*}
    \int_{\Gamma_{\text{out}}}\Omega=&(\sqrt{-1})^n \int_{\mu\in[0,1]}\int_{-[\tilde{\sigma}_a(1)]}-d\log vd\theta_1d\theta_2\cdots d\theta_n\\
    =&(\sqrt{-1})^{n}\int_{[\tilde{\sigma}_a(1)]}\log((t^{\lambda_j}z^{v_j})^{-1}W_{\Sigma}(z))d\theta_1d\theta_2\cdots d\theta_n\\
    &+(\sqrt{-1})^n\int_{[\tilde{\sigma}_a(1)]}\big(\log|t^{\lambda_j}z^{v_j}|-\log|W_{\Sigma}(s_{\real>0}(r_{\add}(1)))|\big)d\theta_1d\theta_2\cdots d\theta_n\\
    =&(\sqrt{-1})^{n}\int_{[\tilde{\sigma}_a(1)]}\log((t^{\lambda_j}z^{v_j})^{-1}W_{\Sigma}(z))d\theta_1d\theta_2\cdots d\theta_n\\
    &+(2\pi\sqrt{-1})^nw(e)(\log t^{\lambda_j}+\log t \langle r_{\add}(1),v_j\rangle -\log\sqrt{2r_0^+(1)}-\log\sqrt{2r_0^-(1)}.
\end{align*}
    
\begin{align*}
    \int_{\Gamma_{\text{in}}}\Omega=&(\sqrt{-1})^n\int_{\mu\in[0,1]}\int_{-[\tilde{\sigma}_a(0)]}d\log ud\theta_1d\theta_2\cdots d\theta_n\\
    =&(\sqrt{-1})^{n}\int_{[\tilde{\sigma}_a(0)]}-\log(W_{\Sigma}(z))d\theta_1d\theta_2\cdots d\theta_n+(\sqrt{-1})^{n}\int_{[\tilde{\sigma}_a(0)]}\log|W_{\Sigma}(s_{\real>0}(r_{\add}(0)))|d\theta_1d\theta_2\cdots d\theta_n\\
    =&(\sqrt{-1})^{n}\int_{[\tilde{\sigma}_a(0)]}-\log(W_{\Sigma}(z))d\theta_1d\theta_2\cdots d\theta_n\\
    &+(2\pi\sqrt{-1})^nw(e)(\log(\sqrt{2r^-_0(0)})+\log(\sqrt{2r^+_0(0)}).
\end{align*}

\begin{align*}
    \int_{\Gamma_{\text{co}}}\Omega=&0.
\end{align*}
As a result, the integral $\int_{\Gamma_{\add}}\Omega$ is the sum of the integrals on all the pieces and it is

\begin{align}
\label{eq_integral_on_Gamma_add}
    \int_{\Gamma_{\add}}\Omega=&(2\pi\sqrt{-1})^nw(e)\big(\langle\log t\cdot r_{\add}(\frac{1}{2}),v_j\rangle+\log t^{\lambda_j}\big)+\frac{1}{2}(2\pi\sqrt{-1})^{n+1}w(e)\\
    &-(\sqrt{-1})^{n}\big(\int_{[\tilde{\sigma}_a(0)]}\log(W_{\Sigma}(z))d\theta_1d\theta_2\cdots d\theta_n-\int_{[\tilde{\sigma}_a(1)]}\log((t^{\lambda_j}z^{v_j})^{-1}W_{\Sigma}(z))d\theta_1d\theta_2\cdots d\theta_n\big).\nonumber
\end{align}

\begin{remark}
\label{remark_relation_RS_formula}
This is an explicit example of the computation of period integrals as in \cite{Ruddat-Siebert_2019} for the case of the relative loop around the focus-focus singularity. The tropical $1$-cycle represented in the $(u,z)$-cluster chart and $(v,z)$-cluster chart are shown in Figure \ref{fig_lifting_of_Gamma+} and Figure\ref{fig_lifting_of_Gamma-} respectively, with the blue arrows representing the chosen sections of the local system induced by the affine structure.  
\end{remark}

Now we set $\Gamma_{a,\text{total}}=\Gamma_a\cup\Gamma_{\add}$, then in order to get $\int_{\Gamma_a}\Omega$, we also need to evaluate $\int_{\Gamma_{a,\text{total}}}\Omega$. To do it, we isotope it to a \textit{singular fiber} of $Y^t$. We need the following lemmas to give the definition of singular fiber.

\begin{lemma}
\label{lemma_sing_fiber_1}
For $c_1,c_2\in\R$ with $c_1\geq0$, there exists a constant $C(c_1,c_2)>0$, such that
$$
   \frac{|r^{c_1}e^{\sqrt{-1}c_2\alpha}-1|}{|re^{\sqrt{-1}\alpha}-1|}<C(c_1,c_2)
$$
for $(r,\alpha)\in([0,1]\times\R/2\pi\Z)\backslash(1,[0])$.
\end{lemma}
\begin{proof}
    It is sufficient to show that there is a neighbourhood $V$ of $(1,[0])$ such that $\frac{|r^{c_1}e^{\sqrt{-1}c_2\alpha}-1|}{|re^{\sqrt{-1}\alpha}-1|}$ is bounded in $V\backslash(1,[0])$. To see this, note that
    \begin{align*}
        \frac{r^{c_1}e^{\sqrt{-1}c_2\alpha}-1}{re^{\sqrt{-1}\alpha}-1}     =r^{c_1}e^{\sqrt{-1}c_1}\frac{e^{\sqrt{-1}(c_2-c_1)\alpha}-1}{e^{\sqrt{-1}\alpha}-1}\frac{e^{\sqrt{-1}\alpha}-1}{re^{\sqrt{-1}\alpha}-1}+\frac{r^{c_1}e^{\sqrt{-1}c_1\alpha}-1}{re^{\sqrt{-1}\alpha}-1}.
    \end{align*}
    Now since $\lim\limits_{(r,\alpha)\rightarrow(1,0)}r^{c_1}e^{\sqrt{-1}c_1}\frac{e^{\sqrt{-1}(c_2-c_1)\alpha}-1}{e^{\sqrt{-1}\alpha}-1}=(c_2-c_1)$, $\lim\limits_{(r,\alpha)\rightarrow(1,0)}\frac{r^{c_1}e^{\sqrt{-1}c_1\alpha}-1}{re^{\sqrt{-1}\alpha}-1}=c_1$, and $|\frac{e^{\sqrt{-1}\alpha}-1}{re^{\sqrt{-1}\alpha}-1}|<\sqrt{2}$ for $0<r<1, \alpha\in(-\frac{1}{2}\pi,\frac{1}{2}\pi)$, such $V$ exists and thus the constant $C(c_1,c_2)$ exists.
\end{proof}

\begin{lemma}
\label{lemma_solution_of_W}
For sufficiently small $|t|$, suppose $x$ is a point in $U_j\cap F_{j,\trop}$ and set $\beta_{x,\text{amoeba}}=(x+\R\cdot v_j)\cap U_j\cap\mathscr{A}_t(W_{\Sigma,j})$, then for a point $\eta=(\eta_1,\eta_2,\cdots,\eta_{n-1})\in[0,2\pi]^{n-1}$, there exists a unique $\tau\in\R$ and $\alpha\in\R/2\pi\Z$ such that
$$
    x+\tau v_j\in\beta_{x,\text{amoeba}}
$$
and 
$$
W_{\Sigma,j}\big(\exp\big(\log t (x+\tau v_j)+\sqrt{-1}(\alpha \gamma_a+\sum_{i=1}^{n-1}\eta_i\xi_a^{\perp,i})\big)\big)=0.
$$
\end{lemma}

\begin{proof}
    Let us first show the uniqueness. Let us use $z(\tau,\alpha)$ to denote $\exp\big(\log t (x+\tau v_j)+\sqrt{-1}(\alpha \gamma_a+\sum_{i=1}^{n-1}\eta_i\xi_a^{\perp,i})\big)$, then it is sufficient to show that for a fixed point $\eta=(\eta_1,\eta_2,\dots,\eta_{n-1})$, $W_{\Sigma,j}(z(\tau_1,\alpha_1))-W_{\Sigma,j}(z(\tau_2,\alpha_2))\neq0$ whenever $(\tau_1,\alpha_1)\neq(\tau_2,\alpha_2)$. Note that 
    \begin{align*}
        &W_{\Sigma,j}(z(\tau_1,\alpha_1))-W_{\Sigma,j}(z(\tau_2,\alpha_2))\\
        =&(t^{\chi_j(x+\tau_1v_j)}e^{\sqrt{-1}\alpha_1}-t^{\chi_j(x+\tau_2v_j)}e^{\sqrt{-1}\alpha_2})\\
        &+\sum_{k\neq j}e^{\sqrt{-1}\sum\limits_{i=1}^{n-1}\eta_i\langle\xi_a^{\perp,i},v_k\rangle}\big(\rho_{j,t}(z(\tau_1,\alpha_1))t^{\chi_k(x+\tau_1v_j)}e^{\sqrt{-1}\alpha_1\langle \gamma_a,v_k\rangle}-\rho_{j,t}(z(\tau_2,\alpha_2))t^{\chi_k(x+\tau_2v_j)}e^{\sqrt{-1}\alpha_2\langle \gamma_a,v_k\rangle}\big).
    \end{align*}
    Since $\rho_{j,t}$ is constant on $\beta_{x,\text{amoeba}}$ according to Appendix \ref{appx_proof_of_defomation_lemmma} and $|\rho_{j,t}|\leq 1$, it is sufficient to show that 
    \begin{align*}
        &\frac{1}{p-1}|t^{\chi_j(x+\tau_1v_j)}e^{\sqrt{-1}\alpha_1}-t^{\chi_j(x+\tau_2v_j)}e^{\sqrt{-1}\alpha_2}|\\
        &>|t^{\chi_k(x+\tau_1v_j)}e^{\sqrt{-1}\alpha_1\langle \gamma_a,v_k\rangle}-t^{\chi_k(x+\tau_2v_j)}e^{\sqrt{-1}\alpha_2\langle \gamma_a,v_k\rangle}|
    \end{align*}
    for $k\neq j$.
    Now if $\tau_1=\tau_2$, we need to show that 
    \begin{align*}
        \frac{1}{p-1}|t^{\chi_j(x+\tau_1 v_j)-\chi_k(x+\tau_1 v_j)}|>|\frac{e^{\sqrt{-1}\langle \gamma_a,v_k\rangle(\alpha_1-\alpha_2)}-1}{e^{\sqrt{-1}(\alpha_1-\alpha_2)}-1}|,
    \end{align*}
    if $\tau_1\neq\tau_2$, say $\tau_1>\tau_2$, then if $\langle v_j,v_k\rangle\geq0$, we need to show that
    \begin{align*}
        \frac{1}{p-1}|t^{\chi_j(x+\tau_2 v_j)-\chi_k(x+\tau_2 v_j)}|>|\frac{t^{\langle v_j,v_k\rangle(\tau_1-\tau_2)}e^{\sqrt{-1}\langle \gamma_a,v_k\rangle(\alpha_1-\alpha_2)}-1}{t^{\langle v_j,v_j\rangle(\tau_1-\tau_2)}e^{\sqrt{-1}(\alpha_1-\alpha_2)}-1}|,
    \end{align*}
    if $\langle v_j,v_k\rangle\leq0$, we need to show that
    \begin{align*}
        \frac{1}{p-1}|t^{\chi_j(x+\tau_2 v_j)-\chi_k(x+\tau_1 v_j)}|>|\frac{t^{-\langle v_j,v_k\rangle(\tau_1-\tau_2)}e^{\sqrt{-1}\langle -\gamma_a,v_k\rangle(\alpha_1-\alpha_2)}-1}{t^{\langle v_j,v_j\rangle(\tau_1-\tau_2)}e^{\sqrt{-1}(\alpha_1-\alpha_2)}-1}|.
    \end{align*}
    Note that $\chi_j(x+\tau_2 v_j)-\chi_k(x+\tau_1 v_j)<c_j-\min_{k\neq j}\{c_k\}<0$ since $\beta_{x,\text{amoeba}}\subset U_j$, so according to Lemma \ref{lemma_sing_fiber_1}, there is a constant $C>0$ which is independent of $\tau,\alpha$ such that the right hand sides of the inequalities are all smaller than $C$ when $t<1$. Thus we can choose sufficiently small $|t|$ so that all the inequalities are satisfied, which gives the uniqueness.\\
    Now for the existence, set
    $$
    F(\tau,\alpha,\eta_1,\eta_2,\dots,\eta_{n-1})=W_{\Sigma,j}\big(\exp\big(\log t (x+\tau v_j)+\sqrt{-1}(\alpha \gamma_a+\sum_{i=1}^{n-1}\eta_i\xi_a^{\perp,i})\big)\big)
    $$
    and 
    $$
    \Theta=\{(\eta_1,\eta_2,\cdots,\eta_{n-1})\in[0,2\pi]^{n-1}\ |\ \exists\ \tau, \alpha \text{ such that }F(\tau,\alpha,\eta_1,\eta_2,\dots,\eta_{n-1})=0\}.
    $$
    Then $\Theta\neq\emptyset$ since $\beta_{x,\text{amoeba}}\neq\emptyset$. We have that $\Theta$ is closed since $\beta_{x,\text{amoeba}}$ is bounded. Since 
    $$
    \det\begin{pmatrix}
    \frac{\partial Re(F)}{\partial\tau} & \frac{\partial Re(F)}{\partial\alpha} \\
    \frac{\partial Im(F)}{\partial\tau} & \frac{\partial Im(F)}{\partial\alpha}
    \end{pmatrix}\neq0
    $$
    for any $(\tau,\alpha,\eta_1,\dots,\eta_{n-1})$,
    we also have that $\Theta$ is open according to implicit function theorem. Thus $\Theta=[0,2\pi]^{n-1}$, which leads to the existence.
\end{proof}
We denote the unique $\tau$ and $\alpha$ by $\tau_x(\eta)$ and $\alpha_x(\eta)$. Note that when $x=a$, we have $\tau_x(\eta)\equiv 0$ and $\alpha_x(\eta)\equiv\pi$.

\begin{definition}
The \textit{singular fiber} of $Y_{j,\text{simple}}^t$ at a point $x\in U_j\cap F_{j,\trop}$ with respect to the edge $e$ with end $a$ is the $(n+1)-$cycle
\begin{align*}
    SF_{x,e}:=\{&(z,u,v)\in(\C^*)^n\times\C^2\ |\\
    &z=\exp\big(\log t (x+\tau_x(\eta) v_j)+\sqrt{-1}((\alpha_x(\eta)+\eta_a)\gamma_a+\sum_{i=1}^{n-1}\eta_i\xi_a^{\perp,i})\big),\\
    &u=|W_{\Sigma,j}(z)|^{1/2}e^{\sqrt{-1}\theta_0},v=W_{\Sigma,j}(z)|W_{\Sigma,j}(z)|^{-1/2}e^{-\sqrt{-1}\theta_0},\\ &\eta\in[0,2\pi]^{n-1},\eta_a\in[0,2\pi],\theta_0\in[0,2\pi]\}.
\end{align*}
\label{def_singularfiber}
\end{definition}
Note that it is well-defined since for $(u,v,z)\in Y_{j,\simple}^t$ with $\Log_t(z)\in U_j$, $|u|<\sqrt{2R''},|v|<\sqrt{2R''}$ and thus $\tau(u,v)=0$.
Figure \ref{fig_the_singular_fiber} shows what a singular fiber looks like. We give the integral of $\Omega$ on the singular fiber.

\begin{figure}[htbp!]
    \centering
    \includegraphics[scale=0.5]{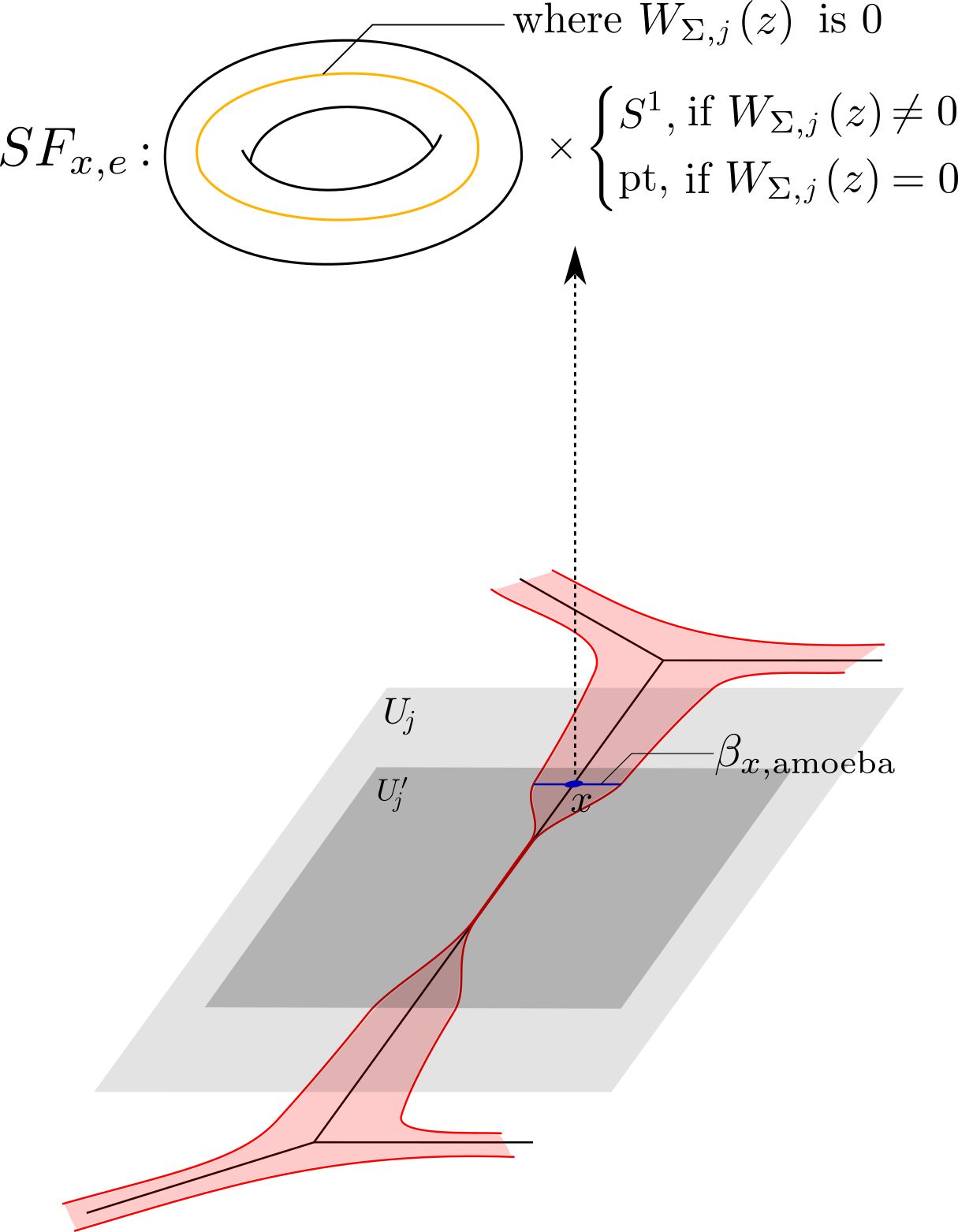}
    \caption{The singular fiber $SF_{x,e}$}
    \label{fig_the_singular_fiber}
\end{figure}

\begin{proposition}
\label{prop_integral_on_singularfiber}
    The integral of $\Omega$ on $SF_{x,e}$ is equal to $(2\pi\sqrt{-1})^{n+1}w(e)$.
\end{proposition}

\begin{proof}
    Note that $SF_{x,e}$ is parametrized by $(\theta_0,\eta_a,\eta)$ as in Definition \ref{def_singularfiber}. So we have that 
    \begin{align*}
        &\int_{SF_{x,e}}\Omega\\
        =&\int_{(\theta_0,\eta_a,\eta)\in[0,2\pi]^{n+1}}d\theta_0\big(\gamma_{a,1}d(\alpha(\eta)+\eta_a)+\sum_{i=1}^{n-1}\xi_{a,1}^{\perp,i}d\eta_k\big)\cdots\big(\gamma_{a,n}d(\alpha(\eta)+\eta_a)+\sum_{i=1}^{n-1}\xi_{a,n}^{\perp,i}d\eta_k\big)\\
        =&\int_{(\theta_0,\eta_a,\eta)\in[0,2\pi]^{n+1}}\det
        \begin{bmatrix}
        \vdots&\vdots& &\vdots\\ \gamma_a&\xi_{a}^{\perp,1}&\cdots&\xi_{a}^{\perp,n-1}\\\vdots&\vdots& &\vdots
        \end{bmatrix}
        d\theta_0d\eta_ad\eta\\
        =&(2\pi\sqrt{-1})^{n+1}w(e).
    \end{align*}
\end{proof}
Now we isotope $\Gamma_{a,\text{total}}$ to a singular fiber within $Y_{j}^t$ for an end $a\in U_j\cap F_{j,\trop}$.

\begin{proposition}
\label{prop_homotopy_to_singular_fiber}
    Suppose $a$ is an end of an edge $e$ of $\beta_{\trop,C}$ with $a\in U_j\cap F_{j,\trop}$, then there is a continuous map 
    $$
    H:\Gamma_{a,\text{total}}\times[0,1]\rightarrow Y_j^t,
    $$
    such that
    $H(\Gamma_{a,\text{total}}\times\{0\})=\Gamma_{a,\text{total}}$
    and $H(\Gamma_{a,\text{total}}\times\{1\})=SF_{\tilde{a},e}$ for a point $\tilde{a}\in F_{j,\trop}\cap U_j\backslash U_j'$.
\end{proposition}
    
\begin{proof}
The idea is that we isotope $\Phi_j(\Gamma_{a,\total})$ to $SF_{a,e}$ within $Y_{j,\simple}^t$ through 
$$
    H^1:\Phi_j(\Gamma_{a,\total})\times[0,1]\rightarrow Y_{j,\simple}^t,
$$
then isotope $SF_{a,e}$ to $SF_{\tilde{a},e}$ within $Y_{j,\simple}^t$ through 
$$
    H^2:SF_{a,e}\times[0,1]\rightarrow Y_{j,\simple}^t.
$$
Then $H$ is the pullback of $H_1$ and $H_2$ along $\Phi_j$, i.e.,
\begin{equation*}
    H((u,v,z),\zeta)=
    \begin{cases}
    \Phi_j^{-1}H^1(\Phi_j(u,v,z),2\zeta), & \zeta\in[0,1/2]\\
    \Phi_j^{-1}H^2(H^1(\Phi_j(u,v,z),1),2\zeta-1), & \zeta\in[1/2,1].
    \end{cases}
\end{equation*}
Now we come to construct $H^1$ and $H^2$. We will guarantee that for any point $(u,v,z)$ in the images of $H^1$ and $H^2$, either $\Log_t(z)\in U_j\backslash U_j'$ or $|u|<\sqrt{2R''}$ and $|v|<\sqrt{2R''}$, thus we always have $uv=W_{\Sigma,j}(z)$.\\
For $H^1$, it consists of two parts
\begin{equation*}
    H^1((u,v,z),\zeta)=
    \begin{cases}
    H^{1,\text{part}1}((u,v,z),2\zeta), & \zeta\in[0,1/2]\\
    H^{1,\text{part}2}(H^{1,\text{part}1}((u,v,z),1),2\zeta-1), & \zeta\in[1/2,1].
    \end{cases}
\end{equation*}
For $H^{1,\text{part}1}$, let us set 
$$
    K(x,z,\zeta)=\zeta\sqrt{|W_{\Sigma,j}(z)|}+(1-\zeta)x
$$ 
for $x\in\R_{\geq0}$, then $H^{1,\text{part}1}:\Phi_j(\Gamma_{a,\total})\times[0,1]\rightarrow Y_{j,\simple}^t$ is given as 
\begin{align*}
    &H^{1,\text{part}1}((u,v,z),\zeta)\\
    =&\begin{cases}
        (e^{-\sqrt{-1}\arg(v)}\frac{W_{\Sigma,j}(z)}{K(|v|,z,\zeta)},e^{\sqrt{-1}\arg(v)}K(|v|,z,\zeta),z) & (u,v,z)\in \Gamma^{-},\\
        (e^{\sqrt{-1}\arg(u)}K(|u|,z,\zeta)^{\mu(u,v,z)}(\frac{|W_{\Sigma,j}(z)|}{K(|v|,z,\zeta)})^{1-\mu(u,v,z)},\\
        e^{\sqrt{-1}\arg(v)}K(|v|,z,\zeta)^{1-\mu(u,v,z)}(\frac{|W_{\Sigma,j}(z)|}{K(|u|,z,\zeta)})^{\mu(u,v,z)},z) & (u,v,z)\in\Gamma_{\text{out}},\\
        (e^{\sqrt{-1}\arg(u)}K(|u|,z,\zeta)^{1-\mu(u,v,z)}(\frac{|W_{\Sigma,j}(z)|}{K(|v|,z,\zeta)})^{\mu(u,v,z)},\\
        e^{\sqrt{-1}\arg(v)}K(|v|,z,\zeta)^{\mu(u,v,z)}(\frac{|W_{\Sigma,j}(z)|}{K(|u|,z,\zeta)})^{1-\mu(u,v,z)},z) & (u,v,z)\in\Gamma_{\text{in}},\\
        (e^{\sqrt{-1}\arg(u)}K(|u|,z,\zeta),\frac{e^{-\sqrt{-1}\arg(u)}W_{\Sigma,j}(z)}{K(|u|,z,\zeta)},z), & \text{Otherwise},
    \end{cases}
\end{align*}
where $\mu(u,v,z)$ is determined through $|v|=|\frac{W_{\Sigma}(z)}{W_{\Sigma}(s_{\real>0}(r_{\add}(1)))}|^{\mu}\sqrt{2r_0^-(1)}$ and $|u|=|\frac{W_{\Sigma}(z)}{W_{\Sigma}(s_{\real>0}(r_{\add}(0)))}|^{\mu}\sqrt{2r_0^+(0)}$ for $(u,v,z)\in\Gamma_{\text{out}},\Gamma_{\text{in}}$ respectively. Note that $|u|=|v|$ on $H^{1,\text{part}1}(\Phi_j(\Gamma_{a,\total}),1)$.\\
For $H^{1,\text{part}2}$, we choose a continuous map
\begin{align*}
    \psi:r_a\cup r_\add\times[0,1]\rightarrow&\R^n\\
    (x,\zeta)\mapsto&\psi(x,\zeta)
\end{align*}
such that 
\begin{align*}
    &\psi(x,0)=x, \psi(x,1)=a,\\
    &\psi(r_a(\tau),\zeta)=r_a(\tau+(1-\tau)\zeta),\\
    &\psi(r_\add(0),\zeta), \psi(r_\add(1),\zeta)\notin\mathscr{A}_t(W_{\Sigma,j}) \text{ for } \zeta\neq1.
\end{align*}
Let us set $z(\zeta)=e^{\sqrt{-1}\arg(z)}t^{\psi(\Log_t(z),\zeta)}$, then $H^{1,\text{part}2}:H^{1,\text{part}1}(\Phi_j(\Gamma_{a,\total}),1)\times[0,1]\rightarrow Y_{j,\simple}^t$ is given as 
\begin{equation*}
    H^{1,\text{part}2}((u,v,z),\zeta)=(e^{\sqrt{-1}\arg(u)}\sqrt{|W_{\Sigma,j}(z(\zeta))|},e^{\sqrt{-1}\arg(v)}\sqrt{|W_{\Sigma,j}(z(\zeta))|},z(\zeta)).
\end{equation*}
\begin{figure}[htbp!]
    \centering
    \hspace*{-2cm}\includegraphics[scale=0.6]{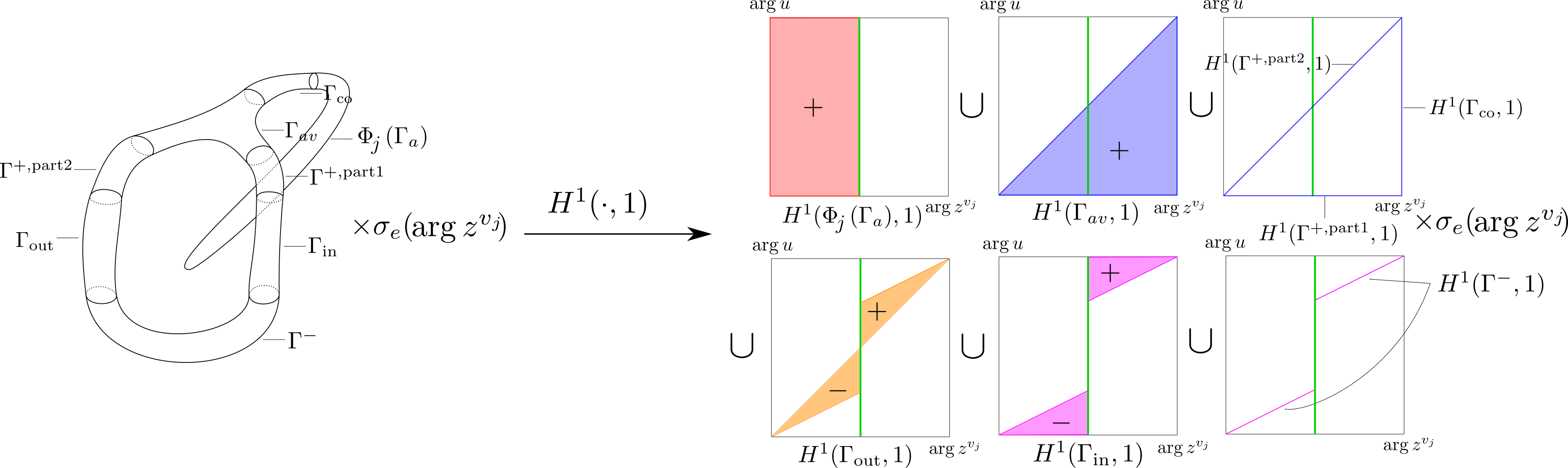}
    \caption{Homotopy to the singular fiber}
    \label{fig_homotopy_to_the_singular_fiber}
\end{figure}
Figure \ref{fig_homotopy_to_the_singular_fiber} shows the process of $H^1$, with $\sigma_e(\arg z^{v_j})=\sigma_e(e^{\sqrt{-1}\arg(z^{v_j})v_j}\cdot s_{\real>0}(\Log_t(z)))$. Its right side shows the pieces of the singular fiber $SF_{a,e}$ to which the pieces of $\Gamma_{a,\total}$ are homotopic to. The green part is where $u=v=0$ and the signs on the regions indicate the orientation of the pieces.\\
Now for $H^2$, we choose a path
\begin{align*}
    r_{SF}:[0,1]\rightarrow & U_j\cap F_{j,\trop}\\
    \zeta \mapsto & r_{SF}(\zeta)
\end{align*}
such that $r_{SF}(0)=a$ and $r_{SF}(1)=\tilde{a}$. Let us set $z_{SF}(\zeta)=\exp\big(\log t (r_{SF}(\zeta)+\tau_{r_{SF}(\zeta)}(\eta(z)) v_j)+\sqrt{-1}((\alpha_{r_{SF}(\zeta)}(\eta(z))+\eta_a(z))\gamma_a+\sum_{i=1}^{n-1}\eta_i(z)\xi_a^{\perp,i})\big)$,
where $\eta(z)=(\eta_1(z),\eta_2(z),\dots,\eta_{n-1}(z))$ and $\eta_a(z)$ are determined through $z=\exp\big(\log t (a+\tau_a(\eta) v_j)+\sqrt{-1}((\alpha_a(\eta)+\eta_a)\gamma_a+\sum_{i=1}^{n-1}\eta_i\xi_a^{\perp,i})\big)$. Then $H^2:SF_{a,e}\times[0,1]\rightarrow Y_{j,\simple}^t$ are given as 
\begin{equation*}
    H^{2}((u,v,z),\zeta)=(e^{\sqrt{-1}\arg(u)}\sqrt{|W_{\Sigma,j}(z_{SF}(\zeta))|},e^{\sqrt{-1}\arg(v)}\sqrt{|W_{\Sigma,j}(z_{SF}(\zeta))|},z_{SF}(\zeta)).
\end{equation*}

\begin{figure}[htbp!]
    \centering
    \includegraphics[scale=0.5]{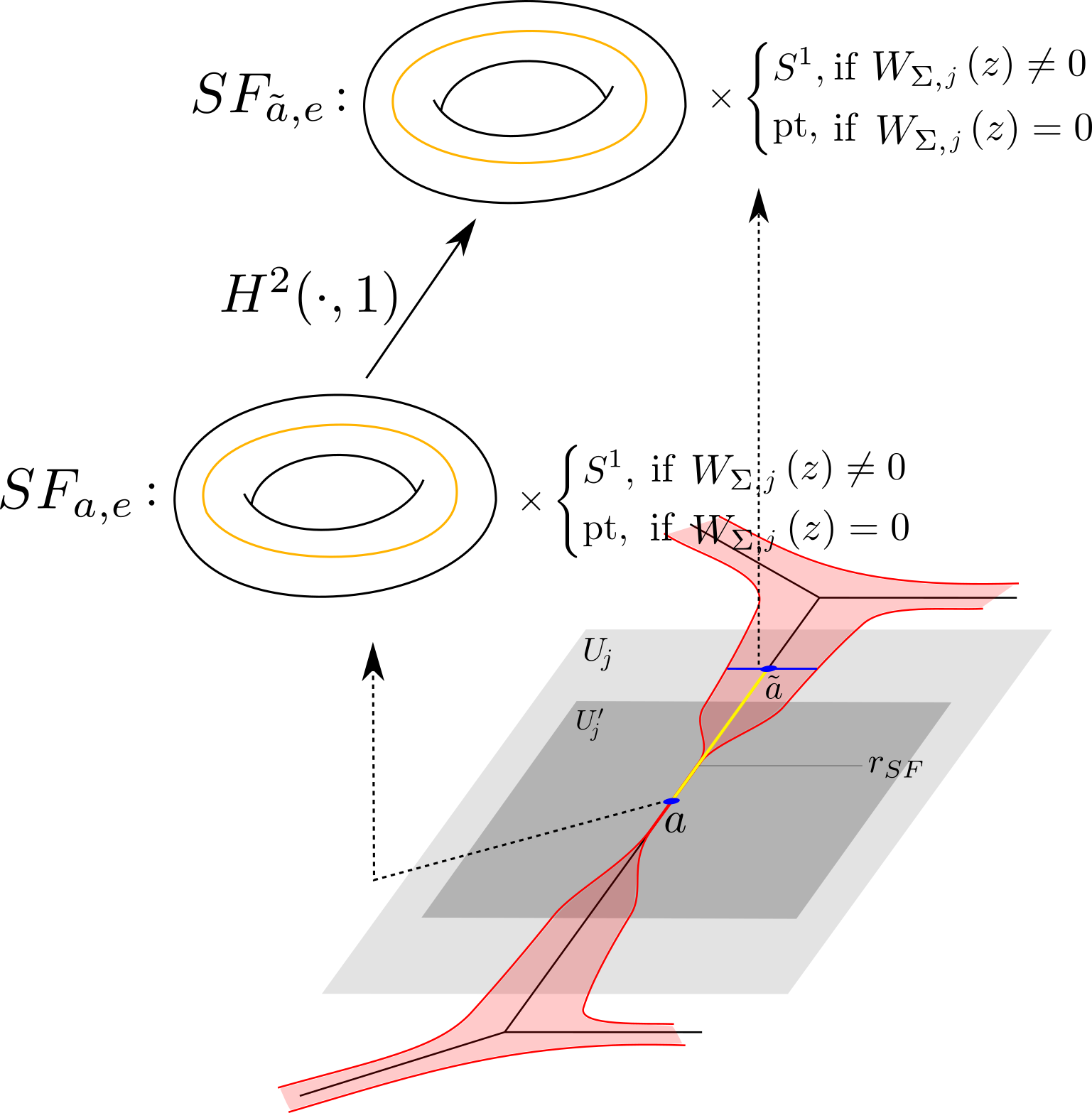}
    \caption{Homotopy between the singular fibers}
    \label{fig_Homotopy_between_the_singular_fibers}
\end{figure}

Figure \ref{fig_Homotopy_between_the_singular_fibers} shows the process of $H^2$.
\end{proof}

Now combining Proposition \ref{prop_homotopy_to_singular_fiber}, Proposition \ref{prop_integral_on_singularfiber} and (\ref{eq_integral_on_Gamma_add}), we have that
\begin{equation}
\label{eq_integral_on_Gamma_a}
    \begin{aligned}
    &\int_{\Gamma_a}\Omega\\
    =&\int_{\Gamma_{\text{total}}}\Omega-\int_{\Gamma_{\add}}\Omega\\ 
    =&\int_{SF_{\tilde{a},e}}\Omega-\int_{\Gamma_{\add}}\Omega\\ 
    =&-(2\pi\sqrt{-1})^nw(e)\big(\langle\log t\cdot r_{\add}(\frac{1}{2}),v_j\rangle+\log t^{\lambda_j}\big)+\frac{1}{2}(2\pi\sqrt{-1})^{n+1}w(e)\\ 
    &+(\sqrt{-1})^{n}\big(\int_{[\tilde{\sigma}_a(0)]}\log(W_{\Sigma}(z))d\theta_1d\theta_2\cdots d\theta_n-\int_{[\tilde{\sigma}_a(1)]}\log((t^{\lambda_j}z^{v_j})^{-1}W_{\Sigma}(z))d\theta_1d\theta_2\cdots d\theta_n\big)
\end{aligned}
\end{equation}
for an end $a\in F_{j,\trop}$.

We can now give the proof of Theorem \ref{thm_maintheorem_2}. We need the following lemma.
\begin{lemma}
\label{lemma_periodevaluation2}
Suppose $m(z_1,z_2,\dots,z_n)\in\C[z_1,z_2,\dots,z_n]$ is a monomial which is not a constant, then for any $(r_1,r_2,\cdots,r_n)\in\R^n$, we have that
$$
    \int_{[0,2\pi]^n}m(e^{r_1+\sqrt{-1}\theta_1},e^{r_1+\sqrt{-1}\theta_1},\dots,e^{r_n+\sqrt{-1}\theta_n})d\theta_1d\theta_2\cdots d\theta_n=0.
$$
\end{lemma}

\begin{proof}
    It follows from the fact that $\int_{[0,2\pi]}e^{r+\sqrt{-1}\theta}d\theta=0$ for any $r\in\R$.
\end{proof}

\begin{proof}[Proof of Theorem \ref{thm_maintheorem_2}]
    Summing up the integrals of $\Omega$ on all the pieces (\ref{eq_integral_on_Gamma_e}), (\ref{eq_integral_on_Gamma_v}) and (\ref{eq_integral_on_Gamma_a}), we have that 
    \begin{align*}
        \int_{L_{\trop,C}^t}\Omega=&(\sqrt{-1})^n\sum\limits_{j=1}^{p}\big(-(2\pi)^nE_j\log t^{\lambda_j}+\sum\limits_{a\in F_{j,\trop}}(\int_{[\tilde{\sigma}_a(0)]}\log( W_{\Sigma}(z))d\theta_1d\theta_2\cdots d\theta_n\\
        &-\int_{[\tilde{\sigma}_a(1)]}\log((t^{\lambda_j}z^{v_j})^{-1}W_{\Sigma}(z))d\theta_1d\theta_2\cdots d\theta_n)\big)+\frac{1}{2}(2\pi\sqrt{-1})^{n+1}(\sum\limits_{j=1}^pE_j+V),
    \end{align*}
    where $V=\sum\limits_{v\in\{\text{vertices of }\beta_{\trop,C}\}}\vol(\sigma_v)$ and $E_j=\sum\limits_{\sigma_e\subset F_j}\vol_{\text{int}}(\sigma_e)$. Replace $(t^{\lambda_1},t^{\lambda_2},\dots,t^{\lambda_p})$ back to $(1,1,\dots,1,t_1,t_2,\dots,t_{p-n})$ and compare it with (\ref{eq_result2_centralcharge_chernclass}), we see that it is sufficient to show that
    \begin{equation}
    \label{eq_main_thm2_formula_need_to_show}
    \begin{aligned}
        &\sum\limits_{j=1}^{p}\big((2\pi)^nE_j\log t_{j-n}\\
        &+\sum\limits_{a\in F_{j,\trop}}(-\int_{[\tilde{\sigma}_a(0)]}\log(W_{\Sigma}(z))d\theta_1d\theta_2\cdots d\theta_n+\int_{[\tilde{\sigma}_a(1)]}\log((t_{j-n}z^{v_j})^{-1}W_{\Sigma}(z))d\theta_1d\theta_2\cdots d\theta_n)\big)\\
        =&\sum_{s=1}^{p-n}(2\pi)^{n}E_{n+s}\log \check{t}_s(t),
    \end{aligned}
    \end{equation}
    where we set $t_{j-n}=1$ for $j=1,2,\dots,n$. \\
    For the left hand side of (\ref{eq_main_thm2_formula_need_to_show}), using the parametrizations
    \begin{align*}
        &\tilde{\sigma}_a(0)=\{\exp(\log t r_{\add}(0)+\sqrt{-1}(\eta_a\gamma_a+\sum\limits_{i=1}^{n-1}\eta_i\xi_a^{\perp,i})))\ |\ (\eta_a,\eta_1,\cdots,\eta_{n-1})\in[0,2\pi]^n\}\\
        &\tilde{\sigma}_a(1)=\{\exp(\log t r_{\add}(1)+\sqrt{-1}(\eta_a\gamma_a+\sum\limits_{i=1}^{n-1}\eta_i\xi_a^{\perp,i})))\ |\ (\eta_a,\eta_1,\cdots,\eta_{n-1})\in[0,2\pi]^n\}
    \end{align*}
    for an end $a$ on $F_{j,\trop}$ with $e$ the edge having $a$ as the end, we have that
    \begin{align*}
        &\int_{[\tilde{\sigma}_a(0)]}\log(W_{\Sigma}(z))d\theta_1d\theta_2\cdots d\theta_n-\int_{[\tilde{\sigma}_a(1)]}\log((t_{j-n}z^{v_j})^{-1}W_{\Sigma}(z))d\theta_1d\theta_2\cdots d\theta_n\\
        =&w(e)\int_{(\eta_a,\eta_1,\cdots,\eta_{n-1})\in[0,2\pi]^n}\log\big(W_{\Sigma}(\exp(\log t r_{\add}(0)+\sqrt{-1}(\eta_a\gamma_a+\sum\limits_{i=1}^{n-1}\eta_i\xi_a^{\perp,i})))\big)\\
        &-\log\big((t_{j-n}z^{v_j})^{-1}W_{\Sigma}(\exp(\log t r_{\add}(0)+\sqrt{-1}(\eta_a\gamma_a+\sum\limits_{i=1}^{n-1}\eta_i\xi_a^{\perp,i})))\big)d\eta_ad\eta_1 \cdots d\eta_{n-1}\\
        =&w(e)(2\pi)^n\big(C_{j,1}(t)-C_{j,2}(t)\big)
    \end{align*}
    according to Lemma \ref{lemma_periodevaluation2}, where $C_{j,1}(t)$ and $C_{j,2}(t)$ are the constant terms of $\log(W_{\Sigma}(z))$ and $\log((t_{j-n}z^{v_j})^{-1}W_{\Sigma}(z))$ when we regard them as power series in $z$. Now recall that $W_{\Sigma}(z)$ is given as
    $$
    W_{\Sigma}(z)=1+\sum_{i=1}^{n}z_i+\sum_{s=1}^{p-n}t_{s}z^{v_{n+s}},
    $$
    and thus 
    \begin{align*}
        &\log(W_{\Sigma}(z))=\sum_{k=1}^{\infty}\frac{(-1)^{k+1}}{k}(\sum_{i=1}^{n}z_i+\sum_{s=1}^{p-n}t_{s}z^{v_{n+s}})^k,\\
        &\log(W_{\Sigma}(z))=\sum_{k=1}^{\infty}\frac{(-1)^{k+1}}{k}(t_{j-n}^{-1}z^{-v_j}(1+\sum_{i=1}^{n}z_i+\sum_{s=1}^{p-n}t_{s}z^{v_{n+s}})-1)^k.
    \end{align*}
    So $C_{j,2}(t)=0$ according to Corollary \ref{cor_Fanocor1}. For $C_{j,1}(t)$, note that it is contributed from the term whenever $k=\sum\limits_{i=1}^{n}A_i+\sum\limits_{s=1}^{p-n}m_s$, with $A_i=\sum\limits_{s=1}^{p-n}-v_{n+s,i}m_s$ and all $A_i$'s and $m_s$'s are nonnegative. As a result,
    \begin{equation*}
    \begin{aligned}
        &C_{j,1}(t)\\
        =&\sum\limits_{A_j\geq0,m_s\geq 0} \frac{{(-1)}^{\sum\limits_{i=1}^{n}A_i+\sum\limits_{s=1}^{p-n}m_s+1}}{\sum\limits_{i=1}^{n}A_i+\sum\limits_{s=1}^{p-n}m_s}\binom{\sum\limits_{i=1}^{n}A_i+\sum\limits_{s=1}^{p-n}m_s}{A_1}\binom{\sum\limits_{i=1}^{n}A_i+\sum\limits_{s=1}^{p-n}m_s-A_1}{A_2}\cdots\binom{A_n+\sum\limits_{s=1}^{p-n}m_s}{A_n}\\
        &\binom{\sum\limits_{s=1}^{p-n}m_s}{m_1}\binom{\sum\limits_{s=1}^{p-n}m_s-m_1}{m_2}\cdots\binom{m_{p-n}}{m_{p-n}}t_1^{m_1}\cdots t_{p-n}^{m_{p-n}}\\
        =&\sum_{\mbox{\tiny$\begin{array}{cc}
        &\exists m_s>0,  \\
        &\sum\limits_{s=1}^{p-n}-v_{n+s,i}m_s\geq0,\forall 1\leq i\leq n 
        \end{array}$}}[(-1)^{1+\sum\limits_{s=1}^{p-n}(\sum\limits_{i=1}^nv_{n+s,i}-1)m_s}\\
        &(-1-\sum\limits_{s=1}^{p-n}(\sum\limits_{i=1}^nv_{n+s,i}-1)m_s)!
        \prod_{i=1}^n\frac{1}{(-\sum\limits_{s=1}^{p-n}v_{n+s,i}m_s)!}\prod_{s=1}^{p-n}\frac{1}{m_s!}t_1^{m_1}t_2^{m_2}\cdots t_{p-n}^{m_{p-n}}].
    \end{aligned}
    \end{equation*}
    Note that it is independent of $j$ and we denote it by $C_1(t)$.
    Thus
    \begin{equation}
    \begin{aligned}
        &\sum\limits_{j=1}^{p}\big((2\pi)^nE_j\log t_{j-n}\\
        &+\sum\limits_{a\in F_{j,\trop}}(-\int_{[\tilde{\sigma}_a(0)]}\log(W_{\Sigma}(z))d\theta_1d\theta_2\cdots d\theta_n+\int_{[\tilde{\sigma}_a(1)]}\log((t_{j-n}z^{v_j})^{-1}W_{\Sigma}(z))d\theta_1d\theta_2\cdots d\theta_n)\big)\\
        =&\sum\limits_{j=1}^pE_j((2\pi)^n\log t_{j-n}-C_1(t)).
    \end{aligned}
    \label{eq_LHS_simplified}
    \end{equation}
    
    For the right hand side of (\ref{eq_main_thm2_formula_need_to_show}), we have that
    \begin{align*}
        \log\check{t}_s(t)=&\frac{\partial w(t,\rho)}{\partial\rho_s}|_{\rho=0}\\
        =&\big(\log t_s-\Gamma'(1)\sum_{j=0}^pl^{(s)}_j+\\
        &\sum_{\mbox{\tiny$\begin{array}{cc}
        &\exists m_k>0,  \\
        \nonumber&\sum\limits_{k=1}^{p-n}l^{(k)}_im_k\geq0,\forall 1\leq i\leq n 
        \end{array}$}}[-\frac{l_0^{(s)}\Gamma'(1+\sum\limits_{k=1}^{p-n}l_0^{(k)}m_k)}{\Gamma^2(1+\sum\limits_{k=1}^{p-n}l_0^{(k)}m_k)}\prod_{j=1}^p\frac{1}{\Gamma(1+\sum\limits_{k=1}^{p-n}l^{(k)}_jm_k)}t_1^{m_1}t_2^{m_2}\cdots t_{p-n}^{m_{p-n}}]\big)\\
        =&\big(\log t_s+\\
        &(\sum\limits_{i=1}^nv_{n+s,i}-1)\sum_{\mbox{\tiny$\begin{array}{cc}
        &\exists m_k>0,  \\
        \nonumber&\sum\limits_{k=1}^{p-n}l^{(k)}_im_k\geq0,\forall 1\leq i\leq n 
        \end{array}$}}[(-1)^{1+\sum\limits_{k=1}^{p-n}(\sum\limits_{i=1}^nv_{n+k,i}-1)m_k}(-1-\sum\limits_{k=1}^{p-n}(\sum\limits_{i=1}^nv_{n+k,i}-1)m_k)!\\
        &\prod_{i=1}^n\frac{1}{(-\sum\limits_{k=1}^{p-n}v_{n+k,i}m_k)!}\prod_{k=1}^{p-n}\frac{1}{m_k!}t_1^{m_1}t_2^{m_2}\cdots t_{p-n}^{m_{p-n}}]\big).
    \end{align*}
    The first equation is because that $\sum_{i=1}^{n}v_{n+s,i}-1<0$ according to Corollary \ref{cor_Fanocor2}. The second equation is because that $\sum\limits_{j=0}^pl_j^{(s)}=0$ and $\frac{\Gamma'(b)}{\Gamma^2(b)}=(-1)^{b+1}(-b)!$ for any $b\in \mathbb{Z}_{\leq0}$. Note that the term 
    \begin{equation*}
        \begin{aligned}
        P(t):=&\sum_{\mbox{\tiny$\begin{array}{cc}
        &\exists m_k>0,  \\
        &\sum\limits_{k=1}^{p-n}l^{(k)}_im_k\geq0,\forall 1\leq i\leq n 
        \end{array}$}}[(-1)^{1+\sum\limits_{k=1}^{p-n}(\sum\limits_{i=1}^nv_{n+k,i}-1)m_k}(-1-\sum\limits_{k=1}^{p-n}(\sum\limits_{i=1}^nv_{n+k,i}-1)m_k)!\\
        &\prod_{i=1}^n\frac{1}{(-\sum\limits_{k=1}^{p-n}v_{n+k,i}m_k)!}\prod_{k=1}^{p-n}\frac{1}{m_k!}t_1^{m_1}t_2^{m_2}\cdots t_{p-n}^{m_{p-n}}]
        \end{aligned}
    \end{equation*}
    is independent of the choice of $s\in\{1,2,\dots,p-n\}$ and is equal to $C_1(t)$. So we have that
    \begin{equation}
        \log\check{t}_s(t)=\log t_s+(\sum\limits_{i=1}^nv_{n+s,i}-1)P(t)=\log t_s+(\sum\limits_{i=1}^nv_{n+s,i}-1)C_1(t).
    \label{eq_mir_map_explicit_form}
    \end{equation}

    Thus
    \begin{equation}
        \sum_{s=1}^{p-n}(2\pi)^{n}E_{n+s}\log \check{t}_s(t)=\sum_{s=1}^{p-n}(2\pi)^{n}E_{n+s}(\log t_s+(\sum\limits_{i=1}^nv_{n+s,i}-1)C_1(t)).
    \label{eq_RHS_simplified}
    \end{equation}
    
    Compare it to (\ref{eq_LHS_simplified}) and what is left to check is 
    $$
        -\sum\limits_{j=1}^pE_j=\sum\limits_{s=1}^{p-n}E_{n+s}(\sum\limits_{i=1}^nv_{n+s,i}-1).
    $$
    Note that $E_j=[D_j][Q_1][Q_2]\cdots[Q_{n-1}]$, so it is sufficient to check that
    $$
        -\sum\limits_{j=1}^p[D_j]=\sum\limits_{s=1}^{p-n}[D_{n+s}](\sum\limits_{i=1}^nv_{n+s,i}-1),
    $$
    which is true since 
    \begin{equation*}
        \sum\limits_{s=1}^{p-n}\sum\limits_{i=1}^nv_{n+s,i}[D_{n+s}]=\sum\limits_{k=1}^n\sum\limits_{i=1}^nv_{k,i}[D_k]=\sum_{k=1}^n[D_k].
    \end{equation*}
\end{proof}
Thus we show that the central charges $Z_t(E_C)$ and $C_t(L_{\trop,C}^t)$ are equal to each other.

\section{Relation to the Gross-Siebert Model of Local Mirror Symmetry}
\label{section4}
The same method of the evaluation of central charges through tropical geometry can be applied to the Gross-Siebert model of local mirror symmetry as in \cite{Gross-Siebert_2014}. The model is given as follows. The canonical bundle $X$ can be put into a one higher dimensional toric variety, and there is a toric degeneration $\mathcal{X}$ within it from $X$ to the toric boundary. Then the dual intersection complex $(B,\mathscr{P},\Delta)$ corresponding to $\mathcal{X}$ and a proper choice of multivalued piecewise linear function give rise to a dual toric degeneration $\mathcal{Y}$, which is regarded as the Gross-Siebert mirror of $X$. We do not give the details of the construction, but focus on a family of open subsets of $\mathcal{Y}$, which is given as 
\begin{equation}
\label{eq_GS_model}
    Y_{GS}^{\check{t}}=\{(\check{u},\check{v},\check{z}_1,\check{z}_2,\dots,\check{z}_n )\in \mathbb{C}^2\times(\mathbb{C}^*)^n\ |\ \check{u}\check{v}=W_{\Sigma,GS}(\check{z})\},
\end{equation}
with $W_{\Sigma,GS}(\check{z})=1+\sum_{i=1}^n\check{z}_i+\sum_{s=1}^{p-n}\check{t}_s\check{z}^{v_{n+s}}+h(\check{t})$, where $h(\check{t})\in\C\llbracket \check{t} \rrbracket $ is to guarantee that $W_{\Sigma,GS}$ satisfies the normalization condition, i.e., $\log(W_{\Sigma,GS}(\check{z}))$ has no $(\check{t})^m$-terms for $m\neq0$ when we regard it as a power series in $\check{z}$. We will later show that $h(\check{t})$ is analytic in a neighbourhood of $\check{t}=0$. Thus we can apply the method in Section \ref{section3.3} to get a piecewise Lagrangian closed submanifold $L_{\trop,C}^{\check{t}}$ in $Y_{GS}^{\check{t}}$ from $\beta_{\trop,C}$. The evaluation of the central charge $C_{\check{t}}(L_{\trop,C}^{\check{t}})$ is similar to $C_t(L_{\trop,C}^t)$, with the only difference that 
$$
    \int_{[\tilde{\sigma}_a(0)]}\log(W_{\Sigma,GS}(\check{z}))d\theta_1\dots d\theta_n=0.
$$
Thus we get Theorem \ref{thm_mainthm_GS} proved.

Note that the Gross-Siebert model (\ref{eq_GS_model}) is a modification of the Hori-Vafa mirror $Y^t$ (\ref{eq_local_mirror_symmetry}) with the correction term $h(\check{t})$. It turns out that these two models are equivalent under the mirror map
\begin{equation}
\label{eq_mirror_map}
    \check{t}_s(t)=\exp(\frac{\partial}{\partial\rho_s}|_{\rho=0}w(t,\rho)),
\end{equation}
where $w(t,\rho)$ is the formal solution to the Picard-Fuchs equation (\ref{eq_solution_to_P-F}).

\begin{proof}[Proof of Theorem \ref{thm_central_charge_match_mirror_map}]
Note that according to (\ref{eq_mir_map_explicit_form}), we have that 
$$
    \check{t}_s(t)=e^{(\sum\limits_{i=1}^nv_{n+s,i}-1)P(t)}t_s.
$$
So 
\begin{align}
\label{eq_two_model_equiv}
    &1+\sum_{i=1}^n \check{z}_i+\sum_{s=1}^{p-n}\check{t}_s(t)\check{z}^{v_{n+s}}+h(\check{t}(t))\\ \nonumber
    =&1+h(\check{t}(t))+\sum_{i=1}^n \check{z}_i+\sum_{s=1}^{p-n}e^{-P(t)}t_s(e^{P(t)}\check{z})^{v_{n+s}}\\ \nonumber
    =&e^{-P(t)}\big((1+h(\check{t}(t)))e^{P(t)}+\sum_{i=1}^n e^{P(t)}\check{z}_i+\sum_{s=1}^{p-n}t_s(e^{P(t)}\check{z})^{v_{n+s}}\big). \nonumber
\end{align}
Taking logarithm on both sides of (\ref{eq_two_model_equiv}) and compare their constant terms, we have that
$$
    0=-P(t)+\check{C}(t),
$$
where $\check{C}(t)$ is the constant term of $\log\big((1+h(\check{t}(t)))e^{P(t)}+\sum_{i=1}^n e^{P(t)}\check{z}_i+\sum_{s=1}^{p-n}t_s(e^{P(t)}\check{z})^{v_{n+s}}\big)$ when we expand it with respect to $\check{z}$. Then we have that $(1+h(\check{t}(t)))e^{P(t)}=1$ since the constant term of $\log\big(1+\sum_{i=1}^n e^{P(t)}\check{z}_i+\sum_{s=1}^{p-n}t_s(e^{P(t)}\check{z})^{v_{n+s}}\big)$ is equal to the constant term of $\log(1+W_{\Sigma}(z))$, which is $P(t)$. Furthermore, since $P(t)$ is analytic in a neighbourhood of $t=0$ and the inverse mirror map $t(\check{t})$ is analytic in a neighbourhood of $\check{t}=0$, we have that $h(\check{t})=e^{-P(t(\check{t}))}-1$ is analytic in a neighbourhood of $\check{t}=0$. Thus $Y^{\check{t}}_{GS}$ is a family of analytic submanifolds. Now, if we set
$$
    z_i=e^{P(t)}\check{z}_i,
$$
we have that 
$$
W_{\Sigma,GS}(\check{t}(t),\check{z})=e^{-P(t)}W_{\Sigma}(t,z).
$$ Thus the diffeomorphism $\Psi_{GS\rightarrow HV}:Y^{\check{t}}_{GS}\rightarrow Y^t$ is given by 
$$
\Psi_{GS\rightarrow HV}(\check{u},\check{v},\check{z}_1,\cdots,\check{z}_n)=(e^{-\frac{1}{2}P(t)}\check{u},e^{-\frac{1}{2}P(t)}\check{v},e^{P(t)}\check{z}_1,e^{P(t)}\check{z}_2,\cdots,e^{P(t)}\check{z}_n).
$$ Then $\Psi_{GS\rightarrow HV}^*(\Omega)=d\log\check{u}d\log\check{z}_1\dots d\log\check{z}_n$ and we have that 
$$
    \int_{\Psi_{GS\rightarrow HV}^{-1}(L_{\beta_{\trop,C}}^t)}d\log\check{u}d\log\check{z}_1\dots d\log\check{z}_n=C_t(L_{\beta_{\trop,C}}^t)(t).
$$ It is also equal to $C_{\check{t}}(L_{\beta_{\trop,C}}^{\check{t}})(\check{t}(t))$ since $\log(\check{t}_s(t))=\log t_s+(\sum\limits_{i=1}^nv_{n+s,i}-1)P(t)$.
\end{proof}

\begin{remark}
\label{remark_instanton_open_GW}
\begin{itemize}
    \item The term $h(\check{t})$ can be interpreted using the open Gromov-Witten invariants of $X$ as in \cite{C-L-L_2012}. It is given by
    $$
        h(\check{t})=\sum_{\alpha\neq0}n_{\beta_0+\alpha}\exp(-\int_{\alpha}\sum_{s=1}^{p-n}(\log \check{t}_{s})[\tilde{D}_{n+s}]),
    $$
    where the sum is over $\alpha\in H_2(X,\Z)\backslash 0$ which can be represented by rational curves, $\beta_0\in H_2(X,L)$ corresponds to $\tilde{v}_0$ as in \cite{Cho-Oh_2006}, and $n_{\beta_0+\alpha}$ is the open Gromov-Witten invariants counting the number of pseudo-holomorphic disks in class $\beta_0+\alpha$. They use this to give the formula of the inverse mirror map $t(\check{t})$. These open Gromov-Witten invariants are also related to closed Gromov-Witten invariants when $X_\Sigma$ is a smooth Fano toric surface, as mentioned in Section 2 of \cite{G-R-Z_2022}. In \cite{Gross-Siebert_2014}, Gross and Siebert interpret $h(\check{t})$ as the counting of tropical trees.
    \item The equivalence of the two models of local mirror symmetry is also discussed in \cite{Lau_2014}.
\end{itemize}
\end{remark}

\appendix
\section{Proof of Lemma \ref{lemma_defomation_near_amoeba}}
\label{appx_proof_of_defomation_lemmma}
To prove Lemma \ref{lemma_defomation_near_amoeba}, we need the following lemmas.

\begin{lemma}
\label{lemma_appxlemma1}
Suppose $Y_\mu,\mu\in[0,1]$ is a smooth family of symplectic submanifolds in $(\C^2\times(\C^*)^n,\omega_0)$ and $Y_\mu$ is independent of $\mu$ away from a compact set of $\C^2\times(\C^*)^n$, then there exists a symplectomorphism $\Phi:Y_0\rightarrow Y_1$.
\end{lemma}

\begin{proof}
Consider the submanifold
$$
    \tilde{Y}:=\{(u,v,z_1,\cdots,z_n,\mu)\in\C^2\times(\C^*)^n\times[0,1]\ |\ (u,v,z_1,\cdots,z_n)\in Y_\mu\}
$$
and let $p:\tilde{Y}\rightarrow[0,1]$ be its projection to the component $[0,1]$. The smoothness of $Y_{\mu}$ guarantees that $p$ is a submersion. Now let us choose a compact subset $K$ of $\C^2\times(\C^*)^n$ such that $Y_\mu$ are independent of $\mu$ on $K^c$ and $Y_\mu$ intersect $\partial K$ transversely, then we can apply Ehresmann's lemma to $p|_{\tilde{Y}\cap (K\times[0,1])}$ and get a diffeomorphism
$$
    i_K:\tilde{Y}\cap (K\times[0,1])\rightarrow(Y_0\cap K)\times[0,1],
$$
such that $i_K|_{(Y_\mu\cap K)\times\{\mu\}}:(Y_\mu\cap K)\times\{\mu\}\rightarrow (Y_0\cap K)\times\{\mu\}$ is a diffeomorphism and $i_K|_{(Y_\mu\cap\partial K)\times\{\mu\}}$ is an identity for any $\mu\in[0,1]$. Thus we get a trivialization of $\tilde{Y}$ by extending $\id_{Y_\mu\cap K^c\times[0,1]}$ with $i_K$, i.e., a diffeormorphism
$$
    i: \tilde{Y}\rightarrow Y_0\times[0,1]
$$
such that $i|_{Y_\mu\times\{\mu\}}:Y_\mu\times\{\mu\}\rightarrow Y_0\times\{\mu\}$ is a diffeomorphism for any $\mu\in[0,1]$. We write $i|_{Y_\mu\times\{\mu\}}$ as a diffeomorphism from $Y_\mu$ to $Y_0$ and write it as $\varphi_\mu:Y_\mu\rightarrow Y_0$. Now we know that $\varphi_\mu^*(\omega_0|_{Y_\mu})$ and $\omega_0|_{Y_0}$ have the same cohomology class in $H^2(Y_0,\C)$ by Poincare's Lemma. Furthermore, $\varphi_\mu^*(\omega_0|_{Y_\mu})$ is equal to $\omega_0|_{Y_0}$ away from a compact subset of $Y_0$, so we can apply Moser's argument to these two symplectic forms on $Y_0$ and get a symplectomorphism $\check{\varphi}_\mu:(Y_0,\varphi_\mu^*(\omega_0|_{Y_\mu}))\rightarrow (Y_0,\omega_0|_{Y_0})$. Then $\Phi:=\check{\varphi}_1\circ\varphi_1:Y_1\rightarrow Y_0$ is the syplectomorphism we want.
\end{proof}

Now we come to give a proof of Lemma \ref{lemma_defomation_near_amoeba}

\begin{proof}[Proof of Lemma \ref{lemma_defomation_near_amoeba}]
We need to construct a family of smooth symplectic submanifolds $Y_\mu$ in $\C^2\times(\C^*)^n$ which satisfies the conditions in Lemma \ref{lemma_appxlemma1} such that $Y_0=Y_j^t$ and $Y_1=Y^t_{j,\simple}$. We first choose $|t|$ small enough and an open set
$$
    V_j:=\{x\in\R^n\ |\ |\chi_j(x)|<b_j, \chi_k(x)>b_k>b_j\text{ for }k\neq j\},
$$
such that $V_j$ contains an open subset of $F_{j,\trop}$ and for any $\rho\in[0,1]$ and $x\in\mathscr{A}_t(W_{\Sigma}(z)-\rho\sum_{k\neq j}t^{\lambda_k}z^{v_k})\cap V_j\cap F_{j,\trop}$, $(x+\R\cdot v_j)\cap V_j$ is the connected component of $x$ in $(x+\R\cdot v_j)\cap \mathscr{A}_t(W_{\Sigma}(z)-\rho\sum_{k\neq j}t^{\lambda_k}z^{v_k})$. Then let $U_j''\subset\subset U_j'\subset\subset U_j$ be the open sets 
\begin{equation}
\label{eq_U_j}
\begin{aligned}
    &U_j:=\{x\in\R^n\ |\ |\chi_j(x)|<c_j, \chi_k(x)>c_k>c_j\text{ for }k\neq j\}\text{ with }c_j>b_j, c_k>b_k,\\
    &U_j':=\{x\in\R^n\ |\ |\chi_j(x)|<c_j', \chi_k(x)>c_k'>c_j'\text{ for }k\neq j\}\text{ with }b_j<c_j'<c_j, c_k'>c_k,\\
    &U_j'':=\{x\in\R^n\ |\ |\chi_j(x)|<c_j'', \chi_k(x)>c_k''>c_j''\text{ for }k\neq j\}\text{ with }b_j<c_j''<c_j', c_k''>c_k'.
\end{aligned}
\end{equation}
Now let us set
\begin{align*}
    \mathcal{R}:=&\{(u,v)\in\C^2\ |\ |u|<\sqrt{2R}, |v|<\sqrt{2R}\},\\
    \mathcal{R}':=&\{(u,v)\in\C^2\ |\ |u|<\sqrt{2R'}, |v|<\sqrt{2R'}\},\\
    \mathcal{R}'':=&\{(u,v)\in\C^2\ |\ |u|<\sqrt{2R''}, |v|<\sqrt{2R''}\},\\
    \mathcal{U}_j:=&\Log^{-1}(U_j),\, \mathcal{U}_j':=\Log^{-1}(U_j'),\,
    \mathcal{U}_j'':=\Log^{-1}(U_j''),
\end{align*}
and let $\tau:\mathcal{R}\rightarrow\R$, $\rho_j:\mathcal{U}_j\rightarrow\R$ be two smooth functions such that
$$
    \tau(u,v)
    \left\{
    \begin{array}{cc}
        =1 & \text{ if } (u,v)\in\mathcal{R}'' \\
        \in[0,1] & \text{ if } (u,v)\in\mathcal{R}'\backslash\mathcal{R}''\\
        =0 & \text{ if } (u,v)\in\mathcal{R}\backslash\mathcal{R}'\\
    \end{array}
    \right.,
$$
$$
    \rho_{j}(z)
    \left\{
    \begin{array}{cc}
        =1 & \text{ if } z\in\mathcal{U}_j'' \\
        \in[0,1] & \text{ if } z\in\mathcal{U}_j'\backslash\mathcal{U}_j''\\
        =0 & \text{ if } z\in\mathcal{U}_j\backslash\mathcal{U}_j'
    \end{array}
    \right.,
$$
and
$$
    \rho_{j,t}(z)=\rho_j(z^{1/\log t}),|d\tau|<C,|d\rho_j|<C
$$
for some constant $C$. We further require that for $x
\in U_j\cap F_{j,\trop}$, $\rho_j|_{\Log^{-1}((x+\R\cdot v_j)\cap (|\chi_j(x)|<c_j''))}$ is constant.\\
Now we set 
$$
    Y_\mu:=\{(u,v,z)\in\C^2\times(\C^*)^n\ |\ uv=1+t^{\lambda_j}z^{v_j}+(1-\mu\tau(u,v)\rho_{j,t}(z))\sum\limits_{k\neq j}t^{\lambda_k}z^{v_k},(u,v)\in\mathcal{R},\Log_t(z)\in U_j\}
$$
for $\mu\in[0,1]$. Then 
$$
    Y_0=Y^t\cap(\mathcal{R}\times\Log_t^{-1}(U_j))
$$ 
and
$$
    Y_1=(uv=W_{\Sigma,j}(u,v,z))\cap(\mathcal{R}\times\Log_t^{-1}(U_j)),
$$
with 
\begin{equation}
\label{eq_W_Sigma,j}
W_{\Sigma,j}(u,v,z)=W_{\Sigma}(z)-\tau(u,v)\rho_{j,t}(z)\sum_{k\neq j}t^{\lambda_k}z^{v_k}.
\end{equation} 
Then we can see that $Y_0=Y_j^t$. Let us set $Y_{j,\simple}^t=Y_1$, then 
$$
    Y^t_{j,\text{simple}}\cap(\mathcal{R}''\times\Log_t^{-1}(U_j''))=(uv=1+t^{\lambda_j}z^{v_j})\cap(\mathcal{R}''\times\Log_t^{-1}(U_j''))
$$
since $\tau(u,v)\rho_{j,t}(z)=1$ for any $(u,v,z)\in\mathcal{R}''\times\Log_t^{-1}(U_j'')$. Note that $\Log_t(\{z\,|\,(0,0,z)\in Y^t_{j,\text{simple}}\})=\Log_t(\{z\,|\,W_{\Sigma,j}(0,0,z)=0\})$, which is contained in $|\chi_j(x)|<c_j''$ since $\rho_{j,t}(z)\in[0,1]$ and $c_j''>b_j$. So for a point $x\in U_j\cap F_{j,\trop}$,
\begin{align*}
&(x+\R\cdot v_j)\cap\Log_t(\{z\,|\,(0,0,z)\in Y^t_{j,\text{simple}}\})\\
=&(x+\R\cdot v_j)\cap\Log_t(\{z\,|\,W_{\Sigma}(z)-\rho_{j,t}(z)\sum_{k\neq j}t^{\lambda_k}z^{v_k}=0\})\cap(|\chi_j(x)|<c_j'').
\end{align*}
Note that $\rho_j|_{\Log^{-1}((x+\R\cdot v_j)\cap (|\chi_j(x)|<c_j''))}$ is constant, so it is connected due to the choice of $V_j$.\\
Now we want to construct $\Phi_j$ using Lemma \ref{lemma_appxlemma1}. we need to check that $Y_\mu$ is smooth, symplectic and independent of $\mu$ away from a compact subset in $\mathcal{R}\times\Log_t^{-1}(U_j)$.
Let $v_{j,1}\neq0$ without loss of generality, we choose $|t|$ small enough such that 
$$
    (p-1)t^{\min\limits_{k\neq j}\{c_k\}-c_j}<\min\{1,\frac{\sqrt{2}}{C},\frac{1}{2\max\limits_{k\neq j}\{|v_{k,1}v_{j,1}^{-1}|\}},\frac{|v_{j,1}|}{2Ce^{\max\limits_{x\in U_j}\{x_1\}}}\}.
$$\\
For smoothness, when $|u|\in(\sqrt{2R''},\sqrt{2R}]$, we have the $dv$-term to be 
$$
    udv+\mu\rho_{j,t}(z)(\sum\limits_{k\neq j}t^{\lambda_k}z^{v_k})\frac{\partial\tau(u,v)}{\partial v}dv\neq0
$$
since $|u|>\sqrt{2R''}$ and $|\mu\rho_{j,t}(z)(\sum\limits_{k\neq j}t^{\lambda_k}z^{v_k})\frac{\partial\tau(u,v)}{\partial v}|<(p-1)t^{\min\limits_{k\neq j}\{c_k\}}C<\sqrt{2}<\sqrt{2R''}$. When $(u,v)\in\mathcal{R}'$ and $u\neq0$, we have the $dv$-term to be 
$$
    udv+\mu\rho_{j,t}(z)(\sum\limits_{k\neq j}t^{\lambda_k}z^{v_k})\frac{\partial\tau(u,v)}{\partial v}dv\neq0
$$
since $\tau(u,v)\equiv1$ for $(u,v)\in\mathcal{R}'$ and $u\neq0$.
When $u=v=0$, we have the $dz_1$-term paired with the tangent vector $v_{j,1}^{-1}(t^{-\lambda_j}z^{-v_j})z_1\frac{\partial}{\partial z_1}$ to be 
\begin{align*}
    &\big(\frac{\partial(t^{\lambda_j}z^{v_j})}{\partial z_1}-\mu\frac{\partial\rho_{j,t}(z)}{\partial z_1}(\sum\limits_{k\neq j}t^{\lambda_k}z^{v_k})+(1-\mu\rho_{j,t}(z))\frac{\partial (\sum\limits_{k\neq j}t^{\lambda_k}z^{v_k})}{\partial z_1}\big)dz_1(v_{j,1}^{-1}(t^{-\lambda_j}z^{-v_j})z_1\frac{\partial}{\partial z_1})\\
    =&1+(1-\mu\rho_{j,t}(z))\sum\limits_{k\neq j}v_{j,1}^{-1}\big(v_{k,1}(t^{\lambda_k-\lambda_j}z^{v_k-v_j})\big)-\mu\frac{1}{\log t}\frac{\partial\rho_{j}(z^{1/\log t})}{\partial z_1}z_1^{1/\log t}v_{j,1}^{-1}\big(\sum\limits_{k\neq j}t^{\lambda_k-\lambda_j}z^{v_k-v_j})\\
    \neq&0
\end{align*}
since 
$$
    |(1-\mu\rho_{j,t}(z))\sum\limits_{k\neq j}v_{j,1}^{-1}\big(v_{k,1}(t^{\lambda_k-\lambda_j}z^{v_k-v_j})\big)|<(p-1)\max\limits_{k\neq j}\{|v_{k,1}v_{j,1}^{-1}|\}t^{\min\limits_{k\neq j}\{c_k\}-c_j}<\frac{1}{2}
$$
and 
$$
    |\mu\frac{1}{\log t}\frac{\partial\rho_{j}(z^{1/\log t})}{\partial z_1}z_1^{1/\log t}v_{j,1}^{-1}\big(\sum\limits_{k\neq j}t^{\lambda_k-\lambda_j}z^{v_k-v_j})|<C(p-1)e^{\max\limits_{x\in U_j}\{x_1\}}|v_{j,1}|^{-1}t^{\min\limits_{k\neq j}\{c_k\}-c_j}<\frac{1}{2}.
$$\\
For symplecticity, when $\tau(u,v)\rho_{j,t}(z)$ is constant in a neighbourhood of a point, $Y_\mu$ is symplectic at that point since $Y_\mu$ is determined by a holomorphic function in a neighbourhood of that point. When $\tau(u,v)\rho_{j,t}(z)$ is non-constant in a neighbourhood of the point $(u,v,z)$, i.e., $(u,v)\in\mathcal{R}'\backslash\mathcal{R}''$ or $(u,v)\in\mathcal{R}''$ and $z\in\Log_t^{-1}(U_j'\backslash U_j'')$, it is sufficient to check that the norm of $vdu+udv-d(W_\Sigma(z))$ dominates the norm of $\mu d(\tau(u,v)\rho_{j,t}(z)(\sum_{k\neq j}t^{\lambda_k}z^{v_k}))$ with the norms of $du,dv,d\log z_1,\dots,d\log z_n$ taken to be $1$. When $(u,v)\in\mathcal{R}'\backslash\mathcal{R}''$,
$$
    \parallel vdu+udv-dW_\Sigma(z)\parallel \text{at least has }O(1)-\text{order},
$$
while
\begin{align*}
    &\parallel-\mu d(\tau(u,v)\rho_{j,t}(z)(\sum\limits_{k\neq j}t^{\lambda_k}z^{v_k}))\parallel\\
    =&\mu\parallel\rho_{j,t}(z)(\sum\limits_{k\neq j}t^{\lambda_k}z^{v_k})d\tau(u,v)+\tau(u,v)(\sum\limits_{k\neq j}t^{\lambda_k}z^{v_k})d\rho_{j,t}(z)+\tau(u,v)\rho_{j,t}(z)d(\sum\limits_{k\neq j}t^{\lambda_k}z^{v_k})\parallel\\
    &\text{at most has }O(t^{\min\limits_{k\neq j}\{c_k\}})-\text{order}.
\end{align*}
When $(u,v)\in\mathcal{R}''$ and $z\in\Log_t^{-1}(U_j'\backslash U_j'')$, we have that
$$
    \parallel vdu+udv-dW_\Sigma(z)\parallel \text{at least has }O(t^{c_j})-\text{order},
$$
since $||(v_{j,1}t^{\lambda_j}z^{v_j}+\sum\limits_{k\neq j}v_{k,1}t^{\lambda_k}z^{v_k})d\log z_1||$ at least has $O(t^{c_j})$-order,
while
\begin{align*}
    &\parallel-\mu d(\tau(u,v)\rho_{j,t}(z)(\sum\limits_{k\neq j}t^{\lambda_k}z^{v_k}))\parallel\\
    =&\mu\parallel\rho_{j,t}(\sum\limits_{k\neq j}t^{\lambda_k}z^{v_k})d\tau(u,v)+\tau(u,v)(\sum\limits_{k\neq j}t^{\lambda_k}z^{v_k})d\rho_{j,t}(z)+\tau(u,v)\rho_{j,t}(z)d(\sum\limits_{k\neq j}t^{\lambda_k}z^{v_k})\parallel\\
    &\text{at most has }O(t^{\min\limits_{k\neq j}\{c_k\}})-\text{order}.
\end{align*}
So $Y_\mu$ are symplectic submanifolds. We also have that $Y_\mu$ are independent of $\mu$ away from a compact subset since $\tau(u,v)\rho_{j,t}(z)=0$ on $\mathcal{R}\times\Log_t^{-1}(U_j)\backslash\mathcal{R}'\times\Log_t^{-1}(U_j’)$. Thus $Y_\mu$ satisfy the conditions of Lemma \ref{lemma_appxlemma1} and we get the symplectomorphism $\Phi_j$.
\end{proof}

\bibliographystyle{alphaurl}
\bibliography{The_Gamma_Conj_for_Tropical_Curves}

\end{document}